\documentclass[10pt]{amsart}
\usepackage{amsmath,amssymb}
\numberwithin{equation}{section}
 \usepackage{xcolor}
 \usepackage{tikz}
\usepackage{empheq}
\usepackage{tikz-cd}

\newtheorem{theorem}{Theorem}[section]

\newtheorem{lemma}[theorem]{Lemma}
\newtheorem{lem}[theorem]{Lemma}
\newtheorem{prop}[theorem]{Proposition}

\newtheorem{corollary}[theorem]{Corollary}

\theoremstyle{definition}

\newtheorem{defi}[theorem]{Definition}
\newtheorem{example}[theorem]{Example}

\newtheorem{remark}[theorem]{Remark}

\makeatletter
\newtheorem*{rep@theo}{\rep@title}
\newcommand{\newreptheo}[2]{%
\newenvironment{rep#1}[1]{%
 \def\rep@title{#2 \ref{##1}}%
 \begin{rep@theo}}%
 {\end{rep@theo}}}
\makeatother

\newtheorem{theo}{Theorem}
\newreptheo{theo}{Theorem}

\newtheorem*{conjecture}{Conjecture}
\newreptheo{conjecture}{Conjecture}

\newtheorem*{corol}{Corollary}
\newreptheo{corol}{Corollary}

\newtheorem*{Defin}{Definition}
\newreptheo{Defin}{Definition}

 \theoremstyle{plain}
\newtheorem*{namedthm}{\namedthmname}
\newcounter{namedthm}

\makeatletter

 \newcommand{\CB}{\mathcal B}
 
 \newcommand{\D}{\mathbb D}
 \newcommand{\R}{\mathbb R}
 \newcommand{\Q}{\mathbb Q}
 \newcommand{\C}{\mathbb C}
  \newcommand{\PP}{\mathbb P}
 \newcommand{\N}{\mathbb N}
 \newcommand{\Z}{\mathbb Z}
 \newcommand{\Ox}{\mathcal O}
 \newcommand{\X}{\mathcal X}
 \newcommand{\Y}{\mathcal Y}
 \newcommand{\A}{\mathbb A}
 \newcommand{\Ll}{\mathcal L}
 \newcommand{\NA}{\mathrm{NA}}
 \newcommand{\FX}{\mathfrak X}
 \newcommand{\CD}{\mathcal D}
 \newcommand{\CZ}{\mathcal Z}
 
 \newcommand{\Id}{\mathrm{Id}}

 \newcommand{\CF}{\mathcal F}
 \newcommand{\rela}{\mathrm{rel}}
 \newcommand{\bDiv}{\mathbf{Div}}
 \newcommand{\bD}{\mathbf{D}}
 \newcommand{\bB}{\mathbf{B}}
 \newcommand{\triv}{\mathrm{triv}}
 
 \newcommand{\cM}{\mathcal M}
 
 \newcommand{\lct}{\mathrm{lct}}

 \newcommand \M{\mathcal M}
 \newcommand \CH{\mathcal H}

 \newcommand \PSH {{\rm PSH}}
  
  \newcommand \divv {{\rm div}}
  \newcommand \ord {{\rm ord}}

 \newcommand \vol{{\rm Vol}}
 
\newcommand{\Ric}{{\rm Ric}}
 \newcommand \cE{\mathcal E }

 \newcommand \mK{\mathcal K}

 \newcommand \lra{\longrightarrow}

 \newcommand{\hooklongrightarrow}{\lhook\joinrel\longrightarrow}

 \usepackage{hyperref}
\hypersetup{
    unicode=false,        
    pdftoolbar=true,      
    pdfmenubar=true,       
    pdffitwindow=false,     
    pdfstartview={FitH},    
    pdftitle={Relative YTD conjecture and stability thresholds},    
    pdfauthor={Trusiani},     
    colorlinks=true,       
   linkcolor=blue,          
    citecolor=blue,        
    filecolor=black,      
    urlcolor=blue}



\begin{document}

\title[A relative YTD conjecture and stability thresholds]{A relative Yau-Tian-Donaldson conjecture\\ and stability thresholds}

\author{Antonio Trusiani}

\address{Institut de Math\'ematiques de Toulouse   \\ Universit\'e de Toulouse \\
118 route de Narbonne \\
31400 Toulouse, France\\}

\email{\href{mailto:antonio.trusiani@math.univ-toulouse.fr}{antonio.trusiani@math.univ-toulouse.fr}}
\urladdr{\href{https://sites.google.com/view/antonio-trusiani/home}{https://sites.google.com/view/antonio-trusiani/home}}
\date{\today}

\begin{abstract}
Generalizing Fujita-Odaka invariant, we define a function $\tilde{\delta}$ on a set of \emph{generalized} $b$-divisors over a smooth Fano variety. This allows us to provide a new characterization of uniform $K$-stability.
A key role is played by a new Riemann-Zariski formalism for $K$-stability.\newline
For any generalized $b$-divisor $\bD$, we introduce a (uniform) $\bD$\emph{-log $K$-stability notion}.
We prove that the existence of a unique K\"ahler-Einstein metric with prescribed singularities implies this new $K$-stability notion when the prescribed singularities are given by the generalized $b$-divisor $\bD$.\newline
We connect the existence of a unique K\"ahler-Einstein metric with prescribed singularities to a uniform \emph{$\bD$-log Ding-stability notion} which we introduce.
We show that these conditions are satisfied exactly when $\tilde{\delta}(\bD)>1$, extending to the $\bD$-log setting the $\delta$-valuative criterion of Fujita-Odaka and Blum-Jonsson.\newline
Finally we prove the \emph{strong} openness of the uniform $\bD$-log Ding stability as a consequence of the strong continuity of $\tilde{\delta}$.
\end{abstract}

 \maketitle


\section*{Introduction}
The Yau-Tian-Donaldson conjecture for smooth Fano varieties has been solved in a series of papers \cite{CDS15I, CDS15II, CDS15III}. It produces a deep link between Differential and Algebraic Geometry, connecting the existence of K\"ahler-Einstein metrics to the algebro-geometric notion called $K$-(poly)stability. There have been several developments in the last decades around the study of $K$-stability (see the survey \cite{Xu20} and reference therein).

The Yau-Tian-Donaldson conjecture has then been extended to singular varieties and to the \emph{log setting} \cite{LTW22, Li22, LXZ22}, i.e. to log Fano pairs $(Y,\Delta)$: $\Delta$ effective $\Q$-divisor such that $(Y,\Delta)$ is klt and $-(K_Y+\Delta)$ is ample. In this case $(Y,\Delta)$ is \emph{log K-stable} if and only if there exists a \emph{log} K\"ahler-Einstein metric.

When a pair $(Y,\Delta)$ \emph{dominates} a fixed Fano manifold $X$ through a birational morphism $Y\to X$, a log K\"ahler-Einstein for $(Y,\Delta)$ can be recovered from a \emph{K\"ahler-Einstein metric with prescribed singularities} on $X$ \cite{Tru20c}. These metrics are given as weak solutions to the Einstein equation
\begin{equation}
    \label{eqn:KEIntro}
    \begin{cases}
        \Ric\big(\omega+\frac{i}{2\pi}\partial \bar{\partial}u\big)=\omega+\frac{i}{2\pi}\partial \bar{\partial}u\\
        u\in \PSH(X,\omega), \,\lvert u-\psi \rvert\leq C
    \end{cases}
\end{equation}
where $\omega$ is a fixed K\"ahler form in $c_1(X)$ and where $\psi\in \PSH(X,\omega)$ ($\omega$-plurisubharmonic function) represents the \emph{prescribed singularities}. The Ricci form is defined in a weak sense and (\ref{eqn:KEIntro}) boils down to solving a complex Monge-Ampère equation (see Definition \ref{Defi:KE}). Let us recall that if $h$ is a smooth metric on $-K_X$ associated to $\omega$ then $\omega+\frac{i}{2\pi}\partial\bar{\partial}u$ is the curvature of the \emph{singular} positive metric $he^{-2u}$ \cite{Dem90}.\newline

The upshot is that, in line with the MMP's philosophy, it is possible to study all log K\"ahler-Einstein metrics of pairs $(Y,\Delta)$ dominating $X$ directly on $X$ through the K\"ahler-Einstein metrics with prescribed singularities.

Let us stress that there are many prescribed singularities that are not recovered from dominating log Fano pairs.\newline

This article has the following two main goals.
\begin{itemize}
    \item[i)] Define algebro-geometric notions to characterize the existence of K\"ahler-Einstein metrics with prescribed singularities (a relative Yau-Tian-Donaldson correspondence).
    \item[ii)] Exploit valuative criteria and analyze how these new algebro-geometric notions are related to the usual $K$-stability. 
\end{itemize}

An analytic analogue of the second goal has been addressed in \cite{Tru20c} through the study of an extended Tian's $\alpha$-invariant \cite{Tian87}.

\subsection*{The $\delta$-invariant functions.\newline}

To determine if a $\Q$-Fano variety is $K$-stable an efficient criterion is provided by the $\delta$-invariant $\delta(X)$  \cite{FO18, BJ20}. Namely, $\delta(X)>1$ if and only if $(X,-K_X)$ is \emph{uniformly} $K$-stable. The latter is a strengthening of the $K$-stability notion which is equivalent to $K$-stability when the automorphism group of $X$ is discrete \cite{BHJ17, Der16, LXZ22}.\newline

We extend $\delta(X)$ in two different ways, obtaining functions $\delta, \tilde{\delta}: \bDiv_{-K_X}(X)\to \R$ where $\bDiv_{-K_X}(X)$ is a particular set of \emph{generalized $b$-divisors} (see Definition \ref{defi:Lpositive}). Here  $\bD=\{D_Y\}_{Y\geq X}$ is a generalized $b$-divisor if it is $b$-divisor in the sense of \cite{Sho03} where infinite countable sum is allowed.
We use the notation $\delta_\bD:=\delta(\bD)$, $\tilde{\delta}_\bD:=\tilde{\delta}(\bD)$.\newline
The definitions of $\delta,\tilde{\delta}:\bDiv_{-K_X}(X)\to \R$ are given using a Riemann-Zariski formalism which we do not state here (see subsection \ref{ssec:deltas}). 
However we can describe the values of $\delta,\tilde{\delta} $ at $\Q$-divisors.

Set $L:=-K_X$ and let $D_Y\subset Y$ be a $\Q$-divisor for $Y\overset{\rho}{\geq} X$, i.e. $Y$ dominates $X$ through a projective birational morphism $\rho:Y\to X$.
If $L_Y:=\rho^*L-D_Y$ is ample and $(Y,B_Y:=D_Y-K_{Y/X})$ is klt, then $D_Y$ naturally defines a $b$-divisors $\bD\in \bDiv_L(X)$ by pulling back and $(Y,B_Y)$ is a log Fano pair. In this case
\begin{equation}
    \label{eqn:DeltaIntro}
    \delta_{\bD}=\delta_{B_Y}(Y,L_Y):=\inf_{E/Y}\frac{A_{(Y,B_Y)}(E)}{S_{L_Y}(E)}
\end{equation}
(Remark \ref{rem:Deltas}). We denoted by $\delta_{B_Y}(Y,L_Y)$ the log version of the \emph{$\delta$-invariant}.
The infimum considers all prime divisors $E$ \emph{over} $Y$, i.e. on all $Z\geq Y$. The log discrepancy of a fixed prime divisor $E$ on $Z\overset{\rho}{\geq}X$ is defined as $A_{(Y,B_Y)}(E):=1+\ord_E(K_{Z/(Y,B_Y)})$, while $S_{L_Y}(E):=\frac{1}{(L_Y^n)}\int_{0}^{+\infty}\mathrm{Vol}_{Z}(\rho^*L_Y-tE)dt$ is the \emph{expected vanishing order of $L_Y$} along $E$. We refer to \cite{FO18, BJ20} for a discussion on how $\delta_{B_Y}(Y,L_Y)$ arises as infimum of log canonical thresholds wrt to $(Y,B_Y)$ of certain $\Q$-divisors $F\sim_\Q L_Y$. \newline
Let us stress that $B_Y$ is \emph{not} necessarily effective, although the negative part is exceptional with respect to $Y\geq X$.

By what said before the log Fano pair $(Y,B_Y)$ defines a prescribed singularities on $X$. Thus one of the main idea of the paper is to modify $\delta_\bD$ so that it measures the singularities of all prime divisors over $X$ with respect to $(X,\psi)$ where $\psi$ is the prescribed singularities describing $\bD$. Algebraically we have
$$
\tilde{\delta}_\bD:=\inf_{E/Y} \frac{A_{(Y,B_Y)}(E)+\ord_E(D_Y)}{S_{L_Y}(E)+\ord_E(D_Y)}
$$
(Remark \ref{rem:Deltas}). Indeed, $A_{(Y,B_Y)}(E)+\ord_E(D_Y)=A_X(E)$ is the log discrepancy of $E$ with respect to $X$, while $S_{L_Y}(E)+\ord_E(D_Y)$ can be seen as an expected vanishing order of $L$ along $E$ where one only considers sections that are \emph{compatible} with the singularities $\psi$.\newline

The prescribed singularities $\psi$ given by $\Q$-divisors $D_Y$ can be described by certain particular $\omega$-plurisubharmonic functions called \emph{algebraic model type envelopes}. Denote by $\cM_D(X,\omega)$ the set of all model type envelopes that can be approximated by algebraic model type envelopes \cite{Tru20c}. Algebraically $\cM_D(X,\omega)$ describes the \emph{closure} of all ample test configurations for $(X,-K_X)$ \cite{DX20}.

We show that $\bDiv_{-K_X}(X)$ is in one-to-one correspondence with
$$
\cM^+_D(X,\omega):=\big\{\psi\in \cM_D(X,\omega)\, :\, \int_X MA_\omega(\psi)>0\big\},
$$
where $MA_\omega$ is the non-pluripolar Monge-Ampère operator (Proposition \ref{prop:Corre}). This analytic description of $\bDiv_{-K_X}(X)$ is advantageous because it allows us to work directly on $(X,\omega)$. \newline

Studying $\tilde{\delta}:\bDiv_{-K_X}(X)\to \R$, we obtain the following new $K$-stability criterion.
\begin{theo}
\label{thmA}
Let $X$ be a smooth Fano variety, $L=-K_X$. Then $(X,L)$ is uniformly $K$-stable if and only if
\begin{equation}
    \label{eqn:Toulouse1}
    \sup_{\bD\in\bDiv_L(X)} \tilde{\delta}_{\bD}^{\frac{1}{n}}\langle (L-\bD)^n \rangle> (L^n)
\end{equation}
\end{theo}
The quantity $\langle (L-\bD)^n\rangle$ is the \emph{volume} of $L-\bD$ in a Riemann-Zariski sense as in \cite{BFJ09}. If $\bD=D_Y$ for a $\Q$-divisor $D_Y$ on $Y\geq X$ then $\langle (L-\bD)^n\rangle=(L_Y^n)$.

If $(X,-K_X)$ is uniformly $K$-stable then $\tilde{\delta}_0=\delta(X)>1$ by \cite{FO18, BJ20}, i.e. one implication of Theorem \ref{thmA} is already known. The novelty of Theorem \ref{thmA} regards instead the reverse arrow. It says that estimating $\tilde{\delta}$ at no trivial elements $\bD\in \bDiv_{-K_X}(X)$ provides information on uniform $K$-stability.
Analytically (\ref{eqn:Toulouse1}) can be rephrased as
$$
\sup_{\psi\in\cM^+_D(X,\omega)}\tilde{\delta}_{\bD_\psi}^{\frac{1}{n}}\int_X MA_\omega(\psi)> \int_X \omega^n
$$
where $\bD_\psi$ is the element in $\bDiv_L(X)$ associated to $\psi\in\cM^+_D(X,\omega)$.

We expect that Theorem \ref{thmA} will help to classify the $K$-stability of Fano manifolds. \newline

As $\cM^+_D(X,\omega)$ has a natural $d_S$-distance \cite{DDNL19}, the set $\bDiv_{-K_X}(X)$ inherits a \emph{strong topology}.
\begin{theo}
\label{thmB}
Let $X$ be a smooth Fano variety, $L=-K_X$. Then $\tilde{\delta}:\bDiv_L(X)\to \R$ is strongly continuous and
\begin{equation}
    \label{eqn:LLL2}
    \delta(X)\leq (L^n)\inf_{\bD\in\bDiv_L(X)}\frac{\tilde{\delta}_\bD}{\langle (L-\bD)^n\rangle}\leq  (L^n)\inf_{\bD\in\bDiv_L(X)} \frac{\lct_X(\bD)}{\langle(L-\bD)^n\rangle}
\end{equation}
where $\lct_X(\bD)=\inf_{Y\geq X} \lct_{(Y,B_Y)}(D_Y)+1$, for $B_Y:=D_Y-K_{Y/X}$;
\end{theo}
The \emph{log-canonical threshold} in Theorem \ref{thmB} coincides with the complex singularity exponent $\lct_X(\psi)$ of the associated model type envelope $\psi$, i.e.
$$
\lct_X(\bD)=\inf_{E/X} \frac{A_X(E)}{\ord_E(\bD)}=\inf_{E/X}\frac{A_X(E)}{\nu(\psi,E)}=\sup\{c>0\, : \, \mathcal{I}(c\psi)=\mathcal{O}_X\}=:\lct_X(\psi).
$$
While in Theorem \ref{thmA} the values of $\tilde{\delta}$ at non-trivial elements give sufficient conditions to the uniform $K$-stability, in Theorem \ref{thmB}, $\bD\to \tilde{\delta}_\bD$ produces necessary conditions, i.e. it provides obstructions to uniform $K$-stability.
Analytically (\ref{eqn:LLL2}) becomes
$$
\delta(X)\leq \int_X \omega^n \inf_{\psi\in\cM^+_D(X,\omega)} \frac{\tilde{\delta}_{\bD_\psi}}{\int_X MA_\omega(\psi)}\leq \int_X \omega^n \inf_{\psi\in\cM^+_D(X,\omega)} \frac{\lct_X(\psi)}{\int_X MA_\omega(\psi)}.
$$

We invite the reader to consult Section \ref{sec:Last} for further and more general results regarding $\tilde{\delta}:\bDiv_L(X)\to \R$.

\subsection*{A relative Yau-Tian-Donaldson correspondence.\newline}
Let $L\to X$ be an ample line bundle over a normal projective variety $X$.
For any $\bD\in\bDiv_L(X)$ we define a natural $\bD$\emph{-log} $K$\emph{-stability notion}.\newline
We first interpret the usual $K$-stability notion in a \emph{Riemann-Zariski perspective}. Intuitively, we consider all the log $K$-stability notions for pairs $(Y,B_Y)$ dominating $X$ as particular algebro-geometric notions over $X$. We refer to Section \ref{sec:K-stability} for this new Riemann-Zariski formalism of $K$-stability.\newline

As a result, we extend the intersection formula of the \emph{Donaldson-Futaki invariant} associated to ample and normal test configurations $(\X,\Ll)$ \cite{Wang12, Oda13} to our setting in a natural way. Namely, we define a $\bD$-log Donaldson-Futaki invariant
$$
(\X,\Ll)\to DF(\X,\Ll;\bD)
$$
which is translation and pull-back invariant (Corollary \ref{cor:DF}).
\begin{Defin}
Let $(X,L)$ be a polarized normal projective variety and let $\bD\in\bDiv_L(X)$. Then $(X,L)$ is said to be
\begin{itemize}
    \item[i)] $\bD$\emph{-log} $K$\emph{-stable} if $DF(\X,\Ll;\bD)\geq 0$ for any ample and normal test configuration $(\X,\Ll)$ with equality if and only if $(\X,\Ll)$ is $\bD$-trivial;
    \item[ii)] \emph{uniformly $\bD$-log $K$-stable} if there exists $A>0$ such that $DF(\X,\Ll;\bD)\geq A J^\NA(\X,\Ll;\bD)$ for any ample and normal test configuration $(\X,\Ll)$.
\end{itemize}
\end{Defin}
The non-negative functional $J^\NA(-;\bD)$ measures the $\bD$-triviality of normal test configurations (Definition \ref{defi:JandE}). It reminds and extends the \emph{Non-Archimedean} $J$\emph{-energy} \cite{BHJ17}. Similarly, the $\bD$-triviality generalizes the triviality given by $\X\simeq X\times \C$ and it can be read analytically in terms of \emph{psh test curves} (Proposition \ref{prop:Dtriviality}).
Clearly uniform $\bD$-log $K$-stability implies $\bD$-log $K$-stability.

When $\bD=D_Y$ for an effective $\Q$-divisor $D_Y$ on $Y\overset{\rho_Y}{\geq} X$ such that $B_Y:=D_Y-K_{Y/X}$ is effective and $(Y,B_Y)$ is klt, the $\bD$-log $K$-stability of $(X,L)$ for $\bD$ corresponds to the usual log $K$-stability of $(Y,L_Y:=\rho_Y^*L-D_Y)$ with respect to the pair $(Y,B_Y)$. More generally, if $D_Y$ is a $\Q$-divisor for any $Y\geq X$, then the $\bD$-log $K$-stability of $(X,L)$ can be interpreted as an \emph{asymptotic log $K$-stability notion of $(Y,L_Y)$ wrt $(Y,B_Y)$} (Proposition \ref{prop:KstClass}).

Similar observations holds for the uniform $\bD$-log $K$-stability notion.\newline

In the Fano case, we propose the following $\bD$-log Yau-Tian-Donaldson conjecture.
\begin{conjecture}
Let $X$ be a smooth Fano variety, $L=-K_X$. Fix $\bD\in\bDiv_L(X)$ and let $\psi\in \cM^+_D(X,\omega)$ be the associated model type envelope. Then the followings are equivalent:
\begin{itemize}
    \item[i)] $(X,-K_X)$ admits a $[\psi]$-KE metric;
    \item[ii)] $(X,-K_X)$ is $\bD$-log $K$-stable.
\end{itemize}
\end{conjecture}
By a $[\psi]$-KE metric, we mean a K\"ahler-Einstein metric with prescribed singularities encoded in $\psi$.\newline

We partially solve this conjecture.
\begin{theo}
\label{theoC}
\label{thmC}
Let $X$ be a smooth Fano variety, $L=-K_X$. Fix $\bD\in\bDiv_L(X)$ and let $\psi\in\cM^+_D(X,\omega)$ be the associated model type envelope. If there exists a unique $[\psi]$-KE metric then $(X,-K_X)$ is uniformly $\bD$-log $K$-stable.
\end{theo}
The uniqueness of $[\psi]$-KE metric is equivalent to ask that the group $\mathrm{Aut}(X,[\psi])=\big\{F\in \mathrm{Aut}(X)\, : \, \lvert \psi-\psi\circ F\rvert\leq C\big\}$ is discrete \cite{Tru20c}, which is a condition generically satisfied.\newline

To prove Theorem \ref{thmC}, we use a variational approach similarly to \cite{BBJ15} passing through a $\bD$\emph{-log Ding stability notion}. Indeed, combining the Riemann-Zariski framework with pluripotential theory, we define a $\bD$\emph{-log Ding functional}
$$
(\X,\Ll)\to D^\NA(\X,\Ll;\bD)
$$
which is translation and pull-back invariant. It generalizes the \emph{Non-Archimedean Ding functional} \cite{Berm16, BHJ17}.
\begin{Defin}
Let $(X,L)$ be a polarized smooth variety and let $\bD\in\bDiv_L(X)$. Then $(X,L)$ is said to be
\begin{itemize}
    \item[i)] $\bD$\emph{-log Ding stable} if $D^\NA(\X,\Ll;\bD)\geq 0$ for any $(\X,\Ll)$ ample and normal test configuration with equality if and only if $(\X,\Ll)$ is $\bD$-trivial;
    \item[ii)]\emph{uniformly} $\bD$\emph{-log Ding stable} if there exists $A>0$ such that $D^\NA(\X,\Ll;\bD)\geq A J^\NA(\X,\Ll;\bD)$ for any ample and normal test configuration $(\X,\Ll)$.
\end{itemize}
\end{Defin}
The $\bD$-log Ding stability can be viewed as an \emph{asymptotic log Ding stability} in the case when $D_Y$ is a $\Q$-divisor for any $Y\geq X$ (Proposition \ref{prop:DStCar}).\newline

Assuming $L=-K_X$, we prove that the $\bD$-log Ding functional dominates the $\bD$-log Donaldson-Futaki invariant (Proposition \ref{prop:DingK}). Thus Theorem \ref{thmC} follows from our next result.
\begin{theo}
\label{thmD}
Let $X$ be a smooth Fano variety, $L=-K_X$. Fix $\bD\in\bDiv_L(X)$ with $\lct_X(\bD)>1$ and let $\psi\in\cM^+_D(X,\omega)$ be the associated model type envelope. The followings are equivalent:
\begin{itemize}
    \item[i)] there exists a unique $[\psi]$-KE metric;
    \item[ii)] $(X,-K_X)$ is uniformly $\bD$-log Ding-stable;
    \item[iii)] $\delta_\bD>1$;
    \item[iv)] $\tilde{\delta}_\bD>1$.
\end{itemize}
\end{theo}
Note that the equivalence $(i)\Longleftrightarrow (ii)$ represents an algebro-geometric characterization of the existence of K\"ahler-Einstein metrics with prescribed singularities, i.e. a first instance of a Yau-Tian-Donaldson correspondence.

The last two points instead give valuative criteria, generalizing \cite{FO18, BJ20} and further justifying the definitions of $\delta$ and $\tilde{\delta}$. In particular, Theorems \ref{thmA}, \ref{thmB} and \ref{thmD} connect the uniform $\bD$-log Ding-stability to the usual uniform $K$-stability.

From the analytic point of view, they relate the existence of $[\psi]$-KE metrics with that of the genuine K\"ahler-Einstein metrics (cf. \cite{Tru20c}).\newline

Moreover we also have the following easy consequence of Theorems \ref{thmB}, \ref{thmD}.
\begin{corol}
\label{corA}
Let $X$ be a smooth Fano variety, $L=-K_X$. Then
$$
\bDiv^{UD}(X):=\big\{\bD\in \bDiv_L(X)\, :\, \lct_X(\bD)>1 \,\mathrm{and}\, (X,L)\, \mathrm{uniformly}\, \bD\mathrm{-log}\, K\mathrm{-stable}\big\}
$$
is strongly open.
\end{corol}

\subsection*{Strategy of the proofs.\newline}
Let us start explaining the main ideas in the proof of Theorem \ref{thmD} as it is dense of techniques and new ideas.

The proof $(i)\Rightarrow (ii)$ is based on two main ingredients: an analytic characterization of the existence of $[\psi]$-KE metrics and a generalization of \emph{Deligne pairings}.
More precisely, in \cite{Tru20c} the author proved that the existence of a unique $[\psi]$-KE metric is equivalent to the coercivity of a functional, the $[\psi]$\emph{-relative Ding functional}. 
This functional is defined on a complete geodesic metric space $\big(\cE^1_\psi,d_1\big)$ containing all $\varphi\in \PSH(X,\omega)$ such that $\lvert\varphi-\psi \rvert\leq C$.
In this paper, we first observe that every normal and ample test configuration defines a \emph{geodesic ray} in $\big(\cE^1_\psi,d_1\big)$ (Proposition \ref{prop:Current}, Corollary \ref{cor:Abo}).
Then we define positive Deligne pairings in the prescribed singularities setting to connect the value $D^\NA (\X,\Ll)$ to the slope at infinity of the $[\psi]$-relative Ding functional along the associate geodesic ray (subsection \ref{ssec:JandE}).
Although the same idea of the proof has been used in \cite{BBJ15} for the absolute setting $\psi=0$, in the prescribed singularities setting there are many new difficulties and technicalities.
Moreover, let us stress that our algebraic-geometric stability notions are defined in a Riemann-Zariski framework.
Thus a careful pluripotential-theoretical analysis is developed to define prescribed singularities that can be read in this abstract context (Section \ref{sec:Ding-StabilitySection}).

To prove $(ii)\Rightarrow (iii)$ we proceed in the following way. Through a generalization of Ross-Witt Nystr\"om to the prescribed singularities setting of the correspondence between \emph{psh test curves} and geodesic rays, we naturally associate to any $E$ prime divisor over $X$ a geodesic ray $\{u_t^{\psi,E}\}_{t\geq 0}$ in $\big(\cE^1,d_1\big)$. Then we relate the slope at infinity of the $[\psi]$-relative Ding functional along $\{u_t^{\psi,E}\}_{t\geq 0}$ to the quantity $A_X(E)-\ord_E(\bD)-S_\bD(E)$, where, generalizing formula (\ref{eqn:DeltaIntro}),
$$
\delta_\bD:=\inf_{E/X}\frac{A_X(E)-\ord_E(\bD)}{S_\bD(E)}
$$
(Definition \ref{defi:delta}). Thus the strategy consists in constructing a sequence of ample and normal test configurations $(\X_k,\Ll_k)$ such that
\begin{itemize}
    \item[i)] the sequence $D^\NA(\X_k,\Ll_k;\bD)$ converges to the slope of the $[\psi]$-relative Ding functional along $\{u_t^{\psi,E}\}_{t\geq 0}$;
    \item[ii)] the sequence $J^\NA(\X_k,\Ll_k;\bD)$ converges to a quantity that dominates a positive multiple of $S_\bD(E)$.
\end{itemize}
We achieve it in Propositions \ref{prop:Esmooth}, \ref{prop:Appoo} combining the Demailly's regularization argument based on multiplier ideal sheaves exploited in \cite{BBJ15} with a \emph{projection argument for different prescribed singularities}. More precisely, we first prove the result in the absolute setting $\psi=0$, recovering \cite{BBJ15} (Propsition \ref{prop:Esmooth}). Then we extend a projection operator, which goes back to \cite{RWN14}, $P_\omega[\psi](\cdot):\big(\cE^1_0,d_1\big)\to \big(\cE^1_\psi,d_1\big)$ to geodesic rays (see Proposition \ref{prop:Proj}). Thus we accomplish our strategy thanks to good properties of such projection operator. Indeed the inspiring thought is that one can recover all the quantities involved, such as $D^\NA(\X_k,\Ll_k;\bD)$ and $J^\NA(\X_k,\Ll_k;\bD)$, from the absolute setting through this pluripotential-theoretical techniques. Let us stress out that in this part of the proof the advantages of working on $(X,\omega)$ varying the prescribed singularities instead than varying the log Fano pairs above $X$ are evident.

Next, we prove the equivalence $(iii)\Leftrightarrow (iv)$ in Proposition \ref{prop:deltacompa}. This is inspired by an analogous result in \cite{Tru20c} regarding generalizations of the Tian's $\alpha$-invariant.

Finally, the proof of $(iv)\Rightarrow (i)$ is influenced by the recent paper \cite{DZ22}. We first prove that $\tilde{\delta}_\bD>1$ implies a \emph{uniform geodesic stability} of the $[\psi]$-relative Ding functional (Theorem \ref{thm:GeodStab}). Namely the slope at infinity of the $[\psi]$-relative Ding functional along normalized geodesic rays is bigger than a uniform positive constant. To accomplish this, we exploit the properties of psh test curves and a formula for the slope of the $[\psi]$-relative Monge-Ampère energy along geodesic rays (Propositions \ref{prop:EnergyFormula}, \ref{prop:Necessaria}). Then to show the coercivity of the $[\psi]$-relative Ding functional, we proceed by contradiction combining a compactness argument in $\big(\cE^1_\psi,d_1\big)$ with the convexity properties of the $[\psi]$-relative Ding functional (Theorem \ref{thm:ConclusionKE}). Finally the existence of a $[\psi]$-KE metric follows by the aforementioned variational approach developed in \cite{Tru20c}.\newline

Theorem \ref{thmC} is a consequence of Theorem \ref{thmD} and of the inequality
\begin{equation}
    \label{eqn:DDF}
     DF(\X,\Ll;\bD)\geq D^\NA(\X,\Ll;\bD),
\end{equation}
which is proved in subsection \ref{ssec:DingToK}. We show (\ref{eqn:DDF}) using a Non-Archimedean point of view as in \cite{BHJ17} when $\bD$ is given by pulling back a $\Q$-divisor $D_Y$ living on $Y\geq X$ (Proposition \ref{prop:DingK}). For general $\bD:=\{D_Y\}_{Y\geq X}$, the inequality (\ref{eqn:DDF}) is reached by an approximation argument. Indeed we recover $D^\NA(\X,\Ll;\bD)$ as liminf of $D^\NA(\X,\Ll;D_Y)$ and similarly for the $\bD$-log Donaldson-Futaki invariant (Propositions \ref{prop:KstClass}, \ref{prop:DStCar}). In particular we interpret the $\bD$-log $K$-stability as an asymptotic log $K$-stability and similarly for the $\bD$-log Ding stability.\newline

Theorems \ref{thmA}, \ref{thmB} follows by an analysis of $\tilde{\delta}_\bD$ along increasing paths $[0,1]\ni t\to \bD_t\in \bDiv_L(X)$. A key point is to adopt the analytic point of view given by the associated decreasing path $[0,1]\ni t\to \psi_t\in \mathcal{M}_D^+(X,\omega)$. Let us stress out again that working directly on $(X,\omega)$ is a main advantage. Indeed it allows us to use many properties of the set $\mathcal{M}_D^+(X,\omega)$ and of $\PSH(X,\omega)$ in general.

\subsection*{About the assumptions.\newline}
All the main Theorems stated in the Introduction require $X$ to be \emph{smooth}. The main reason is to make the paper more readable. Indeed, we expect that the proofs and the techniques developed in this paper will be extended to general $\Q$-Fano varieties in future works.\newline
Let us stress that the definition of $\bD$-log $K$-stability is given in the full generality of a polarized normal variety (see Section \ref{sec:K-stability}).

In the absolute setting it is well-known that a klt assumption is necessary to have Ding stability. Similarly, in \cite{Tru20c} the author observed that $\lct_X(\psi)>1$ is necessary to have a $[\psi]$-KE metric. However, a priori uniform $\bD$-log Ding-stable only implies $\lct_X(\bD)\geq 1$, and this causes the extra hypothesis "$\lct_X(\bD)>1$" in Theorem \ref{thmD} and in the subsequent corollary. More precisely, such hypothesis is exclusively used in the implication $(iii)\Rightarrow (iv)$ of Theorem \ref{thmD} (see Proposition \ref{prop:deltacompa}).

\subsection*{Further generalizations.\newline}
It is possible to extend the result of this paper to the log setting. Namely, replacing $X$ by a log Fano pair $(X,\Delta)$, one can convert all the proofs/definitions to such setting getting for instance an algebro-geometric characterization of the existence of $\Delta$-log $[\psi]$-KE metrics (see \cite{Tru20b}).

When $\mathrm{Aut}(X,[\psi])$ is not discrete one can replace $\delta,\tilde{\delta}, J^\NA(\cdot;\bD)$ and the uniform Ding and $K$-stability notions by their $\mathbb{G}$-equivariant versions similarly to \cite{His19, Li22}. It is then natural to wonder if \cite[Theorem D]{Tru20c} and the Theorems of this paper adapt to the $\mathbb{G}$-equivariant setting.

\subsection*{Recent related works.\newline}
Dervan and Reboulet recently in \cite{DR22}, assuming $-K_X$ being big and klt, proved that the existence of a unique (weak) K\"ahler-Einstein metric implies a uniform Ding stability notion which is similar to the one studied in our setting (see Remark \ref{rem:Llct}). We adapted to the prescribed singularities their use of Deligne pairings (Proposition \ref{prop:Slope}).

Darvas and Zhang in \cite{DZ22} developed a $\delta$-invariant theory when $-K_X$ is big, proving in particular that $\delta(X)>1$ implies the existence of a (weak) K\"ahler-Einstein metric. As already mentioned, some of their ideas are used in Section \ref{sec:YTD} for the implication $(iv)\Rightarrow (i)$ in Theorem \ref{thmD}.

Supposing $-K_X$ big, Xu in \cite{Xu22} proved that the $K$-stability condition (given in terms of the $\delta$-invariant) of $(X,-K_X)$ forced to have a klt anticanonical model, whose stability properties are essentially the same as that of $(X,-K_X)$. Moreover, in \cite{Xu22} it is also showed that the uniform Ding stability introduced in \cite{DR22} implies $\delta(X)>1$. This concludes the equivalence between the existence of weak K\"ahler-Einstein, the uniform Ding stability and $\delta(X)>1$ thanks to \cite{DZ22, DR22}.\newline

It would be interesting to see if the proof of the implication $(ii)\Rightarrow (iv)$ of Theorem \ref{thmD} can be adapted to the big setting. Together with \cite{DR22} and \cite{DZ22} this would give an uniform Ding version of the Yau-Tian-Donaldson conjecture in the case $-K_X$ is big, without passing to the klt anticanonical model as in \cite{Xu22}.
\subsection*{Structure of the paper.\newline}
After recalling some preliminaries (Section \ref{sec:Preliminaries}), Section \ref{sec:K-stability} is dedicated to introduce the $\bD$-log $K$-stability notion. Here, in particular we adopt the Riemann-Zariski perspective which will be central in all the paper.\newline
In Section \ref{sec:Ding-StabilitySection} we prove the one-to-one correspondence between $\bDiv_L(X)$ and $\cM^+_D(X,\omega)$ and we define the $\bD$-log Ding-stability notion through the pluripotential approach of \emph{plurisubharmonic rays} and of \emph{singularity types}. We also prove that the uniform $\bD$-log Ding stability implies the uniform $\bD$-log $K$-stability. \newline
Section \ref{sec:YTD} is then the core of the paper, where we prove Theorems \ref{thmC}, \ref{thmD} and we introduce the $\delta,\tilde{\delta}$-functions. We also interpret $D^\NA(\cdot;\bD)$ in a more classical way in terms of log canonical thresholds as in \cite{Berm16}.\newline
Finally in Section \ref{sec:Last} we study the properties of the $\tilde{\delta}$-functions proving Theorems \ref{thmA}, \ref{thmB} and the strong openness of $\bDiv^{UD}(X)$.

\subsection*{Acknowledgements.\newline}
The author is supported by a postdoctoral grant from the Knut and Alice Wallenberg foundation. The author thanks Dervan and Reboulet for interesting discussions, comments and for sharing their work. He wants to thank Reboulet for also suggesting to prove the openness of $\bDiv^{UD}(X)$. In addition, the author thanks Vincent Guedj and Chung-Ming Pan for helpful remarks, and Sébastien Boucksom to have pointed him a mistake in the previous version.

\section{Preliminaries}
\label{sec:Preliminaries}
\subsection{K\"ahler-Einstein metrics with prescribed singularities}
\label{ssec:KE}
In this subsection, $X$ will be a $n$-dimensional compact K\"ahler manifold. We will use the notation $dd^c:=\frac{i}{2\pi}\partial \bar{\partial}$.

\subsubsection{Singular metrics}
Letting $\theta$ be a smooth and closed $(1,1)$-form, by $\PSH(X,\theta)$ we denote the set of all \emph{$\theta$-plurisubharmonic ($\theta$-psh) functions}, i.e. all $u:X\to \R\cup\{-\infty\}$ such that locally $u+g$ is plurisubharmonic for any $g$ local potential of $\theta$. Since $\PSH(X,\theta)\subset L^1(X)$, the set of $\theta$-psh functions inherits a \emph{weak topology} given by the $L^1$-convergence.

If $\theta\in c_1(L)$ for a pseudoeffective line bundle $L\to X$, then the elements in $\PSH(X,\theta)$ can be thought as \emph{singular (hermitian) metrics} on $L\to X$ in the sense of \cite{Dem90}. Indeed, fixing $h$ smooth metric on $L\to X$ with curvature $\theta$, the one-to-one correspondence between singular metrics on $L\to X$ and $\PSH(X,\theta)$ is given by
$$
\PSH(X,\theta)\ni u\to he^{-2u}.
$$
The curvature of the singular metric $he^{-2u}$ is then the closed and positive $(1,1)$-current $\theta_u:=\theta+dd^c u$. As an important example, if $D\in \lvert mL\rvert$ is an effective divisor defined by a holomorphic section $s_D\in H^0(X,mL)$, then $u_D:=\frac{1}{m}\log\lvert s_D\rvert_{h^m}\in \PSH(X,\theta)$ and
$$
\theta+dd^c u_D=[D],
$$
\emph{current of integration along $D$}.
\subsubsection{Non-pluripolar product}
In \cite{GZ07, BEGZ10} the authors generalized the pioneering works of Bedford-Taylor \cite{BT87} on defining a \emph{non-pluripolar product} among currents. \newline
We invite the reader to consult \cite{BEGZ10} (see also \cite{GZ17}) for the precise definition of this construction. For the purpose of the paper, we just need to recall the following properties.

Denoting by $\mathcal{T}$ the set of all closed and positive $(1,1)$-currents on $X$, the map
$$
\mathcal{T}^n\ni (T_1,\dots,T_n)\longrightarrow \langle T_1 \wedge \cdots \wedge T_n \rangle,
$$
given by the non-pluripolar product, is multilinear, symmetric, homogeneous of degree $1$ in each variable and it has image in the set of positive measures on $X$ which does not put mass on pluripolar sets. Moreover it coincides with the usual wedge product when the $T_j$'s  are actually smooth forms.\newline
By the $\partial \bar{\partial}$-lemma, choosing representatives $\theta_j$ of the cohomology classes $\{T_j\}$, we have $T_j=\theta_j+dd^c u_j$ for $u_j\in \PSH(X,\theta_j)$. In particular one may wonder if the total mass
$$
\int_{X}\langle (\theta_1+dd^c u_1)\wedge \cdots \wedge (\theta_n+dd^c u_n) \rangle
$$
can be computed in cohomology, i.e. if it depends only on $\{\theta_j\}$. This is not the case: considering for instance $T_1=[D]$ one obtains
$$
\langle T_1\wedge \cdots \wedge T_n\rangle=0
$$
as a consequence of the fact that $T_1$ is supported on a pluripolar set.
\subsubsection{Relative full Monge-Ampère mass}
The set $\PSH(X,\theta)$ has a natural partial order given by $u\preccurlyeq v$ if $u\leq v+C$ for a constant $C\in \R$ (and in this case we will say that \emph{$u$ is more singular than $v$}). As proved in \cite[Corollary 2.3]{BEGZ10}, \cite[Theorem 1.2]{WN17} and \cite[Theorem 1.1]{DDNL17b}
$$
\int_X \langle (\theta_1+dd^c u_1)\wedge \cdots \wedge (\theta_n+dd^c u_n) \rangle\leq \int_X \langle (\theta_1+dd^c v_1)\wedge \cdots \wedge (\theta_n+dd^c v_n) \rangle
$$
if $u_j\preccurlyeq v_j$ for any $j=1,\dots,n$, i.e. the total mass of the non-pluripolar product respects the partial order $\preccurlyeq$. In particular, considering the case $\theta_1=\cdots=\theta_n$, the total mass of the \emph{Monge-Ampère operator} $MA_\theta(u):=\langle(\theta+dd^c u)^n\rangle$ is maximal if $u$ has \emph{minimal singularities}. Consider
$$
V_\theta:=\sup\big\{u\in \PSH(X,\theta)\, :\, u\leq 0\big\}.
$$
This is a $\theta$-psh function. Then $u$ is said to have minimal singularities if $u-V_\theta$ is globally bounded, i.e. if $u \approx V_\theta$ in the sense of the partial order $\preccurlyeq$. Let us stress that the set of \emph{full Monge-Ampère mass}
$$
\cE(X,\theta):=\Bigg\{u\in \PSH(X,\theta)\, :\, \int_X MA_\theta(u)=\int_X MA_\theta(V_\theta)\Bigg\}
$$
contains $\theta$-psh functions which are more singular than $V_\theta$.\newline
Furthermore we also recall that $\int_X MA_\theta(V_\theta)>0$ if and only if the class $\{\theta\}$ is big. For instance if $\{\theta\}$ is integral, i.e.  $\{\theta\}=c_1(L)$ for a big line bundle, then $\int_X MA_\theta(V_\theta)=\vol_X(L)$.

From now on we will then assume that the class $\{\theta\}$ is big.

\subsubsection{Model type envelopes}
Changing perspective, one may wonder what are the \emph{essential} singularities of $u\in\PSH(X,\theta)$, i.e. for instance which singularities can be removed preserving the total Monge-Ampère mass. In \cite{RWN14} the authors defined the $\theta$-psh function
$$
\psi:=P_\theta[u](0)=\Big(\sup\big\{v\in \PSH(X,\theta)\, : \, v\leq 0, v\preccurlyeq u\big\}\Big)^*
$$
where the star is for the upper semicontinuous regularization. Denote by
$$
\cM(X,\theta):=\big\{\psi\in \PSH(X,\theta)\, : \, P_\theta[\psi](0)=\psi\big\},
$$
the set of all \emph{model type envelopes} (see \cite{DDNL17b, Tru19}). There are plenty of this functions as the map $P_\theta[\cdot](0):\PSH(X,\theta)\to \PSH(X,\theta)$ is a projection when restricted to $\theta$-psh function with positive total Monge-Ampère mass \cite[Theorem 3.12]{DDNL17b}. Letting $\cM^+(X,\theta)$ be the set of all $\theta$-psh model type envelopes with positive total Monge-Ampère mass, we then have that for $\psi\in \cM^+(X,\theta)$
$$
\cE(X,\theta,\psi):=\Big\{u\in \PSH(X,\theta)\, : \, u\preccurlyeq\psi, \, \int_XMA_\theta(u)=\int_X MA_\theta(\psi) \Big\}
$$
coincides with the preimage of $\psi$ with respect to the projection map $P_\theta[\cdot](0)$, i.e. $P_\theta[u](0)=\psi$ if and only if $u\in\cE(X,\theta,\psi)$ \cite[Theorem 1.3]{DDNL17b}. Elements in $\cE(X,\theta,\psi)$ are said to have \emph{$[\psi]$-relative full Monge-Ampère mass}.\newline

$\cM(X,\theta)$ satisfies the following properties:
\begin{itemize}
    \item[i)] for any $u\preccurlyeq \psi\in \cM(X,\theta)$ the inequality $u\leq \psi+\sup_X u $ holds. In particular the partial order $\preccurlyeq$ coincides with $\leq$ on $\cM(X,\theta)$;
    \item[ii)] The weak closure of any totally ordered set $\mathcal{A}\subset \cM^+(X,\theta)$ belongs to $\cM(X,\theta)$ \cite[Lemma 2.6]{Tru19}, and moreover the Monge-Ampère operator $\mathcal{A}\ni \psi\to MA_\theta (\psi)$ extend to a homeomorphism $\overline{\mathcal{A}}\to MA_\theta(\overline{\mathcal{A}})$ with respect to the natural weak topologies (since the proof in \cite[Lemma 3.12]{Tru20a} easily extends to the big case).\newline
\end{itemize}

When $\omega$ is K\"ahler, some model type envelopes in $\cM(X,\omega)$ are of an algebraic nature.

We recall that $u\in \PSH(X,\omega)$ is said to have \emph{algebraic singularities encoded in $(\mathcal{I},c)$} for $\mathcal{I}$ coherent ideal sheaf and $c>0$ if locally
$$
u=g+c\log \Big(\sum_j\lvert f_j \rvert\Big)
$$
where $\{f_j\}_j$ are local generators of $\mathcal{I}$ while $g$ is smooth. In this case $\psi:=P_\omega[u](0)$ is a model type envelope and $\psi\approx u$ (\cite[Remark 4.6]{RWN14}, \cite[Proposition 4.36]{DDNL17b}). Considering $\rho:Y\to X$ a log resolution of $\mathcal{I}$ we obtain
$$
p^*\big(\omega+dd^c u\big)= \eta+c[D]
$$
for a smooth, closed and semipositive $(1,1)$-form $\eta$, where $p^{-1}\mathcal{I}=\mathcal{O}_Y(-D)$ for an effective divisor $D$. Moreover the map $v\to \tilde{v}:=(v-u)\circ p$ induces a bijection between $\{v\in \PSH(X,\omega)\, : \, v\preccurlyeq \psi \}$ and $\PSH(Y,\eta)$, which also preserve the total Monge-Ampère mass, i.e. $\int_X MA_\omega(v)=\int_Y MA_\eta(\tilde{v})$ \cite[Lemma 4.6]{Tru20b}. This allows to study the properties of $\PSH(Y,\eta)$ and of the nef class $\{\eta\}$ directly over $X$.\newline 
We say that $\psi\in \cM(X,\omega)$ is an \emph{algebraic model type envelope} if $\psi=P_\omega[u](0)$ for $u$ with algebraic singularities type, and we denote by $\cM_D(X,\omega)$ the weak closure under decreasing limits of the set of algebraic model type envelopes.
\begin{remark}
\label{rem:Imodel}
The set $\cM_D(X,\omega)$ has been introduced in \cite[Definition 3.2]{Tru20c} and independently in \cite{DX20} using the \emph{singularity data} encoded in the multiplier ideal sheaves $\mathcal{I}(t\psi)$. More precisely, as a consequence of \cite[Theorem 3.3, Proposition 3.5]{Tru20c}, a model type envelope $\psi$ belong to $\cM_D(X,\omega)$ if and only if $u\preccurlyeq \psi$ for any $u\in \PSH(X,\omega)$ such that $\mathcal{I}(tu)=\mathcal{I}(t\psi)$ for any $t>0$. The elements in $\cM_D(X,\theta)$ are called \emph{$\mathcal{I}$-models} in \cite{DX20} where it is proven that they characterize the closure of test configurations (see \cite[Theorem 1.1]{DX20}).
\end{remark}
We will also need the following fact, which is a consequence of \cite[section 3]{Tru20c} (see also \cite[section 2.4]{DX20}). Assume that $X$ is projective, $Y\to X$ is a birational morphism and let $F$ be a prime divisor on $Y$. If $\psi_k\in\cM_D^+(X,\omega)$ is a decreasing sequence converging to $\psi\in \cM_D^+(X,\omega)$ then  $\nu(\psi_k,F)\nearrow \nu(\psi,F)$ (we denote by $\nu(u,F)$ the generic Lelong number of $u$ along $F$).
\subsubsection{Relative K\"ahler-Einstein metrics}
Assume $X$ Fano and let $\omega$ be a K\"ahler form such that $\{\omega\}=c_1(X)$.\newline

We recall that the Ricci curvature associated to $\omega$ is the $(1,1)$-form defined as $\Ric(\omega)=-dd^c\log \omega^n$. In particular, using the $1$-$1$ correspondence between metrics on $K_X$ and volumes forms given by $\mu_h=e^{-2f}i^{n^2}\Omega\wedge \bar{\Omega}$ where $f:=\log\lvert \Omega \rvert_h$ for any local holomorphic volume form $\Omega$, the Ricci curvature extends to singular metrics $he^{-2u}$ such that $MA_\omega(u)$ corresponds to singular metrics on $K_X$. Indeed, defining $\Ric(\mu):=dd^c f$ if locally $\mu=e^{-2f}i^{n^2}\Omega\wedge \bar{\Omega}$, one uses the notation
$$
\Ric(\omega+dd^c u):=\Ric\big(MA_\omega(u)\big)
$$
if $MA_\omega(u)$ has Lebesgue density.
\begin{defi}
\label{Defi:KE}
Let $\psi\in \cM^+(X,\omega)$. Then $\omega+dd^c u$ is said to be a $[\psi]$\emph{-K\"ahler-Einstein} ($[\psi]$-KE) if $u$ solves
\begin{equation}
    \label{eqn:KE}
    \begin{cases}
    \Ric(\omega+dd^c u)=\omega+dd^c u\\
    u\in \PSH(X,\omega), \, \lvert u-\psi\rvert\leq C
    \end{cases}
\end{equation}
\end{defi}
One checks that $\omega +dd^c u$ is a $[\psi]$-KE metrics if and only if $u\in \PSH(X,\omega)$ solves the complex Monge-Ampère equation
$$
\begin{cases}
MA_\omega(u)=e^{-2u}d\mu\\
\lvert u- \psi\rvert\leq C
\end{cases}
$$
where $d\mu=e^{2\rho}\omega$ up to translate $\rho$ (see \cite[Proposition 4.3]{Tru20b}). We can assume wlog that $d\mu$ is a probability measure: it is the \emph{adapted probability measure to $h$}, i.e. $\Ric(\mu)=\omega$.\newline
A necessary condition to have $[\psi]$-KE metrics is that $\mathcal{I}(\psi)=\mathcal{O}_X$ ($e^{-2\psi}\in L^1$), i.e. that $(X,\psi)$ is \emph{klt}. We refer to this condition adding the subscript "klt". For instance
$$
\cM^+_{klt}(X,\omega):=\big\{\psi\in\cM^+(X,\omega)\, : \, (X,\psi)\, \mathrm{is}\, \mathrm{klt}\big\}.
$$

When $\psi=0$, Definition \ref{Defi:KE} coincides with the definition of the genuine K\"ahler-Einstein metric since any solution $\omega+dd^c u$ of (\ref{eqn:KE}) can be then showed to be smooth.

When $\psi$ is a klt algebraic model type envelope with singularities encoded in $(\mathcal{I},c)$, then by \cite[Proposition 4.8]{Tru20b} there is a $1$-$1$ correspondence between $[\psi]$-KE metrics and \emph{log KE metrics on the weak Fano pair $(Y,\Delta)$} where $\pi:Y\to X$ is given as log-resolution of $\mathcal{I}$ while $\Delta=cD-K_{Y/X}$ for $D$ effective divisor such that $\pi^{-1}\mathcal{I}=\mathcal{O}_Y(-D)$. The upshot is that "log KE metrics \emph{above} $X$" are particular KE metrics with prescribed singularities.
\subsection{Log K-stability}
\label{sec:K}
In this subsection, $X$ will be a projective variety\footnote{By projective variety we mean a reduced and irreducible projective scheme over $\C$}.\newline
Borrowing the Minimal Model Program terminology, a $\Q$-Weil divisor $B$ on $X$ is called \emph{boundary} if $K_{(X,B)}:=K_X+B$ is $\Q$-Cartier. In this case $(X,B)$ is said to be a \emph{pair}.\newline
We also recall that for any proper birational morphism $\mu:Y\to X$ with $Y$ normal and any $F\subset Y$ prime divisor, the \emph{log-discrepancy} function $A_{(X,B)}(\cdot)$ of the pair $(X,B)$ is defined as
$$
A_{(X,B)}(F):=1+{\rm ord}_F(K_{Y/(X,B)}),
$$
where $K_{Y/(X,B)}:=K_Y-\mu^*K_{(X,B)}$. The log-discrepancy measures the singularities of the pair $(X,B)$, and $(X,B)$ is said to be \emph{sublc} (resp. \emph{subklt}) if $A_{(X,B)}(F)\geq 0$ (resp. $A_{(X,B)}(F)>0$). If furthermore the boundary $B$ is effective then we use the terminology \emph{lc} (resp. \emph{klt}).\newline
When $B=0$, we will simply set $A_X:=A_{(X,0)}$.
\subsubsection{Test configurations}
We recall the definition of test configurations.
\begin{defi}
A test configuration $\X$ for $X$ consists of:
\begin{itemize}
    \item[(i)] a flat and projective morphism of schemes $\pi: \X \lra \A^1 $;
    \item[(ii)] a $\C^*$-action on $\X$ lifting the canonical action on $\A^1$;
    \item[(iii)] an isomorphism $\X_1\simeq X$.
\end{itemize}
\end{defi}
Observe that $\X$ is a projective variety (see \cite[Proposition 2.6.(vi)]{BHJ17}).\newline
A test configuration $\X'$ \emph{dominates} $\X$ if the canonical $\C^*$-isomorphism $\X'\setminus \X'_0\simeq \X\setminus \X_0$ extends to a morphism $\X'\to \X$. Note that any two test configurations can be dominated by a third. In particular, for any test configuration $\X$ there exists a test configuration $\X'$ dominating $\X$ and the trivial test configuration $X\times \A^1$: the so-called \emph{determination of $\X$}. Any test configuration $\X$ admits a smooth determination $\X'$.\newline

Assume now that $X$ is endowed of a $\Q$-line bundle $L$.
\begin{defi}
\label{defi:TC}
A test configuration $(\X,\Ll)$ for $(X,L)$ consists of a test configuration $\X$ for $X$ and of:
\begin{itemize}
    \item[(iv)] a $\C^*$-linearized $\Q$-line bundle $\Ll$ on $\X$;
    \item[(v)] an isomorphism $(\X_1,\Ll_1)\simeq (X,L)$ extending the one in (iii).
\end{itemize}
\end{defi}
A test configuration $(\X,\Ll)$ is said to be ample, semiample (resp. normal) if $\Ll$ is ample, semiample (resp. $\X$ is normal).\newline

There are useful operations on test configurations:
\begin{itemize}
    \item given $(\X,\Ll)$ test configuration for $(X,L)$ and given $\X'$ test configuration for $X$ that dominates $\X$ through $\X'\overset{\rho}{\to}\X$, the \emph{pull-back of $(\X,\Ll)$} is the test configuration $(\X',\rho^*\Ll)$;   
    \item given two test configurations $(\X,\Ll), (\X,\Ll')$ for $(X,L), (X,L')$ and $c,c'\in \Q$, we can consider the linear combination $(\X, c\Ll+c'\Ll')$ which is clearly a test configuration for $(X,cL+c'L')$. In particular, for any $c\in \Q$, twisting the $\C^*$-linearization on $\Ll$ by $t^c$ produces a test configuration that can be identified with $(\X,\Ll+c\X_0)$.
\end{itemize}
Another essential tool is given by the \emph{compactification of a test configuration}. Gluing together $\X$ and $X\times (\PP^1\setminus\{0\})$ through the $\C^*$-equivariant isomorphism $\X\setminus \X_0\simeq X\times (\C\setminus \{0\})$, we obtain the compactification $\overline{\X}$. We denote by $\overline{\Ll}$ the $\C$-linearized $\Q$-line bundle on $\overline{\X}$.\newline

We also need to recall that a normal test configuration $(\X,\Ll)$ is said to be \emph{trivial} if $\X\simeq X\times\C$. In this case there exists $c\in\Q$ such that $\Ll=p_1^*L+c\X_0$.\newline
To characterize the triviality of test configurations, in \cite{Der16} and independently in \cite{BHJ17} a \emph{minimum norm} has been introduced. If $\X$ dominates $X\times\C$ through a morphism $\mu_\X:\X\to X\times\C$, then it is defined as
$$
J^{\NA}(\X,\Ll):=V^{-1}\big(\overline{\mu_\X^*p_1^*L}^n\cdot \overline{\Ll}\big)-V^{-1}(n+1)^{-1}\big(\overline{\Ll}^{n+1}\big)
$$
for $V:=(L^n)$, while for the general case it is defined taking a determination of $\X$ and observing that $J^{\NA}$ is pull-back invariant.\newline
Then $J^{\NA}$ it is non-degenerate in the sense that $J^{\NA}(\X,\Ll)\geq 0$ with equality if and only if $(\X,\Ll)$ is a trivial test configuration (\cite[Theorem 1.3]{Der16}, \cite[Corollary B]{BHJ17}).

\subsubsection{Log Donaldson-Futaki invariant}
\label{sssec:DF}
Assume $X$ to be normal, and let $B$ a boundary. Then, for any normal test configuration $(\X,\Ll)$ for $(X,L)$ we denote by
\begin{equation}
    \label{eqn:DFinvariant}
    DF_B(\X,\Ll):=V^{-1}(K_{(\Bar{\X},\Bar{\CB})/\PP^1}\cdot \Bar{\Ll}^n)+\Bar{S}_BV^{-1}\frac{(\Bar{\Ll}^{n+1})}{n+1}
\end{equation}
the \emph{log Donaldson-Futaki invariant of $(\X,\Ll)$ with respect to the pair $(X,B)$}, where:
\begin{itemize}
    \item $V:=(L^n)$ is the volume of the polarized variety $(X,L)$;
    \item $\Bar{S}_B:=nV^{-1}(-K_{(X,B)}\cdot L^{n-1})$ is the log version of the mean scalar curvature;
    \item $\CB$ (resp. $\Bar{\CB}$) is the $\Q$-Weil divisor on $\X$ (resp. on $\Bar{\X}$) given by the component-wise Zariski closure in $\X$ (resp. $\Bar{\X}$) of the $\C^*$-invariant $\Q$-Weil divisor on $\X\setminus \X_0$ obtained by extending $B\subset\X_1$;
    \item $K_{(\Bar{\X},\Bar{\CB})/\PP^1}:=K_{(\Bar{\X},\Bar{\CB})}-\pi^*K_{\PP^1}=K_{\Bar{\X}}+\Bar{\CB}-\pi^*K_{\PP^1}$.
\end{itemize}
It is not difficult to check that the log Donaldson-Futaki invariant is preserved under pull-back, i.e. $DF_B(\X,\Ll)=DF_B(\X',\rho^*\Ll)$ if $\rho:\X'\to \X$, and under translation by a constant, i.e. $DF_B(\X,\Ll+c\X_0)=DF_B(\X,\Ll)$.
\begin{defi}
Let $(X,L)$ be a polarized projective variety. Then $(X,L)$ is said to be \emph{log $K$-stable} with respect to the boundary $B$ if $DF_B(\X,\Ll)\geq 0$ for any normal ample test configuration $(\X,\Ll)$ with equality if and only if $(\X,\Ll)$ is trivial.\newline
$(X,L)$ is said to be \emph{uniformly log $K$-stable} with respect to the boundary $B$ if there exists $\delta>0$ such that $DF_B(\X,\Ll)\geq \delta J^{\NA}(\X,\Ll)$ for any normal ample test configuration $(\X,\Ll)$.
\end{defi}
By what recalled previously, uniform log $K$-stability is a strengthening of log $K$-stability.
\begin{remark}
In the case $B=0$, the Donaldson-Futaki invariant can be defined for general test configuration and the usual definition of K-stability requires to consider all ample test configurations. On the other hand any test configuration $(\X,\Ll)$ admits a \emph{normalization} $(\Tilde{\X},\Tilde{\Ll})$, which is still a test configuration for $(X,L)$, given by $\nu:\Tilde{\X}\to \X$, normalization of $\X$, and by $\Tilde{\Ll}:=\nu^*\Ll$. Then \cite[Proposition 3.15]{BHJ17} implies that it is enough to consider normal ample configurations.
\end{remark}
\subsection{The Non-Archimedean perspective}
\label{ssec:NA}
To study the (log) K-stability notions there is a useful Non-Archimedean point of view which we will now briefly recall.
\subsubsection{Valuations}
A \emph{valuation} on a finitely generated field extension $K$ of $k$ is a group homomorphism $v:K^*\to (\R,+)$ such that $v(f+g)\geq \min\{v(f),v(g)\}$ and $v_{|k*}\equiv 0$. It will be often convenient to set $v(0)=+\infty$.
The trivial valuation $v_{triv}$ is defined by $v_{triv}\equiv 0$.\newline
We will denote by $X^{\divv}$ the set of all \emph{divisorial valuations}, i.e. the set of all valuation of the form $v=c\ord_F$ where $c>0$ and $F$ prime divisor \emph{over} $X$, i.e. there exists a normal variety $Y$ and a proper birational morphism $\mu:Y\to X$. In this case the generic point $\xi$ of $\mu(F)$ represents the \emph{center} of the valuation, i.e. $v\geq 0$ on $\Ox_{X,\xi}$ and $v>0$ on its maximal ideal.\newline

The inclusion $K\subset K(t)$ produces a natural restriction map $v\to r(v)$ which sends divisorial valuations to trivial or divisorial valuations \cite[Lemma 4.1]{BHJ17}. On the other hand the \emph{Gauss extension} of a valuation $v$ on $K$ is the valuation on $K(t)$ defined as
$$
G(v)(f):=\min_{\lambda\in\Z}\big(v(f_\lambda)+\lambda\big)
$$
for all $f$ with Laurent polynomial expansion $f=\sum_{\lambda\in\Z}f_\lambda t^{\lambda}$. In particular $r\big(G(\nu)\big)=\nu$.

\subsubsection{Test configurations as Non-Archimedean metrics}
Each test configuration $(\X,\Ll)$ defines a function on $X^{\divv}$ in the following way. Let $(\Tilde{\X},\Tilde{\Ll})$ be a smooth determination of $(\X,\Ll)$, i.e. $\X$ is smooth and there are $\C^*$-equivariant morphisms $\rho:\tilde{\X}\to \X$ and $\mu :\tilde{\X}\to X\times \C$. Then
$$
\tilde{\Ll}:=\rho^*\Ll=\mu^*p_X^*L+F
$$
for a unique $\Q$-divisor $F$ supported on $\tilde{\X}_0$, and $\varphi_{(\X,\Ll)}:X^\divv\to \R$ is defined as
$$
\varphi_{(\X,\Ll)}(\nu):=G(\nu)(F).
$$
By construction, $\varphi_{(\X_1,\Ll_1)}=\varphi_{(\X_2,\Ll_2)}$ if and only if there exists $\X_3\overset{f_i}{\geq} \X_i$ such that $f_1^*\Ll_1=f_2^*\Ll_2$. Namely, iff they are \emph{equivalent} in the sense of \cite[Definition 6.1]{BHJ17}. As in \cite{BHJ17} we denote by $\CH^\NA(L):=\{\varphi_{(\X,\Ll)}\, :\, (\X,\Ll) \, \mathrm{ample}\}$ the equivalence classes of test configurations for $(X,L)$ that contains an ample test configurations, and we refer to \cite[section 6.8]{BHJ17} (and references therein) for the interpretation of $\CH^\NA(L)$ as set of positive \emph{Non-Archimedean metrics}.

\subsection{Riemann-Zariski space}
\label{ssec:RZ}
Let $X$ to be a projective variety. \newline
To work with classes living on arbitrary modifications of $X$ it is natural to introduce the following construction. As in \cite{BFJ09}, we endow the set of all possible birational morphism $Y\to X$ where $Y$ is a smooth variety with the partial order relation $Y'\geq Y$ if there exists a birational morphism $Y'\to Y$, and we define the \emph{Riemann-Zariski space} as the projective limit of this directed and ordered family, i.e.
$$
\FX:=\varprojlim_{Y} Y.
$$
According to \cite[Definition 1.1]{BFJ09}, denoting by $N^p(Y)$ the real vector space of numerical equivalence of $p$-codimensional cycles, the \emph{space of $p$-codimensional Cartier classes on $\FX$} is defined as the inductive limit of $N^p(Y)$ where the arrows are induced by pull-back, i.e.
$$
CN^p(\FX):=\varinjlim_{Y}N^p(Y).
$$
Similarly, the \emph{space of $p$-codimensional Weil classes on $\FX$} is defined as the projective limit of $N^p(Y)$, i.e.
$$
N^p(\FX):=\varprojlim_{Y}N^p(Y).
$$
These two spaces are endowed by their natural topologies. \newline
In practice, any element $D\in N^p(\FX)$ is described by its family of incarnations $D_Y\in N^p(Y)$ for any $Y$, where if $Y'\overset{\rho}{\geq} Y$ then
$$
\rho_{*}D_{Y'}=D_{Y}.
$$
Note also that a sequence $\{D_{k}\}_{k\in\N}\subset N^p(\FX)$ converges to $D\in N^p(\FX)$ iff $D_{Y,k}\to D_Y$ in $N^p(Y)$ for any $Y$ (this is also called \emph{weak topology}).\newline
A element $D\in CN^p(\FX)$ can be instead described by its \emph{determination} $D_Y\in N^p(Y)$, i.e. by its incarnation on $Y$, where $Y$ is chosen so that for any $Y'\overset{\rho}{\geq} Y$
$$
D_{Y'}=\rho^*D_Y.
$$
It is then not difficult to check that there are natural injective continuous maps
\begin{gather*}
    N^p(Y)\hooklongrightarrow CN^p(\FX),\\
    CN^p(\FX)\hooklongrightarrow N^p(\FX),
\end{gather*}
and that moreover the second map has dense image in $N^p(\FX)$ \cite[Lemma 1.2]{BFJ09}.
\subsubsection{Positivity notions}
Given $Y'\overset{\rho}{\geq} Y$ and $D\in N^1(Y)$, it is well-known that $\rho^*D$ is nef (resp. big, psef) iff $D$ is nef (resp. big, psef), i.e. these two positivity notions are preserved under pull-back. Hence the space $CN^1(\FX)$ naturally inherits these positivity notions.\newline
In particular, we can say that $D_1\geq D_2$ for $D_1,D_2\in CN^1(\FX)$ if $D:=D_1-D_2\in CN^1(\FX)$ is psef, i.e. if $D_Y$ is psef for some determination of $D$.\newline
Since the psef notion makes also sense among $p$-codimensional numerical classes, we say that $\alpha\geq \alpha'$ for $\alpha,\alpha'\in CN^p(\FX)$ if $\alpha-\alpha'$ is pseudoeffective.\newline

A class $\alpha\in N^p(\FX)$ is said to be \emph{psef} if $\alpha_Y\in CN^p(\FX)$ is psef for any $Y\geq X$. Note that the nomenclature chosen is coherent since a class $\alpha\in CN^p(\FX)$ is psef as a Weil class if and only if it is psef as a Cartier class.

Similarly a nef Cartier class $\alpha\in CN^1(\FX)$ can also be interpreted as a \emph{nef} Weil class in $N^1(\FX)$, since the latter are defined as weak closure of nef Cartier classes.
\subsubsection{Positive intersection product}
\label{ssec:PIPS}
Given $D_1,\dots,D_p\in CN^1(\FX)$ nef, their intersection product is naturally defined as
$$
D_{1,Y}\cdots D_{p,Y}\in CN^p(\FX)
$$
where $Y$ is chosen so that $D_1,\dots,D_p$ have a common determination on $Y$.\newline
In \cite{BDPP13, BFJ09} the authors extended the intersection product to big divisors in $CN^p(\FX)$. Given $D_1,\dots,D_p\in CN^1(\FX)$ big, their \emph{positive intersection product}
$$
\langle D_1\cdots D_p \rangle\in N^p(\FX)
$$
is defined as the least upper bound of the set of classes
$$
(D_1-F_1)\cdots(D_p-F_p) 
$$
where $F_i$ varies over the effective Cartier $\Q$-divisors on $\FX$ such that $D_i-F_i\in CN^1(\FX)$ is nef.
\begin{theorem}{\cite[Theorem 3.5]{BDPP13}, \cite[Propositions 2.9, 2.12, 2.13]{BFJ09}}
\label{thm:PIP}
The positive intersection product is symmetric homogeneous of degree 1, super-additive in each variable and it varies continuously on the $p$-fold product of the big cone in $CN^1(\FX)$. Moreover $\langle D_1\cdots D_p\rangle$ coincides with the least upper bound of the set of all intersection products $(F_1\cdots F_p)$ where $F_i\in CN^1(\FX)$ are nef classes such that $F_i\leq D_i$. In particular if $D_1,\dots,D_p$ are nef then
$$
\langle D_1\cdots D_p\rangle=(D_1\cdots D_p).
$$
\end{theorem}
\subsubsection{Volume \& Restricted volume}
\label{ssec:Vol}
To any Cartier divisor $D$ on a variety $Y$ is associated its \emph{volume}
$$
\vol_Y(D):=\limsup_{k\to +\infty}\frac{n!}{k^n}h^0(Y,kD)\in \R_{\geq 0}
$$
which is $>0$ exactly when $D$ is big. Then Fujita's Theorem says that
$$
\vol_Y(D)=\langle D^n\rangle
$$
when $D$ is big \cite{Fuj94}, where on the right hand side we obviously mean the positive intersection product of the Cartier class $D\in CN^1(\FX)$ associated to the big divisor $D$.\newline
It is well-known that the volume extends to a continuous function on $N^1(Y)$ (see for instance \cite{Laz04}). It actually defines a $C^1$-function on the big cone: for any $F=\sum_j\alpha_jF_j$ divisor on $Y$, $F_j$ prime divisors, the equality
$$
\frac{d}{dt}_{|t=0}\vol_Y(D+tF)=n\sum_j \alpha_j \vol_{Y|F_j}(D)
$$
holds (see \cite[Corollary C]{LM09}, \cite[Theorem A, Theorem B]{BFJ09}). The \emph{restricted volumes} on the right hand side are defined as
$$
\vol_{Y|F_j}(D):=\limsup_{k\to +\infty}\frac{(n-1)!}{k^{n-1}}h^0(Y|F_j,kD)
$$
for $h^0(Y|F_j,kD)$ rank of the restriction map $H^0(Y,kD)\to H^0(F_j,kD_{|F_j})$.\newline
The restricted volume has also the following interpretation in terms of the positive intersection product:
$$
\vol_{Y|F_j}(D)=\langle D^{n-1}\rangle \cdot F_j
$$
(see \cite[Theorem B]{BFJ09}), where the pairing $N^{n-1}(\FX)\times CN^1(\FX)\to N^n(\FX)\simeq \R$ extends the usual intersection product of Cartier divisors by the projection formula. In particular the restricted volume $\vol_{Y|F}(\cdot)$ naturally extends to a continuous function on the big cone.

\section{$\bD$-Log $K$-stability}
\label{sec:K-stability}
Let $X$ be a normal variety and let $L\to X$ be an ample line bundle.\newline

With the notation $Y\geq X$ we mean that $Y$ is a smooth projective variety and that there exists a birational morphism $Y\to X$.
\begin{defi}
\label{defi:genbdiv}
For any $Y\geq X$ let $D_Y=\sum a_{F_Y} F_Y$ where the formal sum is over all prime divisor $F_Y\subset Y$ and $a_{F,Y}\in\R$. We will then say that $\bD:=\{D_Y\}_{Y\geq X}$ is a \emph{generalized $b$-divisor on $X$} if
\begin{itemize}
    \item[i)] for any $Y\geq X$ the support $\mathrm{Supp}(D_Y)=\left\{F_Y\, : \, a_{F_Y}\neq 0\right\}$ is at most countable (we then write $D_Y=\sum_j a_{Y,j} F_{Y,j}$);
    \item[ii)] for any $Y'\overset{p}{\geq} Y$ the equality $p_* D_{Y'}=D_Y$ holds, where $p_*D_{Y'}=\sum_j a_{Y',j} p_*(F_{Y',j})$.
\end{itemize}
We will denote by $\bDiv^g(X)$ the set of generalized $b$-divisors on $X$.
\end{defi}
Clearly a generalized $b$-divisor $\bD$ is a $b$-divisor in the sense of \cite{Sho03} if and only if $D_Y$ is a $\R$-divisor for any $Y\geq X$ (i.e. the sum in $(i)$ is actually a finite sum).
\begin{defi}
\label{defi:Lpositive}
We say that a generalized $b$-divisor $\bD\in \bDiv^g(X)$ is \emph{$L$-positive} if
\begin{itemize}
    \item[i)] $\bD$ is effective and increasing:  $a_{F_Y}\geq 0$ for any $Y\geq X$, $F\subset Y$ prime divisor and if $Y'\geq Y$ through a birational morphism $p:Y'\to Y$ then $p^*D_Y\leq D_{Y'}$;
    \item[ii)] $L-\bD$ is nef;
    \item[iii)] there exists a constant $a>0$ such that $\vol_Y(\rho_Y^*L-D_Y)\geq a$ for any $Y\overset{\rho_Y}{\geq} X$.
\end{itemize}
We denote by $\bDiv_L(X)$ the set of $L$-positive generalized $b$-divisors on $X$.
\end{defi}
When $D_Y=\sum_n a_nF_n$ is a infinite sum, the numerical classes $\{\sum_{n=1}^ka_n F_n\}\in N^1(Y)$ form an increasing sequence of pseudoeffective classes which is uniformly bounded above by $\rho_Y^*L$. Thus numerical classes $\{D_Y\}\in N^1(Y)$, given as limit in $N^1(Y)$, defines a class $\alpha\in N^1(\FX)$ such that the incarnations $(L-\alpha)_Y$ are \emph{uniformly big} by the condition $(iii)$.
The second condition instead means that the numerical class $L-\alpha$ is nef as a Weil class in $N^1(\FX)$, i.e. it belongs to the weak closure of the set of all nef Cartier classes (see subsection \ref{ssec:RZ}). Combining \cite[Proposition 2.4]{Bou04} and \cite[Lemma 2.10]{BdFF12}, the latter condition is equivalent to say that $L_Y:=\rho_Y^*L-D_Y$ is \emph{nef in codimension 1}, which in turn can be expressed by saying that $L_Y$ has a closed and positive current with zero Lelong number along any prime divisor thanks to \cite[Propositions 3.2, 3.6]{Bou04} and the fact that $L_Y$ is big.
Let us stress that $(i)$ is encoded in the generalized $b$-divisor $\bD$ and not only on its class $\alpha\in N^1(\FX_{\PP^1}^{\C^*})$, i.e. it is not a \emph{numerical condition}.

\subsection{Riemann-Zariski perspective}
To define a $K$-stability notion with respect to a $L$-positive generalized $b$-divisor $\bD\in \bDiv_L(X)$, the inspiring idea is that the \emph{positivity} to study is encoded in the classes associated to $L_Y:=\rho_Y^*L-D_Y$ where $Y\overset{\rho_Y}{\geq}X$.\newline
Thus, similarly to subsection \ref{ssec:RZ}, it appears natural to consider all test configurations for $(Y,L_Y)$ on equal footing defining a $\C^*$-equivariant version of the Riemann-Zariski space.\newline

We consider all $\CZ$ smooth projective varieties that dominates $X\times \PP^1$ through a $\C^*$-equivariant birational morphism with respect to the canonical $\C^*$-action on $\PP^1$. This set, endowed with its natural partial order (denoted by $\CZ\overset{\C^*}{\geq}X\times\PP^1$), is non-empty and directed.
\begin{defi}
The projective limit of smooth projective varieties $\CZ\overset{\C^*}{\geq}X\times\PP^1$,
$$
\FX_{\PP^1}^{\C^*}:=\varprojlim_{\CZ}\CZ,
$$
will be called the $\C^*$\emph{-equivariant Riemann-Zariski space over $X\times\PP^1$}.
\end{defi}
\emph{The space of $\C^*$-invariant $p$-codimensional Cartier (resp. Weil) classes on $\FX_{\PP^1}^{\C^*}$} $CN^p(\FX_{\PP^1}^{\C^*})$ (resp. $N^p(\FX_{\PP^1}^{\C^*})$) can be then defined similarly as done in subsection \ref{ssec:RZ} considering $\C^*$-invariant classes.\newline
Clearly, if we denote by $\FX_{\PP^1}$ the Riemann-Zariski space over $X\times \PP^1$ we have continuous injective maps
\begin{gather*}
    CN^1(\FX_{\PP^1}^{\C^*})\hooklongrightarrow CN^1(\FX_{\PP^1}),\\
    N^1(\FX_{\PP^1}^{\C^*})\hooklongrightarrow N^1(\FX_{\PP^1}).
\end{gather*}

By what said in section \ref{sec:K}, any test configuration $\X$ for $X$ has a natural compactification over $\PP^1$ and any smooth and dominating test configuration $(\X,\Ll)$ for $(X,L)$, defines an element in $CN^1(\FX_{\PP^1}^{\C^*})$.
\begin{defi}
\label{defi:Lamp}
Let $L\to X$ be an ample line bundle. We say that $\alpha\in N^1(\FX_{\PP^1}^{\C^*})$ is a \emph{test configuration} class (resp. an \emph{ample test configuration class}) if there exists a (resp. an ample and) normal test configuration $(\X',\Ll')$ for $(X,L)$ such that $\alpha=\Ll$ through the injective continuous maps
$$
N^1(\X)\hooklongrightarrow CN^1(\FX_{\PP^1}^{\C^*})\hooklongrightarrow N^1(\FX_{\PP^1}^{\C^*})
$$
where $\X$ is a smooth projective variety that $\C^*$-equivariantly dominates $\overline{\X'}$ while $\Ll$ is the $\Q$-line bundle on $\X$ given by pulling back $\overline{\Ll'}$.\newline
To lighten the notation, we will denote by $\FX_\alpha$ the set of all pairs $(\X,\Ll)$ such that \begin{itemize}
    \item[i)] $(\X,\Ll)$ is associated to $\alpha$ as above;
    \item[ii)] $\X$ dominates $X\times\PP^1$;
    \item[iii)] $\X\setminus\X_0\simeq X\times(\PP^1\setminus\{0\})$ (i.e. $\X$ is given as compactification of a smooth test configuration).
\end{itemize}
\end{defi}

As said in subsection \ref{ssec:NA}, two test configurations $(\X_1,\Ll_1), (\X_2,\Ll_2)$ for $(X,L)$ are said to be \emph{equivalent}, $(\X_1,\Ll_1)\sim (\X_2,\Ll_2)$ if there exists a third test configuration $(\X_3,\Ll_3)$ for $(X,L)$ which is the pullback of both $(\X_1,\Ll_1),(\X_2,\Ll_2)$ (\cite[Definition 6.1]{BHJ17}). Equivalently $(\X_1,\Ll_1)\sim (\X_2,\Ll_2)$ if $\varphi_{(\X_1,\Ll_1)}=\varphi_{(\X_2,\Ll_2)}$.
\begin{lemma}
\label{lem:Nice}
The natural map
$$
\iota:\left\{(\X,\Ll)\, \mathrm{normal}\, \mathrm{test}\, \mathrm{configuration}\, \mathrm{for}\, (X,L)\right\}/\sim \longrightarrow N^1(\FX_{\PP^1}^{\C^*})
$$
is injective with image the set of test configuration classes. In particular, $\CH^\NA(L)$ can be identified with the set of ample test configuration classes.
\end{lemma}
Clearly $\iota([(\X,\Ll)])=\Ll$ where $(\X,\Ll)$ is a smooth representative and where by an abuse of notation on the right hand side we mean the numerical class associated to the compactification of $\Ll$ (i.e. $\{\overline{\Ll}\}\in N^1(\X)\subset N^1(\FX_{\PP^1}^{\C^*})$).
\begin{proof}
By definition two normal test configurations $(\X_1,\Ll_1), (\X_2,\Ll_2)$ for $(X,L)$ defines the same test configuration class $\alpha \in N^1(\FX_{\PP^1}^{\C^*})$ if and only if $\alpha$ is also defined by a normal test configuration $(\X_3,\Ll_3)$ such that $\X_3\overset{\rho_i}{\geq}\X_i$ and $\Ll_3\equiv_{num}\rho_i^*\Ll_i$. In particular $\iota$ is well-defined and to conclude the proof it is enough to prove that if $(\X,\Ll_1),(\X,\Ll_2)$ are two smooth and dominating $X\times\C$ test configurations such that $\Ll_1\equiv_{num}\Ll_2$ then $\Ll_1=\Ll_2$. But by definition of test configurations
$$
\Ll_i=\mu^*p_X^*L + D_i
$$
for a unique $\Q$-divisor $D_i$ supported on the central fiber. Thus, $\Ll_1\equiv_{num}\Ll_2$ implies that $D_1-D_2\equiv_{num} 0$, and an application of Zariski's Lemma (see for instance the proof of \cite[Proposition 3.2.18]{Sjo17}) leads to $D_1=D_2$ and hence to $\Ll_1=\Ll_2$.
\end{proof}

Next, we observe that there is a injective continuous map
$$
CN^1(\FX)\hooklongrightarrow CN^1(\FX_{\PP^1}^{\C^*})
$$
induced by the projection map $p_X:X\times\PP^1\to X$. Indeed, given a Cartier class $\alpha\in CN^1(\FX)$ and letting $Y\geq X$ such that $\alpha$ has a determination on $Y$, the Cartier class $\alpha\times\PP^1\in CN^1(\FX_{\PP^1}^{\C^*})$ is induced by the class $\alpha_Y\times\PP^1\in CN(Y\times\PP^1)$. In particular, letting $\bD\in \bDiv_L(X)$ be a $L$-positive generalized $b$-divisor, the net of Cartier classes $\{D_Y\times\PP^1\}_{Y\geq X}\subset N^1(\FX_{\PP^1}^{\C^*})$ is increasing and bounded above by $L\times \PP^1$. Thus, \cite[Proposition 1.5]{BFJ09} implies that the net converges to an element which we will denote by $\bD\times \PP^1\in N^1(\FX_{\PP^1}^{\C^*})$. \newline
In practice, if $\CZ_{\pi^{-1}\C^*}\simeq Y\times \C^*$, one has
$$
(\bD\times\PP^1)_\CZ=f_*\mu'^* (D_Y\times \PP^1)
$$
with respect to any diagram
$$
\begin{tikzcd}
    \CZ' \ar[dr, "\mu'"] \ar[r, "f"] & \CZ \ar[d, dashed, "\mu"]\\
    & Y\times\PP^1 
\end{tikzcd}
$$
Similarly, one can construct the Weil class $(L-\bD)\times \PP^1\in N^1(\FX_{\PP^1}^{\C^*})$.\newline

Moreover, any generalized $b$-divisor $\bD=\{D_Y\}_{Y\geq X}$ on $X$ defines a $\C^*$-invariant generalized $b$-divisor $\CD$ on $X\times\PP^1$ in a natural way. Indeed, letting $\CZ$ be a smooth projective variety $\CZ\overset{\C^*}{\geq} X\times \PP^1$ such that $ \CZ_{\pi^{-1}\C^*}\simeq Y\times\C^* $
for a smooth and projective variety $Y\geq X$, we define $\CD_{\CZ}:=\sum a_{F_Y}\CF_{Y,\CZ}$ if $D_Y=\sum a_{F_Y}F_Y$ where $\CF_{Y,\CZ}$ is the divisor given as Zariski closure in $\CZ$ of $F_Y\times \C^*$ with respect to the open embedding of $Y\times \C^*$ into $\CZ$.

Observe that $\CD=\{\CD_{\CZ}\}_{\CZ}$ defines a pseudoeffective class in $N^1(\FX_{\PP^1}^{\C^*})$ such that
$$
\CD\leq \bD\times\PP^1
$$
in $N^1(\FX_{\PP^1}^{\C^*})$ since, by construction, for any $\tilde{\CZ}\geq X\times\PP^1$ one has
$$
\CD_{\tilde{\CZ}}\leq D_{Y'}\times\PP^1
$$
where $\tilde{\CZ}_{\pi^{-1}\C^*}\simeq Y'\times\C^*$.
\begin{remark}
\label{rem:NI}
Two \emph{numerically equivalent} generalized $b$-divisors $\bD=\{D_Y\}_{Y\geq X}, \bD'=\{D_Y'\}_{Y\geq X}$, i.e. $\{\bD\}=\{\bD'\}\in N^1(\FX)$ as Weil classes on $\FX$, may define $b$-divisors $\CD,\CD'$ with different Weil classes $\{\CD\}\neq \{\CD'\}$ on $\FX_{\PP^1}^{\C^*}$. Indeed it is enough to consider $\mu_\X:\X:=\mathrm{Bl}_{(p,0)}(X\times \PP^1)\to X\times\PP^1$ where $p\in D, p\notin D'$ for two numerical equivalent divisor $D,D'$ to have $\{\CD'_\X\}=\{\mu_\X^*p_X^*D\}\in N^1(\X)$ while $\{\CD_\X\}=\{\mu_\X^*p_X^*D-c\mathcal{E}\}\in N^1(\X)$ where $\cE$ is the exceptional divisor while $c>0$ is a constant depending on the multiplicity of $D$ at $p$ and on $X$.
\end{remark}

\subsubsection{From test configurations to relatively positive classes}
By what observed before, the most natural classes to consider in this context are those given by
$$
\Ll-\CD\in N^1(\FX_{\PP^1}^{\C^*})
$$
where $\Ll\in CN^1(\FX_{\PP^1}^{\C^*})$ varies among all ample test configuration classes (Definition \ref{defi:Lamp}). However, a priori $\Ll-\CD$ may not be a psef class in $N^1(\FX_{\PP^1}^{\C^*})$. 
\begin{lem}
\label{lem:Effe}
Let $\Ll\in N^1(\FX_{\PP^1}^{\C^*})$ be a test configuration class. Then
$$
m_{0}:=\inf\{m\in\Q\, : \, \Ll+m(X\times\{0\})-L\times\PP^1 \, \mathrm{is}\, \mathrm{pseudoeffective}\}
$$
is a finite quantity, where $X\times\{0\},L\times\PP^1 \in N^1(\FX_{\PP^1}^{\C^*})$ are the Weil classes associated respectively to $X\times\{0\}, L\times \PP^1\in N^1(X\times \PP^1)$. Moreover, by an abuse of notations, for any $(\X,\Ll)\in \FX_\Ll$ and for any $m>m_0$ there exists a unique effective $\R$-divisor $F$ supported on the central fiber $\X_0$ such that the $\R$-divisor
$$
\mu_\X^*p_X^*L+F
$$
is a representative of $\Ll+m\X_0\in N^1(\X)$ where $\mu_\X: \X\to X\times\PP^1$.
\end{lem}
\begin{proof}
This is a well-known result, which follows from what said during the proof of Lemma \ref{lem:Nice}.
\end{proof}
\begin{prop}
\label{prop:Big}
Let $\Ll\in N^1(\FX_{\PP^1}^{\C^*})$ be a test configuration class. Then there exists $m_0\in \R$ such that
$$
\Ll_m:=\Ll+m(X\times \{0\})-\CD\geq (L-\bD)\times \PP^1 + (m-m_0)(X\times\{0\})
$$
for any $m\geq m_0$. Moreover, for any $m>m_0$, the net of big Cartier classes
$$
\{\Ll_{m,\CZ}\}_{\CZ\geq X\times\PP^1}:=\{\Ll+m(X\times\{0\})-\CD_\CZ\}_{\CZ\geq X\times\PP^1}\subset N^1(\FX_{\PP^1}^{\C^*})
$$
weakly converging to $\Ll_m$ satisfies
$$
\langle \Ll_{m,\CZ}^{n+1}\rangle\geq (n+1)(m-m_0)\langle L_Y^n\rangle
$$
for any $Y\geq X$ if $\CZ_{\pi^*\C^*}\simeq Y\times\C^*$, where $L_Y=\rho_Y^*L-D_Y$ for $\rho_Y:Y\to X$ birational morphism.
\end{prop}
\begin{proof}
The first assertion immediately follows from Lemma \ref{lem:Effe} observing that $\CD\leq \bD\times \PP^1$ and that $(L-\bD)\times \PP^1=L\times\PP^1-\bD\times\PP^1$. Then, if $\CZ_{\pi^*\C^*}\simeq Y\times\C^*$, as $\CD_\CZ\leq \bD\times\PP^1$, Theorem \ref{thm:PIP} implies
\begin{multline*}
    \langle\Ll_{m,\CZ}^{n+1} \rangle\geq \langle \left(L_Y\times \PP^1 + (m-m_0)X\times\{0\}\right)^{n+1}\rangle=\\
    =(n+1)\langle L_Y^n\rangle(m-m_0)\vol_{\PP^1}(H)=(n+1)(m-m_0)\langle L_Y^n\rangle
\end{multline*}
where by $H=\{0\}\in\PP^1$ we denoted the numerical class of any point in $\PP^1$.
\end{proof}
\begin{remark}
Since the Cartier classes $\Ll_{m,Y}$ in Propoostion \ref{prop:Big} are $\C^*$-invariant, the positive intersection product, which a priori considers big Cartier classes in Riemann-Zariski space $\FX_{\PP^1}$, can be seen as defined in the natural $\C^*$-equivariant way. Namely the product $\langle \alpha_1\cdots \alpha_p \rangle\in N^p(\FX_{\PP^1}^{\C^*})$ for $\alpha_j\in CN^1(\FX_{\PP^1}^{\C^*})$ big classes, is given as the least upper bound of the set of classes
$$
(\alpha_1-F_1)\cdots (\alpha_p-F_p)
$$
where $F_i$ varies over the effective and $\C^*$-invariant Cartier $\Q$-divisor on $\FX_{\PP^1}^{\C^*}$ such that $D_i-F_i$ in nef and it coincides with the product $\langle \alpha_1\cdots \alpha_p\rangle\in N^p(\FX_{\PP^1})$ where the classes $\alpha_j$ are seen as big Cartier classes in $\FX_{\PP^1}$. In particular the properties stated in Theorem \ref{thm:PIP} and the results listed in subsection \ref{ssec:Vol} keep to hold in the space $\FX_{\PP^1}^{\C^*}\subset \FX_{\PP^1}$.
\end{remark}

\subsection{The Donaldson-Futaki invariant}
Using the Riemann-Zariski perspective of the previous subsection, we want here to attach a natural weight $DF(\Ll;\bD)$, the \emph{Donaldson-Futaki invariant}, to any test configuration class $\Ll\in N^1(\FX_{\PP^1}^{\C^*})$.

\subsubsection{Restricted volume for Weil classes}
As recalled in section \ref{ssec:RZ} the projection formula allows to extend the intersection product $\alpha\cdot \beta$ for $\alpha \in CN^{n-1}(\FX_{\PP^1}^{\C^*}), \beta\in CN^1(\FX_{\PP^1}^{\C^*})$ to the case when either $\alpha$ or $\beta$ is just a Weil class.
\begin{defi}
\label{defi:Pro}
Let $\alpha\in N^{n-1}(\FX_{\PP^1}^{\C^*}), \beta\in N^1(\FX_{\PP^1}^{\C^*})$. Then
$$
\alpha \cdot \beta:=\sup_{\CZ'}\inf_{\CZ\geq \CZ'}\alpha_{\CZ}\cdot \beta_{\CZ}.
$$
when $\alpha_{\CZ}\cdot \beta_{\CZ}\in \R$ is given as product on $\CZ$.
\end{defi}
It follows from the projection formula that this definition extends the usual intersection product when at least one of the two classes is Cartier. Moreover when $\alpha, \beta$ are given as intersection of nef Weil classes then it coincides with the intersection product defined in \cite[Section 3]{DF21}.
\subsubsection{The intersection formula for the Donaldson-Futaki invariant}
Let $\Ll\in N^1(\FX_{\PP^1}^{\C^*})$ be a test configuration class. Then, by Proposition \ref{prop:Big} there exists $m_0\in\Q$ such that for any $m>m_0$
$$
\{\Ll_{m,\CZ}\}_{\CZ\geq X\times\PP^1}=\{\Ll+m(X\times\{0\})-\CD_\CZ\}_{\CZ\geq X\times\PP^1}\subset N^1(\FX_{\PP^1}^{\C^*})
$$
is a net of big Cartier classes which weakly converges to
$$
\Ll_m:=\Ll+m(X\times\{0\})-\CD\in N^1(\FX_{\PP^1}^{\C^*}).
$$
It is then natural to wonder what happens to the positive intersection products of $\Ll_{m,\CZ}$.
\begin{lem}
\label{lem:Ammappate}
Let $k=1,\dots, n+1$ and let $\CZ, \CZ'$ such that $\CZ_{\pi^{-1}\C^*}\simeq \CZ'_{\pi^{-1}\C^*}\simeq Y\times\C^*$ for $Y\geq X$. If they also dominate $\X$ for $(\X,\Ll)\in\FX_\Ll$, then
$$
\langle \Ll_{m,\CZ}^k\rangle =\langle\Ll_{m,\CZ'}^k \rangle,
$$
as classes in $N^k(\FX_{\PP^1}^{\C^*})$, for any $m>m_0$. Moreover the net $\{\langle \Ll_{m,\CZ}^k\rangle\}_{\CZ\geq \X}\subset N^k(\FX_{\PP^1}^{\C^*})$ is decreasing for any $k=1,\dots,n+1$ and for any $m>m_0$ fixed.
\end{lem}
\begin{proof}
Since the set of projective smooth varieties $\CZ$ that $\C^*$-invariantly dominates $\X$ and that satisfies $\CZ_{\pi^{-1}\C^*}\simeq Y\times\C^*$ is \emph{directed}, we can assume that $\CZ'$ dominates $\CZ$ without loss of generality.\newline
By definition, we then have $\CD_\CZ\geq \CD_{\CZ'}$ as pseudoeffective Weil classes in $N^1(\FX_{\PP^1}^{\C^*})$, which leads to $\Ll_{m,\CZ}\leq \Ll_{m,\CZ'}$ as big Cartier classes in $N^1(\FX_{\PP^1}^{\C^*})$. Hence $ \langle \Ll_{m,\CZ}^k \rangle\leq \langle \Ll_{m,\CZ'}^k\rangle $ in $N^k(\FX_{\PP^1}^{\C^*})$.\newline
Vice versa, by Theorem \ref{thm:PIP} it is enough to prove that any $\alpha\in CN^1(\FX_{\PP^1}^{\C^*})$ nef class such that $\alpha\leq \Ll_{m,\CZ'}$ satisfies $\alpha\leq \Ll_{m,\CZ}$. But, letting $f':\CZ'\to \CZ$, this follows from the fact that $f'^*\CD_\CZ-\CD_{\CZ'}$ is an effective $f'$-exceptional divisor and from $\Ll_{m,\CZ'}=f'^*\Ll_{m,\CZ}+\left(f'^*\CD_\CZ-\CD_{\CZ'}\right)$ in $N^1(\CZ')$. \newline
Next, to prove that $\{\langle\Ll_{m,\CZ}^k\rangle\}_{\CZ\geq \X}$ is decreasing it is sufficient to prove it for $\Y',\Y$ such that $\Y'\overset{\tilde{\rho}}{\geq}\Y$, $\Y'_{\pi^{-1}\C^*}\simeq Y'\times\C^*,\Y_{\pi^{-1}\C^*}\simeq Y\times\C^*,$ for $Y'\overset{\rho}{\geq}Y$. Denoting by $\CD_{Y,\Y'}$ the component-wise Zariski closure in $\Y'$ of $(\rho^*D_Y)\times\C^*$ with respect to the open embedding $Y'\times\C^*$ into $\Y'$, we then have
$$
\Ll_{m,\Y'}\leq \Ll_{m,\Y'}+\CD_{\Y'}-\CD_{Y,\Y'}=:\tilde{\Ll}_{m,\Y'}
$$
as $D_{Y'}\geq \rho^*D_Y$. In particular, 
$$
\langle \Ll_{m,\Y'}^k \rangle\leq \langle \tilde{\Ll}_{m,\Y'}^k \rangle
$$
for any $k=1,\dots,n$. Then, by construction we also have $\tilde{\Ll}_{m,\Y'}=\tilde{\rho}^*\Ll_{m,\Y}+\tilde{\rho}^*\CD_\Y-\CD_{Y,\Y'}$ in $N^1(\Y')$, and the equality $\langle \tilde{\Ll}_{m,\Y'}^k  \rangle=\langle \Ll_{m,\Y}^k \rangle$ again follows from Theorem \ref{thm:PIP} and the fact that $\tilde{\rho}^*\CD_\Y-\CD_{Y,\Y'}$ is an effective $\tilde{\rho}$-exceptional divisor.
\end{proof}
\begin{lem}
\label{lem:PIPWeil}
Let $k=1,\dots,n+1$ let $m\gg 1$ big enough. Then the net
$$
\{\langle \Ll_{m,\CZ}^k \rangle\}_{\CZ\geq X\times\PP^1}\subset N^k(\FX_{\PP^1}^{\C^*})
$$
of Weil psef classes has a unique weak accumulation point, which we will denote by $\langle\Ll_m^k\rangle$. Moreover, letting $(\X,\Ll)\in\FX_\Ll$, if $\CZ_{\pi^{-1}\C^*}\simeq Y\times \C^*$ for $Y\geq X$ and $\CZ\geq \X$ then
\begin{equation}
    \label{eqn:Inca}
    \langle \Ll_m^k\rangle_{\CZ}= \langle \Ll_{m,\Y}^k\rangle_{\CZ}
\end{equation}
as incarnations on $\CZ$ for any $\Y\geq \X$ such that $\Y_{\pi^{-1}\C^*}\simeq Y\times \C^*$.
\end{lem}
\begin{proof}
The first part immediately follows from the (weak) compactness of $\{\beta\in N^k(\FX_{\PP^1})\, : \, 0\leq \beta\leq \alpha\}$ for $\alpha\in N^k(\FX_{\PP^1})$ psef class \cite[Proposition 1.5]{BFJ09} and by the fact that the net $\{\langle\Ll_{m,\CZ}^k\rangle\}_{\CZ\geq \X}$ is decreasing by Lemma \ref{lem:Ammappate}. Moreover Lemma \ref{lem:Ammappate} also says that the positive intersection product $\langle \Ll_{m,\CZ}^k \rangle$ depends on the smooth variety $Y\geq X$ such that $\CZ_{\pi^{-1}\C^*}\simeq Y\times\C^*$ but it does not depend on the choice of $\CZ$. In particular (\ref{eqn:Inca}) follows and concludes the proof. 
\end{proof}
\begin{remark}
When $k=n+1$, $\langle \Ll_{m,\CZ}^{n+1} \rangle$ can be seen as a net of real numbers by integration over $\CZ$. 
Thus, by Lemma \ref{lem:Ammappate}
$\langle \Ll_m^{n+1}\rangle\in \R$ is given as decreasing limit of volumes of big line bundles (see also subsection \ref{ssec:Vol}) considering $\langle\Ll_{m,\Y}^{n+1}\rangle_{\Y}$ for $\Y$ graphs of the birational maps $\mu_\X^{-1}\circ(\rho_Y\times\Id):Y\times\PP^1\dashrightarrow \X$ where $(\X,\Ll)\in \FX_\Ll$. In the rest of the paper any top intersection product will be implicitly considered as a real number.
\end{remark}
An analogous to Lemma \ref{lem:PIPWeil} holds for the decreasing family $\{L_Y\}_{Y\geq X}=\{(L-\bD)_Y\}_{Y\geq X}$ of big Cartier classes in $N^1(\FX)$. Hence we will denote by $\langle (L-\bD)^k\rangle$ the associated weak accumulation points. Since $L-\bD$ is also nef, in this case the classes $\langle (L-\bD)^k\rangle$ coincide with that studied in \cite[Section 3]{DF21}.\newline

Recalling Definition \ref{defi:Pro}, we can now introduce the Donaldson-Futaki invariant.
\begin{defi}
\label{defi:DFRZ}
Let $\bD\in \bDiv_L(X)$ and let $\Ll\in N^1(\FX_{\PP^1}^{\C^*})$ be a test configuration class. Then for $m\gg 1$ sufficiently big we define
\begin{equation}
    \label{eqn:DFRZ}
    DF(\Ll;\bD):=V_\bD^{-1}\langle \Ll_m^n \rangle\cdot \mK_{\bD,\rela}+\overline{S}_{\bD}V_\bD^{-1}\frac{\langle \Ll_m^{n+1} \rangle}{n+1},
\end{equation}
where
\begin{itemize}
    \item[i)] $V_\bD:=\langle (L-\bD)^n \rangle$ is the \emph{volume} of the data $(X,L;\bD)$;
    \item[ii)] $\overline{S}_{\bD}:=nV_\bD^{-1}\langle(L-\bD)^{n-1}\rangle\cdot(-K_{\bD})$ is the log version of the mean scalar curvature with respect to the data $(X,L;\bD)$. Here $K_{\bD}\in N^1(\FX)$ is the Weil class defined by its incarnations $K_{(Y,B_Y)}\in N^1(Y)$, for $B_Y:=D_Y-K_{Y/X}$;
    \item[iii)] $\mK_{\bD,\rela}\in N^1(\FX_{\PP^1}^{\C^*})$ is the Weil class defined by its incarnations $K_{(\CZ,\CB_\CZ)/\PP^1}\in CN^1(\CZ)$ where $\CB_\CZ$ is the component-wise Zariski closure in $\CZ$, $\CZ_{\pi^{-1}\C^*}\simeq Y\times \C^*$ of the $\C^*$-invariant extension of $B_Y$.
\end{itemize}
We call it the \emph{Donaldson-Futaki invariant of $\Ll$ with respect to $\bD$}.
\end{defi}
We need to check that this is a good definition, i.e. that $DF(\Ll;\bD)$ does not depend on the choice of $m\gg 1$.
\begin{lem}
\label{lem:Below}
Let $\Ll$ be a $\C^*$-linearized big $\R$-line bundle on a smooth and projective variety $\Y$ such that there exists a $\C^*$-equivariant birational morphism $\mu:\Y\to Y\times \PP^1$ for $Y$ smooth projective variety. Assume also that $\Ll$ is numerically equivalent to $\mu^*p_1^*L+F$ for a big line bundle $L\to Y$ and for an effective $\R$-divisor $F$ supported in the central fiber. Then
\begin{equation}
    \label{eqn:RV}
    \vol_{\Y|\Y_1}(\Ll)=\vol_{Y}(L).
\end{equation}
and the map
$$
\R_{\geq 0}\ni m\to \vol_\Y(\Ll+m\Y_0)
$$
is affine, where $\Y_0:=\Y_{|\mu^{-1}(Y\times\{0\})}$. More precisely
$$
\vol_\Y(\Ll+m\Y_0)=\vol_\Y(\Ll)+m(n+1)\vol_Y(L).
$$
\end{lem}
\begin{proof}
To prove (\ref{eqn:RV}), by continuity in the big cone of the quantity involved, we can assume that $\Ll$ is a $\Q$-line bundle. Then, (\ref{eqn:RV}) easily follows from the definitions since any section $s\in H^0(Y,kL)$ extends to a global section $\tilde{s}\in H^0(\Y,k\Ll)$ thanks to the $\C^*$-action and to the fact that $\Ll-\mu^*p_1^*L$ is effective.\newline
Next, setting $\Ll_m:=\Ll+m\Y_0$ and fixing $m_1\geq 0$, by what recalled in subsection \ref{ssec:Vol} the map $t\to \vol_{\Y}(\Ll_{m_1}+t\Y_0)$ is differentiable at $t=0$ and
$$
\frac{d}{dt}_{|t=0}\vol_\Y(\Ll_{m_1}+t\Y_0)=(n+1)\langle\Ll_{m_1}^n \rangle\cdot \Y_0=(n+1)\langle \Ll_{m_1}^n \rangle\cdot \Y_1=(n+1)\vol_{\Y|\Y_1}(\Ll_{m_1}).
$$
The equality $\vol_Y(L)=\vol_{\Y|\Y_1}(\Ll_{m_1})$ given by (\ref{eqn:RV}) concludes the proof.
\end{proof}
\begin{prop}
\label{prop:TI}
Let $\bD,\Ll$ be as in Definition \ref{defi:DFRZ}. Then the right quantity in (\ref{eqn:DFRZ}) is constant for any $m>m_0$ where $m_0\in\Q$ is given by Lemma \ref{lem:Effe}.
\end{prop}
\begin{proof}
For any $m>m_0$, denote by
$$
\Ll_{m,\CZ}=\Ll+m(X\times\{0\})-\CD_\CZ
$$
the net of big Cartier classes on $\FX_{\PP^1}^{\C^*}$ which approximate $\Ll_m$ and whose positive intersection products $\langle \Ll_{m,\CZ}^k \rangle$ decreasingly approximates $\langle \Ll_m^k\rangle$ (Proposition \ref{prop:Big} and Lemma \ref{lem:PIPWeil}).\newline
Fixing $(\X,\Ll)\in \FX_\Ll$, to estimate $\{\langle \Ll_{m,\CZ}^{n+1} \rangle\}_{\CZ\geq X\times\PP^1}$ we can restrict to consider $\Y$ the graph of the birational map $\mu_\X^{-1}\circ (\rho_\Y\times\Id):Y\times\PP^1\dashrightarrow \X$ thanks to Lemmas \ref{lem:Ammappate}, \ref{lem:PIPWeil}. Namely, given the following commutative diagram
$$
\begin{tikzcd}
    \Y \ar[r, "\tilde{\rho}_Y"] \ar[d, "\mu_\Y"']& \X \ar[d, "\mu_\X"]\\
    Y\times\PP^1 \ar[r,"\rho_Y\times \Id"] & X\times \PP^1 
\end{tikzcd}
$$
$\langle \Ll_{m,\Y}^{n+1} \rangle$ is computed by the top positive intersection product of $ \Ll_{m,\Y}:=\tilde{\rho_Y}^*\Ll-\CD_\Y+m\Y_0 \in N^1(\Y)$. In particular, letting $m_1> m_2>m_0$, Lemmas \ref{lem:Effe}, \ref{lem:Below} imply
\begin{multline*}
    \langle \Ll_{m_1,\Y}^{n+1} \rangle=\langle \Ll_{m_2,\Y}^{n+1}\rangle+(m_1-m_2)(n+1)\vol_{\Y|\Y_1}(\Ll_{m_2,\Y}^n)=\langle \Ll_{m_2,\Y}^{n+1}\rangle+(m_1-m_2)(n+1)\langle L_Y^n\rangle
\end{multline*}
since $\Ll_{m_2,\Y}$ has a representative that is numerically equivalent to $\mu_\Y^*p_Y^*L_Y+G$ for $G$ effective $\R$-divisor supported on the central fiber. Therefore
\begin{equation}
    \label{eqn:1C}
    \overline{S}_{\bD}\frac{\langle \Ll_{m_1}^{n+1}\rangle}{n+1}=    \overline{S}_{\bD}\frac{\langle \Ll_{m_2}^{n+1}\rangle}{n+1}+(m_1-m_2)n\langle (L-\bD)^{n-1}\rangle \cdot (-K_{\bD})
\end{equation}
by passing to the weak limits.\newline
Next, let $\CZ\geq X\times\PP^1$ and let $Y\geq X$ such that $\CZ_{\pi^{-1}\C^*}\simeq Y\times \C^*$. We then have the following commutative diagram:
$$
\begin{tikzcd}
    \CZ' \ar[d, "f"] \ar[dr, "g"] & &\\
    \CZ \ar[dr, dashed, "\mu"] & \Y \ar[r, "\tilde{\rho}_Y"] \ar[d, "\mu_\Y"] & \X \ar[d, "\mu_\X"]\\
    & Y\times\PP^1 \ar[r, "\rho_Y\times\Id"] & X\times\PP^1. 
\end{tikzcd}
$$
where $\CZ'$ is given by the graph of $\CZ\dashrightarrow \Y$. Thus, by Lemma \ref{lem:PIPWeil} we get
\begin{multline*}
    \langle \Ll_{m_1}^n\rangle_\CZ\cdot K_{\bD,\rela,\CZ}=\langle \Ll_{m_1,\Y}^n\rangle_\CZ\cdot K_{(\CZ,\CB_\CZ)/\PP^1}
    =f^*\big(\langle \Ll_{m_1,\Y}^n\rangle_\CZ\big)\cdot K_{(\CZ',\CB_{\CZ'})/\PP^1}=\\
    =g^*\big(\langle \Ll_{m_1,\Y}^n\rangle_\Y\big)\cdot K_{(\CZ',\CB_{\CZ'})/\PP^1}= \langle \Ll_{m_1,\Y}^n\rangle_\Y \cdot K_{(\Y,\CB_\Y)/\PP^1}=\langle \Ll_{m_1,\Y}^n\rangle\cdot K_{(\Y,\CB_\Y)/\PP^1}
\end{multline*}
where we also used the fact that $K_{(\CZ',\CB_{\CZ'})/\PP^1}-f^*K_{(\CZ,\CB_\CZ)/\PP^1}$ is $f$-exceptional. Moreover, again by Lemma \ref{lem:Below} and the results recalled in subsection \ref{ssec:Vol}, it follows that
\begin{multline*}
    (n+1)\langle \Ll_{m_1,\Y}^n\rangle \cdot K_{(\Y,\CB_\Y)/\PP^1}=\frac{d}{dt}_{|t=0}\vol_{\Y}\left(\Ll_{m_2,\Y}+tK_{(\Y,\CB_\Y)/\PP^1}+(m_1-m_2)\Y_0\right)=\\
    =\frac{d}{dt}_{|t=0}\left(\vol_\Y(\Ll_{m_2,\Y}+tK_{(\Y,\CB_\Y)/\PP^1})+(m_1-m_2)(n+1)\vol_Y(L_Y+tK_{(Y,B_Y)})\right)=\\
    =(n+1)\langle \Ll_{m_2,\Y}^n\rangle \cdot K_{(\Y,\CB_\Y)/\PP^1}+(m_1-m_2)(n+1)\frac{d}{dt}_{|t=0}\vol_{Y}(L_Y+tK_{(Y,B_Y)})=\\
    =(n+1)\langle \Ll_{m_2,\Y}^n\rangle \cdot K_{(\Y,\CB_\Y)/\PP^1}+(m_1-m_2)(n+1)n\langle L_Y^{n-1}\rangle \cdot K_{(Y,B_Y)}.
\end{multline*}
Summarizing, we have
\begin{multline}
    \label{eqn:2C}
    \langle \Ll_{m_1}^n \rangle\cdot K_{\bD,\rela}=\sup_{\CZ'}\inf_{\CZ\geq \CZ'} \langle \Ll_{m_1}^n\rangle_{\CZ}\cdot  K_{(\CZ,\CB_\CZ)/\PP^1}=\sup_{Y'}\inf_{Y\geq Y'}\langle \Ll_{m_1,\Y}^n \rangle\cdot K_{(\Y,\CB_\Y)/\PP^1}=\\
    =\sup_{Y'}\inf_{Y\geq Y'}\Big(\langle \Ll_{m_2,\Y}^n \rangle\cdot K_{(\Y,\CB_\Y)/\PP^1}+n(m_1-m_2)\langle L_Y^{n-1}\rangle\cdot K_{(Y,B_Y)}\Big).
\end{multline}
Next, since $K_{(Y,B_Y)}=D_Y+\rho_Y^*K_X$, if $\rho_Y:Y\to X$ factorises through $\rho_{Y,\tilde{Y}}:Y\to \tilde{Y}$ and $\rho_{\tilde{Y}}:\tilde{Y}\to X$ then
$$
\langle L_Y^{n-1}\rangle \cdot K_{(Y,B_Y)}\geq \langle L_Y^{n-1}\rangle \cdot\rho_{Y,\tilde{Y}}^*(D_{\tilde{Y}}+\rho_{\tilde{Y}}^*K_X)=\langle L_Y^{n-1}\rangle\cdot K_{(\tilde{Y},\tilde{B}_Y)}
$$
where the first inequality follows from the monotonicity $\rho_{Y,\tilde{Y}}^*D_{\tilde{Y}}\leq D_Y$.\newline
Namely, the extra term appearing on the right hand side in (\ref{eqn:2C}) is monotone in $Y\geq X$ and it converges to $\langle (L-\bD)^{n-1}\rangle \cdot K_{\bD}$. Therefore, combining (\ref{eqn:2C}) with Lemma \ref{lem:Convergence} below, we obtain
\begin{equation}
    \label{eqn:3C}
    \langle \Ll_{m_1}^n \rangle\cdot K_{\bD,\rela}=\langle \Ll_{m_2}^n\rangle\cdot K_{\bD,\rela}+ +n(m_1-m_2)\langle (L-\bD)^{n-1}\rangle \cdot K_{\bD},
\end{equation}
which together with (\ref{eqn:1C}) concludes the proof.
\end{proof}
\begin{lem}
\label{lem:Convergence}
Let $(I,\leq)$ be a directed set and let $\{x_j\}_{j\in I},\{y_i\}_{i\in I}\subset \R$ be two convergent nets respectively to $x,y\in\R$. If $\{y_i\}_{i\in I}$ converges monotonically then
\begin{itemize}
    \item[i)] $\sup_{j\in I}\inf_{i\geq j}x_iy_i=y\sup_{j\in I}\inf_{i\geq j}x_i$;
    \item[ii)] $\sup_{j\in I}\inf_{i\geq j}(x_i+y_i)=y+\sup_{j\in I}\inf_{i\geq j}x_i$.
\end{itemize}
\end{lem}
\begin{proof}
This is an easy exercise left to the reader.
\end{proof}
Proposition \ref{prop:TI} allows us to reformulate the properties of the Donaldson-Futaki invariant in the following more classical way.
\begin{corollary}
\label{cor:DF}
Let $\bD\in \bDiv_L(X)$. Then, to any normal test configuration $(\X,\Ll)$ for $(X,L)$ in the sense of Definition \ref{defi:TC} is associated a weight, the Donaldson-Futaki invariant with respect to $\bD$
$$
(\X,\Ll)\longrightarrow DF(\X,\Ll;\bD),
$$
which is preserved under pullback and which is translation invariant, i.e. $DF(\X,\Ll+c\X_0;\bD)=DF(\X,\Ll;\bD)$ for any $c\in\Q$.
\end{corollary}
\begin{proof}
It immediately follows from Proposition \ref{prop:TI} setting
\begin{equation}
    \label{eqn:DFTC}
    DF(\X,\Ll;\bD):=DF(\Ll;\bD)
\end{equation}
where on the right hand side $\Ll\in N^1(\FX_{\PP^1}^{\C^*})$ is the test configuration class associated to $\Ll$.
\end{proof}
\begin{remark}
Similarly to the absolute setting the Donaldson-Futaki invariant depends on the generalized $b$-divisors $\bB:=\{B_Y\}_{Y\geq X}$, $\bD$ on $X$ and not only on their numerical classes in $N^1(\FX)$ (see also Remark \ref{rem:NI}). Therefore given two numerically equivalent generalized $b$-divisors $\bD,\bD'$ on $X$, there may exist ample test configuration classes $\Ll$ such that
$$
DF(\Ll,\bD)\neq DF(\Ll,\bD').
$$
\end{remark}
\subsection{(Uniform) $\bD$-log $K$-stability}
In this subsection we will define the notion of $\bD$-log $K$-stability and its \emph{uniform} version.\newline

We recall that a test configuration $(\X,\Ll)$ for $(X,L)$ (Definition \ref{defi:TC}) is said to be \emph{trivial} if $\X_{\pi^{-1}\C}\simeq X\times\C$, i.e. if it is a product \cite[Definition 2.9]{BHJ17}. Note that there are examples where $\X$ is trivial but its compactification $\overline{\X}$ is not a product \cite[Example 2.8]{BHJ17}. However, the compactification $\overline{\X}$ of any test configuration satisfies
$$
\overline{\X}_{\pi^{-1}(\PP^1\setminus\{0\})}\overset{\C^*}{\simeq}X\times(\PP^1\setminus\{0\}). 
$$
In particular, \cite[Lemma 2.10]{BHJ17} implies that a test configuration class $\Ll\in N^1(\FX_{\PP^1}^{\C^*})$ is associated to a trivial normal test configuration $(\X,\Ll)$ if and only if
$$
\Ll=L\times\PP^1-m_0(X\times\{0\})
$$
for $m_0\in\Q$. Observe that in this case, for any $m>m_0$, the equality
$$
\big\langle \left(\Ll+m(X\times\{0\})\right)^{n+1} \big\rangle=(n+1)\left\langle (L\times\PP^1)^n,\left(\Ll+m (X\times\{0\})\right)\right\rangle,
$$
holds.
\begin{defi}
\label{defi:Triviality}
Let $\bD\in \bDiv_L(X)$ and let $\Ll\in N^1(\FX_{\PP^1}^{\C^*})$ be a test configuration. We say that $\Ll$ is \emph{$\bD$-trivial} if
$$
\langle\Ll_m^{n+1}\rangle=(n+1)\langle \big((L-\bD)\times\PP^1\big)^n,\Ll_m\rangle
$$
for a(ny) $m>m_0:=\inf\{m\in \Q\, : \, \Ll+m(X\times\{0\})-L\times\PP^1\in N^1(\FX_{\PP^1}^{\C^*})\, \mathrm{is}\, \mathrm{pseudoeffective}\}$, where as usual we set $\Ll_m:=\Ll+m(X\times\{0\})-\CD$.
\end{defi}
By Lemma \ref{lem:Effe} the pseudoeffective threshold $m_0$ is a finite quantity and by Proposition \ref{prop:Big} $\Ll_m\in N^1(\FX_{\PP^1}^{\C^*})$ is pseudoeffective. Moreover, letting $(\X,\Ll)\in\FX_{\Ll}$, proceeding as in Lemmas \ref{lem:Ammappate}, \ref{lem:PIPWeil} it is not hard to check that
$$
\left\{\big\langle \big((L-D_Y)\times\PP^1\big)^n, \Ll_{m,\CZ} \big\rangle\right\}_{\CZ\geq \X, \CZ_{\pi^{-1}\C^*}\simeq Y\times\C^*}
$$
is constant for any $Y\geq X$ fixed and it is a decreasing net moving $Y\geq X$. In particular the quantity $\langle \big((L-\bD)\times\PP^1\big)^n,\Ll_m\rangle$ is given as the unique weak accumulation point.
\begin{remark}
There are plenty of examples of non trivial test configuration classes which are $\bD$-trivial. Indeed $\bD$-triviality can be characterized in test of \emph{psh test curves} as proved in Proposition \ref{prop:Dtriviality} in section \ref{ssec:TestCurves}.
\end{remark}
To quantify how much a test configuration class is \emph{far} from being $\bD$-trivial, we introduce the following quantity.
\begin{defi}
\label{defi:JandE}
Let $\bD\in \bDiv_L(X)$ and let $\Ll\in N^1(\FX_{\PP^1}^{\C^*})$ be a test configuration class. Then, recalling that $V_\bD:=\langle (L-\bD)^n \rangle$,
$$
J^{\NA}(\Ll;\bD):=V_\bD^{-1}\langle \big((L-\bD)\times\PP^1\big)^n, \Ll_m\rangle-E^{\NA}(\Ll+m(X\times\{0\});\bD)
$$
for any $m>m_0$ where
$$
E^{\NA}(\Ll+m(X\times\{0\});\bD):=\frac{\langle \Ll_m^{n+1} \rangle}{(n+1)V_D}.
$$
$J^{\NA}(\Ll;\bD)$ will be called the \emph{$\bD$-log Non-Archimedean $L^1$-norm of $\Ll$}, while $E^{\NA}(\Ll+m(X\times\{0\});\bD)$ will be called the \emph{$\bD$-log Non-Archimedean Monge-Ampère energy of $\Ll+m(X\times\{0\})$}.
\end{defi}
As seen during the proof of Proposition \ref{prop:TI},
$$
E^\NA\big(\Ll+m_1(X\times\{0\});\bD\big)=E^\NA\big(\Ll+m_2(X\times\{0\});\bD\big)+(m_1-m_2)
$$
for any $m_1\geq m_2>m_0$. Similarly one can prove that
$$
\langle \big((L-\bD)\times\PP^1\big)^n, \Ll_{m_1} \rangle=\langle \big((L-\bD)\times\PP^1\big)^n, \Ll_{m_2} \rangle+ (m_1-m_2)V_\bD
$$
for any $m_1\geq m_2>m_0$, i.e. $J^\NA(\Ll;\bD)$ is well-defined.
\begin{corollary}
\label{cor:J}
Let $\bD\in \bDiv_L(X)$ and let $\Ll\in N^1(\FX_{\PP^1}^{\C^*})$ be a test configuration class. Then $J^{\NA}(\Ll;\bD)$ is well-defined, it is non-negative and it vanishes if and only if $\Ll$ is $\bD$-trivial.
\end{corollary}
\begin{proof}
The unique property left to prove is $J^\NA(\Ll;\bD)\geq 0$. We do not show it here as it will be a consequence of Corollary \ref{cor:SlopeEJ}.
\end{proof}
We can now give the following key definition.
\begin{defi}
\label{defi:K-stability}
Let $D\in \bDiv_L(X)$. Then $(X,L)$ is said to be \emph{$\bD$-log K-stable} if
$$
DF(\Ll;\bD)\geq 0
$$
for any ample test configuration class $\Ll\in N^1(\FX_{\PP^1}^{\C^*})$ with equality if and only if $\Ll$ is $\bD$-trivial. Moreover $(X,L)$ is said to be \emph{uniformly $\bD$-log K-stable} if there exists $\delta>0$ such that
$$
DF(\Ll;\bD)\geq \delta J^{\NA}(\Ll;\bD)
$$
for any ample test configuration class $\Ll\in N^1(\FX_{\PP^1}^{\C^*})$.
\end{defi}
\begin{remark}
\label{rem:K-st}
In more classical terms, $(X,L)$ is $\bD$-log K-stable if
$$
DF(\X,\Ll;\bD)\geq 0
$$
for any normal and ample test configuration $(\X,\Ll)$ for $(X,L)$ with equality if and only if $\X$ is $\bD$-trivial, where the Donaldson-Futaki invariant is defined as in Corollary \ref{cor:DF}. Similarly for the uniform $\bD$-log K-stability.
\end{remark}
The next result immediately follows from the definitions.
\begin{prop}
Let $\bD\in \bDiv_L(X)$. If $(X,L)$ is uniformly $\bD$-log K-stable then $(X,L)$ is $\bD$-log K-stable. 
\end{prop}
\subsection{Cartier $b$-divisors}
\label{ssec:Cartier}
Let $(\Y,\Ll_\Y)$ be the compactification of test configuration for $(Y,L_Y:=\rho_Y^*L-D_Y)$ where we assume $D_Y$ to be a $\Q$-divisor, and let $B_Y$ a boundary on $Y$. If $\Ll_\Y$ is big then one can modify the usual definition of the log Donaldson-Futaki invariant (see subsection \ref{sssec:DF}) in the following way:
\begin{equation}
    \label{eqn:DFBig}
    DF_{B_Y}(\Y,\Ll_\Y):=V^{-1}\langle \Ll_\Y^n \rangle \cdot K_{(\Y,\CB_\Y)/\PP^1}+\overline{S}_{B_Y} V^{-1}\frac{\langle\Ll_\Y^{n+1} \rangle}{n+1}
\end{equation}
where $V:=\langle L_Y^n \rangle$,  $\overline{S}_{B_Y}:=-nV^{-1}\langle L_Y^{n-1}\rangle\cdot K_{(Y,B_Y)}$ and $\CB_\Y$ is the divisor given as the compontentwise Zariski closure in $\Y$ of the $\C^*$-invariant divisor $B_Y\times\C^*$ (see also \cite[Theorem 1.1]{Li21}). When $\Ll_\Y$ is nef then this definition coincides with the usual log Donaldson-Futaki invariant.
\begin{prop}
\label{prop:KstClass}
Let $\bD\in \bDiv_L(X)$ such that $D_Y$ is a $\Q$-divisor for any $Y\geq X$, $\Ll\in N^1(\FX_{\PP^1}^{\C^*})$ be test configuration class, $(\X,\Ll)\in \FX_\Ll$, and consider for any $Y\geq X$ the commutative diagram
$$
\begin{tikzcd}
    \Y \ar[r, "\tilde{\rho}_Y"] \ar[d, "\mu_\Y"']& \X \ar[d, "\mu_\X"]\\
    Y\times\PP^1 \ar[r,"\rho_Y\times\Id"] & X\times \PP^1
\end{tikzcd}
$$
where $\Y$ is the graph of $\mu_\X^{-1}\circ (\rho_Y\times\Id):Y\times\PP^1\dashrightarrow \X$. Then for any $m\gg 1$ big enough 
\begin{equation}
    DF(\X,\Ll;\bD)=\sup_{Y'}\inf_{Y\geq Y'} DF_{B_Y}(\Y,\Ll_{m,\Y})
\end{equation}
where
$B_Y:=D_Y-K_{Y/X}$, $(\Y,\Ll_{m,\Y})$ is the big test configuration for $(Y,L_Y:=\rho_Y^*L-D_Y)$ given by $\Ll_{m,\Y}:=\tilde{\rho}_Y^*\Ll+m\Y_0-\CD_\Y$ and where the Donaldson-Futaki invariant on the right hand side is computed as in (\ref{eqn:DFBig}). In particular if $\bD=D_Y$ for $D_Y$ $\Q$-divisor on $Y$ then
$$
DF(\X,\Ll;\bD)=DF_{B_Y}(\Y,\Ll_{m,\Y})
$$
for any $m\gg 1$ big enough.
\end{prop}
\begin{proof}
During the proof of Proposition \ref{prop:TI} we already observed that, for any $m>m_0$ ($m_0$ given by Lemma \ref{lem:Effe})
$$
\langle  \Ll_m^n\rangle\cdot K_{\bD,\rela}=\sup_{Y'}\inf_{Y\geq Y'}\langle \Ll_{m,\Y}^n \rangle\cdot K_{(\Y,\CB_\Y)/\PP^1}.
$$
Moreover by Lemma \ref{lem:PIPWeil} the net $\{\langle \Ll_{m,\Y}^{n+1}\rangle\}_{Y\geq X}$ decreases to $\langle \Ll_m^{n+1}\rangle$, $V_{D_Y}:=\langle L_Y^n \rangle$ decreases to $V_{\bD}$, while
$$
V_{D_Y}\overline{S}_{D_Y}=n\langle L_Y^n \rangle\cdot (-K_{(Y,B_Y)})\searrow V_{\bD}\overline{S}_{\bD}
$$
as $K_{(Y,B_Y)}=\rho_Y^*K_X+D_Y$ and $\{D_Y\}_{Y\geq X}$ is increasing.
Thus, using repeatedly Lemma \ref{lem:Convergence}, it is easy to obtain
\begin{multline*}
    DF(\X,\Ll;\bD)=\sup_{Y'}\inf_{Y\geq Y'}\Bigg(V_{D_Y}^{-1}\langle \Ll_{m,\Y}^n \rangle\cdot K_{(\Y,\CB_\Y)/\PP^1}+V_{D_Y}^{-1}\overline{S}_{D_Y}\frac{\langle \Ll_{m,\Y}^{n+1} \rangle}{n+1}\Bigg)=\sup_{Y'}\inf_{Y\geq Y'}DF_{B_Y}(\Y,\Ll_{m,\Y}),
\end{multline*}
which conclude the proof.
\end{proof}
\begin{remark}
\label{rem:KstClass}
When $D_Y$ is not a $\Q$-divisor, or when $D_Y$ is not even a $\R$-divisor, i.e. when $\bD$ is just a generalized $b$-divisor, an natural adaptation of Proposition \ref{prop:KstClass} still holds. Indeed, one can consider $(\Y,\{\Ll_\Y\})$ the \emph{cohomological test configuration} associated to $\{L_Y\}$ in the sense of \cite[Definition 3.2.5]{Sjo17} (independently introduced in \cite{DR17}), and define the log Donaldson-Futaki invariant in the same way as in (\ref{eqn:DFBig}) taking the cohomology classes associated to $B_Y, \CB_\Y$. Then the proof of Proposition \ref{prop:KstClass} remains invariant, interpreting the $\bD$-log $K$-stability as an \emph{asymptotic log $K$-stability}.
\end{remark}

\section{$\bD$-log Ding-stability}
\label{sec:Ding-StabilitySection}
In this section, assuming $X$ to be smooth projective variety and $L\to X$ be an ample line bundle, we will define a natural \emph{$\bD$-log Ding-stability} notion and we will compare it with the $\bD$-log K-stability defined in the previous section in the Fano case.\newline

Let $\omega$ be a K\"ahler form in $c_1(L)$. The following result connects the class $\M_D^+:=\M_D^+(X,\omega)$ to $L$-positive generalized $b$-divisors.
\begin{prop}
\label{prop:Corre}
There is a one-to-one correspondence
$$
\M_D^+\overset{1-1}{\longleftrightarrow} \bDiv_L(X).
$$
More precisely, to any $\psi\in \M_D^+$ is associated $\bD=\{D_Y\}_{Y\geq X}\in \bDiv_L(X)$ where, for $Y\overset{\rho_Y}{\geq}X$,
$$
D_Y:=\sum_{F\subset Y\, \mathrm{prime}\, \mathrm{divisor}}\nu(\psi\circ \rho_Y,F_Y)F_Y
$$
is the divisorial part in the Siu decomposition of $\rho_Y^*(\omega_\psi)$. 
\end{prop}
We use the notation $\omega_\psi:=\omega+dd^c \psi$.
\begin{proof}
Let $\psi\in \M_D^+$.\newline
As suggested, for any $Y\geq X$ we define $D_Y$ as the divisorial part in the Siu decomposition of the closed and positive $(1,1)$-current $\rho_Y^*(\omega_\psi)$ \cite{Siu74}. Namely $D_Y=\sum_F \nu(\psi\circ \rho_Y, F)F $ where the sum runs over all prime divisors $F\subset Y$. As $\{y\in Y\, :\, \nu(\psi\circ \rho_Y,y)\geq a\}$ is analytic for any $a>0$, the sum involved in the definition of $D_Y$ is at most countable. If $Y'\overset{\rho_{Y',Y}}{\geq}Y$, then by construction it is not hard to check that $\rho_{Y',Y}^*D_Y\leq D_{Y'}$ and that $\rho_{Y',Y,*}D_{Y'}=D_Y$.
Moreover, letting $R$ be the closed and positive $(1,1)$-current such that $\rho_Y^*(\omega_\psi)=[D_Y]+R$, we have
$$
\vol_Y(L_Y)\geq \int_Y R^n=\int_X MA_\omega(\psi)>0,
$$
where the first inequality follows by monotonicity of the non-pluripolar product \cite[Theorem 1.2]{WN17}, while the equality is because the non-pluripolar product does not charge pluripolar sets. Thus to prove that $\bD=\{D_Y\}_{Y\geq X}\in \bDiv_L(X)$ it remains to show that $L-\bD$ is nef. But, this follows from \cite[Proposition 2.4]{Bou04} and \cite[Lemma 2.10]{BdFF12} as by construction $\rho_Y^*(\omega_\psi)-[D_Y]$ is a closed and positive current with zero Lelong numbers along any prime divisor on $Y$. Summarizing,
the map
\begin{equation}
    \label{eqn:Map}
    \Phi:\cM^+_D\to \bDiv_L(X)
\end{equation}
is well-defined and, with obvious notations, $\bD_\psi\geq \bD_{\psi'}$ if $\psi\leq \psi'$. Moreover, if $\bD_\psi=\bD_{\psi'}$ then by the resolution of the Strong Openness Conjecture \cite{GZ14} and \cite[Theorem A]{BFJ08} we obtain $\mathcal{I}(t\psi)=\mathcal{I}(t\psi')$ for any $t>0$. Hence $\psi=\psi'$ (see also Remark \ref{rem:Imodel}), and the map $\Phi$ is injective.\newline
It then remains to prove the surjectivity. Let $\bD\in \bDiv_L(X)$. For any $Y\overset{\rho_Y}{\geq} X$, $L_Y:=\rho_Y^*L-D_Y$ is big. Thus there exists a sequence of proper birational maps $\mu_{Y,k}:Y_k\to Y$ such that $\mu_{Y,k}^*\{L-D_Y\}=\{E_{Y,k}\}+\beta_{Y,k}$ and such that $\vol_Y(L-D_Y)-\frac{1}{k}\leq \vol_{Y_k}(\beta_{Y,k})\leq \vol_Y(L-D_Y)$ where $E_{Y,k}$ is an effective $\R$-divisor and $\beta_{Y,k}$ is a big and nef class. Moreover it is possible to construct such approximation requiring that $\beta_{Y,k}\leq \beta_{Y,k+1}$ as nef classes in $CN^1(\FX)$ (see for instance \cite[Section 3]{BDPP13}). As the fibers of $\mu_{Y,k}$ are connected, we construct a model type envelope $\psi_{Y,k}$ such that
$$
\mu_{Y,k}^*\rho_Y^* (\omega_{\psi_{Y,k}})=S_{Y,k}+[E_{Y,k}]+[\mu_{Y,k}^*D_Y]
$$
for $S_{Y,k}$ current with minimal singularities in $\beta_{Y,k}$. As $\beta_{Y,k}$ is big and nef, an easy perturbation argument gives $\psi_{Y,k}\in \cM^+_D$ and $\int_X MA_{\omega}(\psi_{Y,k})=\vol_{Y_k}(\beta_{Y,k})$. The monotonicity of $\{\beta_{Y,k}\}_{k\in\N}$ implies that $\{\psi_{Y,k}\}_{k\in\N}$ is an increasing sequence in $\cM^+_D$. By \cite[Lemma 2.21]{DX20} we then obtain that the weak limit $\psi_Y:=\lim_{k\to +\infty}\psi_{Y,k}$ belongs to $\cM^+_D$ and by \cite[Lemma 3.12]{Tru20a} we also get $\int_X MA_\omega(\psi_Y)=\vol_Y(L-D_Y)$. In particular,
$T_Y:=\rho_Y^*(\omega_{\psi_Y})-[D_Y]$ is a current with full Monge-Ampère mass. Therefore it has zero Lelong number along any prime divisor combining \cite[Theorem 1.1]{DDNL17a} with the fact that $L-D_Y$ is big and nef in codimension $1$ (see also \cite[Propositions 3.2, 3.6]{Bou04}). Thus
$$
\rho_Y^*(\omega_{\psi,Y})=T_Y+[D_Y]
$$
coincides with the Siu decomposition of $\rho_Y^*(\omega_{\psi,Y})$.

Next, as $\{D_Y\}_{Y\geq X}$ is monotone, by construction we deduce that the associated family $\{\psi_Y\}_{Y\geq X}\subset \cM^+_D$ is ordered, i.e. $\psi_Y\geq \psi_{Y'}$ if $Y\leq Y'$.
Thus Lemma \ref{lem:Inf} below implies that
$$
\psi:=\inf_{Y\geq X}\psi_Y\in \cM_D^+
$$
and that there exists an increasing sequence $\{Y_k\}_{k\in\N}$ such that $\psi_{Y_k}\searrow \psi$. In particular, letting $k\to +\infty$, we have convergence of Lelong numbers on any $Z\geq X$ (subsection \ref{ssec:KE}). As by construction, for any $Y'\geq Y$, $T_{Y,Y'}:= \rho_Y^*(\omega_{\psi_{Y'}})-[D_Y]$ has zero Lelong numbers along any prime divisor over $Y$, it is then easy to deduce that $S_Y:=\rho_Y^*(\omega_\psi)-[D_Y]$ coincides with the positive part in the Siu decomposition of $\rho_Y^*(\omega_\psi)$, i.e. $\bD=\Phi(\psi)$.

\end{proof}
By Choquet's Lemma (\cite[I.4.23]{DemNotes}) the upper envelope of any family of $\omega$-psh functions can be reduced to an upper envelope of a countable subfamily. This is the key point in proving that the upper semicontinuous regularization of any relatively compact family of $\omega$-psh functions is $\omega$-psh. The following Lemma shows an analog for the infimum of $\omega$-psh functions under the condition that the family is \emph{directed}.
\begin{lemma}
\label{lem:Inf}
Let $(I,\leq)$ be a directed poset, and let $\{\psi_i\}_i\in I$ be a relatively compact family of $\omega$-psh functions such that $\psi_i\geq \psi_j$ if $i\leq j$. Then there exists a sequence $\{\psi_{i_k}\}_{k\in\N}$ such that
$$
\psi_{i_k}\searrow \psi:=\inf_{i\in I}\psi_i\in \PSH(X,\omega).
$$
\end{lemma}
\begin{proof}
By relatively compactness there exist constants $A,B\in \R_{>0}$ such that $\sup_X \psi_i \in (-B,A)$ for any $i\in I $. Without loss of generality we will assume $A=0$. Then, for $dV$ fixed volume form on $X$ with $dV(X)=1$,
$$
\lVert \psi_i \rVert_{L^1}=\int_X(-\psi_i)dV\leq B<\infty
$$
for any $i$. In particular there exists $\{\psi_{i_k}\}_{k\in\N}$ such that
$$
\lVert \psi_{i_k}\rVert_{L^1}\nearrow \sup_{i\in I}\lVert \psi_i \rVert_{L^1}.
$$
Moreover, since $\psi_i\geq \psi_j$ is $i\leq j$ and $I$ is directed, we can construct $\{\psi_{i_k}\}_{k\in \N}$ such that $i_k\leq i_{k+1}$ for any $k\in \N$. Setting $\tilde{\psi}:=\lim_{k\to+\infty}\psi_{i_k}\in \PSH(X,\omega)$, we clearly have
$$
\tilde{\psi}\geq \psi=\inf_{i\in I}\psi_i.
$$
On the other hand for fixed $j\in I$, considering $j_k:=\max(i_k,j)$, we have $ \psi_j\geq \psi_{j_k}, \psi_{i_k}\geq \psi_{j_k} $ for any $k\in\N$, and
$$
\lVert \psi_{i_k}-\psi_{j_k} \rVert_{L^1}=\lVert \psi_{j_k}\rVert_{L^1}-\lVert \psi_{i_k}\rVert_{L^1}\longrightarrow \lVert \lim_{k\to +\infty} \psi_{j_k} \rVert_{L^1}-\sup_{i\in I}\lVert \psi_i \rVert_{L^1}\leq 0,
$$
i.e. $\psi_{j_k}\searrow \tilde{\psi}$. Hence $\psi_j\geq \tilde{\psi}$ for any $j\in I$, which clearly gives $\psi\geq \tilde{\psi}$ and concludes the proof.
\end{proof}
We will write $\bD_\psi\in \bDiv_L(X)$ (resp. $\psi_\bD\in \M_D^+$) when we want to emphasize that the generalized $b$-divisor on $X$ (resp. the model type envelope) is associated to $\psi\in\M_D^+$ (resp. to $\bD\in \bDiv_L(X)$) thanks to the correspondence given by Proposition \ref{prop:Corre}.
\subsection{Psh rays}
\label{ssec:PSHRays}
We need to recall the definition of psh paths and psh rays.\newline
For simplify the notations in this subsection we will set $\PSH:=\PSH(X,\omega)$.
\begin{defi}
Let $I\subset \R_{>0}$ be an open interval and let $\D_I:=\{\tau\in\C^*\, | \, -\log \lvert \tau \rvert\in I\}$. A map $U:I\to \PSH$ (also denoted by $\{u_t\}_{t\in I}$) is said to be a \emph{psh path} if $v(x,\tau):=u_{-\log\lvert \tau\rvert}(x)\in \PSH(X\times\D_I,p_1^*\omega)$.\newline
When $I=(a,b)$ is bounded, a psh path $U:(a,b)\to \PSH$ is said to be dominated by $u_a, u_b$ if $\lim_{t\to a^+}u_t\leq u_a, \lim_{t\to b^-}u_t\leq u_b$. In this case $U$ is said to be the \emph{psh geodesic path joining $u_a,u_b$} if any other psh path $V:(a,b)\to \PSH$ dominated by $u_a,u_b$ satisfies $V\leq U$.\newline
When $I=(a,+\infty)$ is unbounded, a psh path $U:\to \PSH$ is also called \emph{psh ray} and it is said to be a \emph{psh geodesic ray} if for any $a<b<c$ the restriction of $U$ to $(b,c)$ is a psh geodesic path joining $u_b$ and $u_c$.
\end{defi}
When $\lim_{t\to a^+}u_t=u_a$, we will also write $U:[a,b)\to \PSH$ and similarly for other intervals.

It is well-known that if there exists a psh path $U:(a,b)\to \PSH$ dominated by $u_a,u_b$ then it is possible to construct the geodesic path joining $u_a,u_b$ by a Perron-Bremermann envelope argument.
\begin{defi}
A psh ray $U:\R_{\geq 0}\to \PSH$ is said to be a \emph{$[\psi]$-relative psh ray} if $u_t\preccurlyeq \psi$ for any $t\in \R_{\geq 0}$.
\end{defi}
Similarly to \cite{BBJ15}, we will say that a $[\psi]$-relative psh ray $U:\R_{\geq 0}\to \PSH$ has \emph{linear growth} if $\sup_X u_t=O(t)$, i.e iff
$$
\lambda_{\max}:=\lim_{t\to+\infty} \frac{\sup_X u_t}{t}\in \R
$$
since $t\to \sup_X u_t$ is a convex function. In particular, $u_t-at$ is bounded above as $t\to +\infty$ for any $a\geq \lambda_{\max}$, i.e. the associated $S^1$-invariant function $v_a\in PSH(X\times\D^* , p_1^*\omega)$,
$$
v_a(x,\tau)=u_{-\log|\tau|}(x)+a\log|\tau|
$$
extends to a $p_1^*\omega$-psh function over $X\times\D$. We thus define $U_{\NA}:X^{\divv}\to \R$ as
$$
U_{\NA}(\nu):=-G(\nu)(v_a)+aG(\nu)(\tau)
$$
where $G(\nu)\in (X\times \C)^{\divv}$ is the Gauss extension of $\nu$ (see subsection \ref{ssec:NA}). 
Clearly $G(\nu)(v_a)$ is recovered as the generic Lelong number of the extension of $v_a$ on a suitable blow-up chosen so that $G(\nu)$ reads as vanishing orders along a divisor. Note that by linearity of the Lelong numbers, the definition of $U_{\NA}$ does not depend on the constant $a\geq \lambda_{\max}$. \newline

The main difference with respect to the absolute setting, i.e. when $\psi=0$, is that the singularities of $U$ are not enclosed in the central fiber $X\times \{0\}$. Indeed since $\sup_X u_t=\sup_X (u_t-\psi)$ (subsection \ref{ssec:KE}), we find out that $v_a(x,\tau)\leq \psi(x)+(a-\lambda_{\max})\log|\tau| $. Thus, by construction
\begin{equation}
    \label{eqn:Trivial}
    U_{\NA}\leq \Psi_{\NA}+\lambda_{\max}
\end{equation}
where $\Psi_{\NA}:X^{\divv}\to \R$ is the function $\Psi_{\NA}(\nu):=-\nu(\psi)$ which is associated to the $[\psi]$-relative trivial psh ray $\Psi: [0,+\infty)\ni t\to \psi$. Note that from (\ref{eqn:Trivial}) we easily deduce that
$$
\sup_{X^{\divv}}U_{\NA}=\lambda_{\max}=U_{\NA}(\nu_{\triv})
$$
where we recall that $G(\nu_{\triv})=\ord_{X\times\{0\}}$ (see also \cite[Lemma 4.3]{BBJ15}).
\subsubsection{Induced psh metrics on test configurations}
Similarly to \cite{BBJ15}, we give the following definition.
\begin{defi}
We say that a $[\psi]$-relative psh ray $U:\R_{\geq 0}\to \PSH$ \emph{induces a psh metric} on a test configuration class $\Ll\in N^1(\FX_{\PP^1}^{\C^*})$ if for a(ny) normal test configuration $(\X,\Ll)$ associated to the class $\Ll$ the $S^1$-invariant metric $e^{-2U}p_1^*h$ ($h$ metric defined by $\omega$) on $X\times\D^*$ extends to a psh metric on $(\X,\Ll)_{|\D}$.
\end{defi}
Observe that the metric $e^{-2U}p_1^*h$ extends to a psh metric on $(\X,\Ll)_{|\D}$ for a normal test configuration iff it extends to $(\X',\Ll')_{|\D}$ for an higher normal test configuration when $\Ll'$ is given by pull back.\newline

Recall that any normal test configuration $(\X,\Ll)$ defines a function $\varphi_{(\X,\Ll)}:X^{\divv}\to \R$ (subsection \ref{ssec:NA}), and that two normal test configurations $(\X,\Ll), (\X',\Ll')$ defines the same class $\Ll\in N^1(\FX_{\PP^1}^{\C^*})$ if and only if $\varphi_{(\X,\Ll)}=\varphi_{(\X',\Ll')}$ (Lemma \ref{lem:Nice}). Hence we will use the notation $\varphi_{\Ll}$.
\begin{lem}
\label{lem:Indu}
Let $U:\R_{\geq 0}\to \PSH$ be a $[\psi]$-relative psh ray and let $\Ll\in N^1(\FX_{\PP^1}^{\C^*})$ be a test configuration class. Then the following are equivalent:
\begin{itemize}
    \item[i)] $U$ induces a psh metric on $\Ll$;
    \item[ii)] $U$ has linear growth and $U_{\NA}\leq \varphi_{\Ll}.$
\end{itemize}
\end{lem}
\begin{proof}
It immediately follows from \cite[Lemma 4.4]{BHJ17} since any $[\psi]$-relative psh ray is clearly a psh ray.
\end{proof}

\begin{prop}
\label{prop:Abo}
Let $\Ll\in N^1(\FX_{\PP^1}^{\C^*})$ be a test configuration class, let $(\X,\Ll)\in \FX_\Ll$, and let $\CZ$ be a smooth projective variety such that $\CZ\geq \X,\CZ\geq Y\times\PP^1$ where $Y\geq X$ satisfies $\CZ_{\pi^{-1}\C^*}\simeq Y\times\C^*$. Then any $[\psi]$-relative psh ray $U:\R_{\geq 0}\to \PSH$ with $U_{\NA}\leq \varphi_\Ll$ induces a psh metric on the incarnation $\Ll_\CZ:=(\Ll-\CD)_\CZ$. 
\end{prop}
\begin{proof}
By hypothesis we have the commutative diagram
$$
\begin{tikzcd}
\CZ \ar[r, "f"] \ar[d, "\mu_Z"] & \X \ar[d, "\mu_\X"]\\
Y\times\PP^1 \ar[r, "\rho_Y\times \Id"] & X\times\PP^1
\end{tikzcd}
$$
and the $[\psi]$-relative psh ray $U:\R_{\geq 0}\to \PSH(X,\omega)$ induces a singular metric on $(\X,\Ll)_{|\D}$ with current of curvature $T\geq 0$ (Lemma \ref{lem:Indu}). Then the properties of the Siu decomposition implies that $f^*T$ induces a metric on $(\CZ,f^*\Ll-\CD_\CZ)_{|\D}$ if and only if $f^*T\geq [\CD_\CZ]_{|\D}$. By the $\C^*$-invariance and the definition of $\CD_\CZ$ this holds if and only if for any $t\in \D^*$ fixed
\begin{equation}
    \label{eqn:Ine}
    f^*T_{|\CZ_t}\geq [\CD_{\CZ}]_{|\CZ_t}
\end{equation}
Since $T$ is given as curvature of the $S^1$-invariant metric $e^{-2U}p_1^*h$ ($h$ metric associated to $\omega$), the inequality (\ref{eqn:Ine}) holds if and only if
$$
\rho_Y^*(\omega+dd^c u_t)\geq [D_Y].
$$
But this is an easy consequence of the fact that $u_t$ is more singular than $\psi$ while $D_Y=\sum_F \nu(\psi\circ \rho_Y,F)F$ where the sum runs over all prime divisors $F\subset Y$ (Proposition \ref{prop:Corre}).
\end{proof}

\subsection{Algebraic psh rays}
\begin{defi}
\label{defi:WBig}
We will say that $\alpha\in N^1(\FX_{\PP^1}^{\C^*})$ is \emph{big} if there exists $\X\geq X\times\PP^1$ such that $\alpha_{\CZ}$ is big for any $\CZ\geq \X$ and 
$$
\langle \alpha^{n+1} \rangle:=\sup_{\CZ'}\inf_{\CZ\geq \CZ'}\langle \alpha_{\CZ}^{n+1} \rangle>0.
$$
\end{defi}
If $\Ll\in N^1(\FX_{\PP^1}^{\C^*})$ is a test configuration class, then there exists $m_0$ such that $\Ll_m:=\Ll+m(X\times\{0\})-\CD\in N^1(\FX_{\PP^1}^{\C^*})$ is big for any $m>m_0$ (see Lemmas \ref{lem:Ammappate}, \ref{lem:PIPWeil}).
\begin{prop}
\label{prop:Current}
Let $\Ll\in N^1(\FX_{\PP^1}^{\C^*})$ be a test configuration class such that $\Ll-\CD\in N^1(\FX_{\PP^1}^{\C^*})$ is big. On any $(\X,\Ll)\in\FX_\Ll$ there is a canonical $S^1$-invariant closed and positive $(1,1)$-current $T^{(\X,\Ll,\bD)}\in c_1(\Ll)$ such that
\begin{equation}
    \label{eqn:Totalmass}
    \langle (\Ll-\CD)^{n+1} \rangle= \int_\X \langle T^{(\X,\Ll,\bD),n+1} \rangle.
\end{equation}
\end{prop}
The wedge product on right hand side of (\ref{eqn:Totalmass}) is given by the non-pluripolar product (subsection \ref{ssec:KE}).
\begin{proof}
Fix $Y\geq X$. \newline
Let $\Y$ be a smooth projective variety $\C^*$-equivariantly dominating $\X$ and $Y\times\PP^1$ such that $\Y_{\pi^{-1}\C^*}\simeq Y\times\PP^1$, and set $\Ll_\Y:=f^*\Ll-\CD_\Y$. Then, since $f:\Y\to \X$ has connected fibres, for any closed and positive $(1,1)$-current with minimal singularities $S^\Y_{\min}\in c_1(\Ll_\Y)$, there exists a closed and positive $(1,1)$-current $T^\Y\in c_1(\Ll)$ such that
$$
f^*T^\Y=S_{\min}^\Y+[\CD_\Y].
$$
Fix now $\Omega$ be a smooth closed $(1,1)$-form in $c_1(\Ll)$. By the $\partial\bar{\partial}$-lemma we have $T^\Y:=\Omega+dd^c u^\Y$ for $u^\Y\in \PSH(\X,\Ll)$. We define
$$
\Psi^Y:=P_\Omega[u^\Y](0),
$$
i.e. the model type envelope associated to the singularities of $T^\Y$ with respect to $\Omega$ (changing $\Omega'$ would modify $\Psi^Y$ by a bounded function). Clearly $\Psi^Y$ does not depend on the choice of the current with minimal singularities $S_{\min}$, i.e. a priori $\Psi^Y$ is associated to the data $\big((\X,\Ll),\bD,\Omega,\Y\big)$, where as said above the dependence on $\Omega$ is very weak. \newline
Furthermore, letting $\CZ$ be another smooth projective variety $\C^*$-equivariantly dominating $\X$ and $Y\times\PP^1$ such that $\CZ_{|\pi^{-1}\C^*}\simeq Y\times\C^*$, we claim that
\begin{equation}
    \label{eqn:D}
    P_\Omega[u^\CZ](0)=P_\Omega[u^\Y](0),
\end{equation}
i.e. that $\Psi^Y$ depends on the variety $Y\geq X$ but not on the choice of $\Y$ such that $\Y_{\pi^{-1}\C^*}\simeq Y\times\C^*$, as the notation suggests. Indeed, since the set of projective varieties considered in (\ref{eqn:D}) is directed we can assume that $\CZ\geq \Y$. Denoting with $g:\CZ\to \Y$ the birational morphism,
\begin{equation}
    \label{eqn:D2}
    S^\CZ:=g^*S^\Y_{\min}+[g^*\CD_\Y-\CD_\CZ]
\end{equation}
is a closed and positive $(1,1)$-current in $c_1(\Ll_\CZ)$, where we clearly used the fact that $\Ll_\CZ=g^*\Ll_\Y+g^*\CD_\Y-\CD_\CZ$ and that by construction $[g^*\CD_\Y-\CD_\CZ]$ is a $g$-exceptional closed and positive $(1,1)$-current. Thus (\ref{eqn:D2}) immediately leads to $P_\Omega[u^\CZ](0)\geq P_\Omega[u^\Y](0)$. On the other hand, letting $S^\CZ_{\min}$ be a closed and positive $(1,1)$-current with minimal singularities in $c_1(\Ll_\CZ)$, we get that $g_* S^\CZ_{\min}$ is a closed and positive $(1,1)$-current in $c_1(\Ll_\Y)$. Hence the current $T^\CZ$ given by $f^*T^\CZ:=g_*S^\CZ_{\min}+[\CD_\Y]$ is more singular than $T^\Y$, and since the morphism $\CZ\to \X$ factorises through $g:\CZ\to \Y$ and $f:\Y\to \X$ it follows that $P_\Omega[u^\CZ](0)\leq P_\Omega[u^\Y](0)$, which concludes the proof of the claim.\newline
Furthermore, letting $\Y'\overset{h}{\geq}\Y$ where $\Y'_{\pi^{-1}\C^*}\simeq Y'\times\C^*$ for $Y'\geq Y$, reasoning as before it is immediate to see that $h_* S^{\Y'}_{\min}$ is more singular than $S^{\Y}_{\min}$, i.e. $T^{\Y'}$ is more singular than $T^{\Y}$. Hence $\Psi^{Y'}\leq \Psi^{Y}$.\newline
Summarizing we have constructed a decreasing net $\{\Psi^Y\}_{Y\geq X}\subset \PSH(\X,\Omega)$ of model type envelopes. Therefore Lemma \ref{lem:Inf} implies that
\begin{equation}
    \label{eqn:Psi}
    \Psi:=\inf_{Y}\Psi^Y,
\end{equation}
belongs to $\PSH(\X,\Omega)$. Moreover, by what recalled in subsection \ref{ssec:KE}, $\Psi\in \cM^+(\X,\Omega)$ since
$$
\int_\X MA_\Omega(\Psi^Y)=\vol_\Y(\Ll_\Y)=\langle \Ll_\Y^{n+1} \rangle\geq \langle (\Ll-\CD)^{n+1} \rangle>0
$$
and again Lemma \ref{lem:Inf} leads to
$$
\int_\X MA_\Omega(\Psi)=\inf_{Y\geq X}\int_\X MA_\Omega(\Psi^Y)=\inf_{\Y\geq\X} \langle\Ll_\Y^{n+1} \rangle=\langle (\Ll-\CD)^{n+1} \rangle.
$$
Hence setting $T^{(\X,\Ll,\bD)}:=\Omega+dd^c \Psi$ concludes the proof.
\end{proof}
\begin{corollary}
\label{cor:Abo}
Let $\Ll\in N^1(\FX_{\PP^1}^{\C^*})$ be a test configuration class such that $\Ll-\CD\in N^1(\FX_{\PP^1}^{\C^*})$ is big. Then any $(\X,\Ll)\in\FX_\Ll$ defines a $[\psi]$-relative psh ray $U^{(\X,\Ll,\bD)}:\R_{\geq 0}\to \PSH$.\newline
Moreover, if $\Ll-L\times\PP^1$ is pseudoeffective any other $[\psi]$-relative psh ray $U:\R_{\geq 0}\to \PSH$ such that $U_{\NA}\leq \varphi_{\Ll}$ satisfies $U\leq U^{(\X,\Ll,\bD)}+O(1)$. In particular, the associated function $U^{(\X,\Ll,\bD)}_{\NA}:X^{\divv}\to \R$ does not depend on the choice of $(\X,\Ll)\in\FX_\Ll$, i.e. the class $\Ll\in N^1(\FX_{\PP^1}^{\C^*})$ defines a function $U_{\NA}^{(\Ll,\bD)}:X^{\divv}\to \R$.
\end{corollary}
\begin{proof}
The psh ray $U^{(\X,\Ll,\bD)}$ is clearly given by Proposition \ref{prop:Current}, i.e. it is induced by the current $T^{(\X,\Ll,\bD)}\in c_1(\Ll)$ in the following way.\newline
Let $F$ be the divisor supported in the central fiber such that $\Ll=\mu_\X^*p_X^*L+F$ as a $\Q$-line bundle on $\X$, where the morphism $\mu_\X:\X\to X\times\PP^1$ is given by the dominance. Write $F=\sum_j \alpha_j F_j$ for prime divisors $F_j$ and coefficients $\alpha_j\in\Q$, and let $\theta_{F_j}$ be a $\C^*$-invariant smooth and closed $(1,1)$-form in $c_1(F_j)$ such that $\Omega=\mu_\X^*p_X^*\omega+\sum_j\alpha_j \theta_j$. Let also $s_j$ be $\C^*$-invariant sections of $F_j$ so that $\theta_j+dd^c\log |s_j|_{h_j}=[F_j]$ where $h_j$ is the metric on $\mathcal{O}_\X(F_j)$ associated to $\theta_j$. Then over $\D^*$ we have
$$
0\leq T^{(\X,\Ll,\bD)}_{|\D^*}= \big(T^{(\X,\Ll,\bD)}-\sum_j \alpha_j[F_j]\big)_{|\D^*}= \Big(p_1^*\omega+ dd^c\big(\Psi-\sum_j \alpha_j\log|s_j|_{h_j}\big)\Big)_{|\D^*},
$$
and $U^{(\X,\Ll,\bD)}$ is induced by the $S^1$-invariant function $v^{(\X,\Ll,\bD)}:=\Psi-\sum_j\alpha_j\log |s_j|_{h_j}\in \PSH(X\times\D^*,p_1^*\omega)$, i.e.
$$
v^{(\X,\Ll,\bD)}(x,\tau)=u_{-\log|\tau|}^{(\X,\Ll,\bD)}(x).
$$
Observe that a different choice of $\Omega, \theta_j, s_j$ would produce a different psh ray $U^{(\X,\Ll,\bD)'}:\R_{\geq 0}\to \PSH$ such that $\lvert u_t^{(\X,\Ll,\bD)'}-u_t^{(\X,\Ll,\bD)}\rvert \leq C$ uniformly in $t\in [0,+\infty)$.\newline
By construction for any $t\in [0,+\infty)$ and for any $Y\overset{\rho_Y}{\geq} X$
$$
\rho_Y^*\big(\omega+u_t^{(\X,\Ll,\bD)}\big)\geq [D_Y].
$$
In particular, letting $\psi'\in \M_D$ algebraic model type envelope such that $\psi'\geq \psi$, we consider $\rho_{Y'}:Y'\to X$ log resolution of the ideal sheaf associated to $[\psi']$ to find out that $D_{Y'}\geq \sum_F\nu(\psi'\circ \rho_Y,F)F$ by definition of $D_{Y'}$ (Proposition \ref{prop:Corre}), which reads as $u_t^{(\X,\Ll,\bD)}$ being more singular than $\psi'$, i.e. $u_t^{(\X,\Ll,\bD)}\leq \psi'+\sup_X u_t^{(\X,\Ll,\bD)}$ (subsection \ref{ssec:KE}). Hence, considering $\psi_k\searrow\psi$ decreasing sequence of model type envelopes with algebraic singularities, we deduce that $u_t^{(\X,\Ll,\bD)}\leq \psi+\sup_X u_t^{(\X,\Ll,\bD)}$, i.e. $U^{(\X,\Ll,\bD)}:\R\to \PSH$ is a $[\psi]$-relative psh ray.\newline
Next, if $\Ll-L\times\PP^1$ is pseudoeffective, then by Lemma \ref{lem:Effe} the trivial psh ray $\Psi:\R_{\geq 0}\to \PSH$, $t\to \psi$, induces a psh metric on $\Ll$. In particular, if $U:\R_{\geq 0}\to \PSH$ is a $[\psi]$-relative psh ray such that $U_{\NA}\leq \varphi_{\Ll}$, then $\max(U,\Psi):\R_{\geq 0}\to \PSH $, $t\to \max(u_t,\psi)$, induces a psh metric on $\Ll$ and clearly $U\leq \max(U,\Psi)$. Passing to the associated $S^1$-invariant $p_1^*\omega$-psh functions on $X\times\D^*$, with obvious notations, we obtain
$$
v(x,\tau)\leq \max(v,\psi)(x,\tau)\leq \max\big(\psi(x)+\sup_{x\in X} u_{-\log\lvert \tau \rvert}(x), \psi(x)\big)
$$
where the last inequality follows from $u\leq \psi+\sup_X u$ if $u\preccurlyeq\psi$ (subsection \ref{ssec:KE}). Furthermore, as by Lemma \ref{lem:Indu} $U$ has linear growth, without loss of generality we can also assume that $ u_t\leq -1+Ct $ for any $t\in \R_{\geq 0}$ where $C>0$. Thus we get
$$
\max(v,\psi)=\psi
$$
over $X\times \D_{[0,1/C)}$ where $\D_I:=\{\tau\in \D\, :\, -\log\lvert \tau\rvert\in I\}$. In other words, up to translate $U$ and to replace it by $\max(U,\Psi)$, we can find a $S^1$-invariant closed and positive current $T\in c_1(\Ll)$ such that
$$
T_{\PP^1\setminus\{0\}}=\Big(p_1^*\omega+dd^c \tilde{v}\Big)_{\PP^1\setminus\{0\}}
$$
where $\tilde{v}(x,\tau)= v(x,\tau)=u_{-\log \lvert \tau \rvert}(x)$ over $\D$ while $\tilde{v}=\psi$ over $\PP^1\setminus\D$.\newline
Next, proceeding as in Proposition \ref{prop:Abo} we get that $T$ is more singular than $T^{\CZ}$ for any $\CZ\geq \X$ (see Proposition \ref{prop:Current} for the construction of $T^\CZ$). At level of potentials, setting $\Phi:=\tilde{v}+\sum_j \alpha_j\log \lvert s_j\rvert_{h_j}$, this means that $\Phi\leq P_\Omega[u^{\CZ}](0)+\sup_\X \Phi$ for any $\CZ\geq \X$. Thus $\Phi\leq \Psi+\sup_\X \Phi$ where $T^{(\X,\Ll,\bD)}=\Omega+dd^c\Psi$. Hence, since $U^{(\X,\Ll,\bD)}$ is induced by $T^{(\X,\Ll,\bD)}$ we deduce that
$$
u_{-\log \lvert \tau \rvert}(x)\leq \tilde{v}(x,\tau)+O(1)\leq v^{(\X,\Ll,\bD)}(x,\tau)+O(1)=u_{-\log\lvert \tau \rvert}^{(\X,\Ll,\bD)}(x)+O(1)
$$
as requested. Applying this result to $(\X,\Ll), (\X',\Ll')\in \FX_\Ll$, we clearly get
$$
\lvert U^{(\X,\Ll,\bD)}-U^{(\X',\Ll',\bD)} \rvert\leq C,
$$
which concludes the proof.
\end{proof}
As consequence of Proposition \ref{prop:Current} and Corollary \ref{cor:Abo} we give the following definition.
\begin{defi}
\label{defi:Alge}
A $[\psi]$-relative psh ray $U:\R\to \PSH$ will be said \emph{algebraic} if $U_{\NA}=U_{\NA}^{(\Ll,\bD)}$ for a test configuration class $\Ll$ such that $\Ll-\CD\in N^1(\FX_{\PP^1}^{\C^*})$ is big and such that $\Ll-L\times\PP^1\in N^1(\FX_{\PP^1}^{\C^*})$ is pseudoeffective. 
\end{defi}

\subsection{$\bD$-log Ding functional}
In this subsection we introduce the $\bD$-log Ding functional, we define the (uniform) $\bD$-log Ding stability and we prove that the latter implies the (uniform) $\bD$-log $K$-stability.

\subsubsection{$\bD$-log $L$-functional}
\begin{defi}
Let $U:\R\to \PSH$ be a $[\psi]$-relative psh ray with linear growth. Then
$$
L^{\NA}(U_{\NA})=\inf_{X^{\divv}}\Big\{A_X+U_{\NA}\Big\}.
$$
When $U=U^{(\Ll,\bD)}$ is algebraic (Definition \ref{defi:Alge}), we will also use the notation
$$
L^{\NA}(\Ll;\bD):=L^{\NA}\big(U_{\NA}^{(\Ll,\bD)}\big).
$$
\end{defi}
When $\bD=0$, i.e. $\psi=0$, this definition recovers the $L$-functional studied in \cite{BBJ15}, which in turn is a generalization of the functional introduced in \cite[Proposition 3.8]{Berm16}. In Remark \ref{rem:Llct} below we will interpret $L^\NA(\Ll;\bD)$ in terms of log canonical thresholds similarly to \cite{Berm16} when $L=-K_X$.\newline

Observe that $L^{\NA}(\Ll;\cdot)$ is decreasing. Namely, if $\bD'\geq \bD$ as generalized $b$-divisors and if $\Ll\in N^1(\FX_{\PP^1}^{\C^*})$ is a test configuration class such that $\Ll-\CD'$ is big and $\Ll-L\times\PP^1$ is psef, then $L^{\NA}(\Ll;\bD')\leq L^{\NA}(\Ll;\bD)$ as an immediate consequence of $U_{\NA}^{(\Ll,\bD')}\leq U_{\NA}^{(\Ll,\bD)}$ (Proposition \ref{prop:Current} and Corollary \ref{cor:Abo}).\newline

Next, similarly to subsection \ref{ssec:Cartier}, assuming $D_Y$ is a $\Q$-divisor, we can consider $(\Y,\Ll_\Y)$ compactification of a dominating (i.e. $\Y\overset{\mu_\Y}{\geq}Y\times\C$) test configuration for $(Y,L_Y:=\rho_Y^*L-D_Y)$ where $Y\geq X$. If $\Ll_Y$ is big we can then consider a $\C^*$-invariant current with minimal singularities $S_{\min}^\Y$ to induce a psh ray $U^{\Ll_\Y}:\R_{\geq 0}\to \PSH(Y,L_Y)$. Note that a different choice of the $\C^*$-invariant current with minimal singularities would produce another psh ray $\tilde{U}^{\Ll_Y}:\R_{\geq 0}\to \PSH(Y,L_Y)$. In particular $U^{\Ll_Y}_\NA: Y^\divv\to \R$ is a well-defined function.
\begin{prop}
\label{prop:LF}
Let $U:\R\to \PSH$ be a $[\psi]$-relative psh ray with linear growth, and let $U^Y:\R\to \PSH(Y,\rho_Y^*\omega-\theta_Y)$ be the psh ray on $Y\overset{\rho_Y}{\geq} X$ obtained as $U^Y=U\circ \rho_Y-u_Y$ where $u_Y\in \PSH(Y,\theta_Y)$ such that $\theta_Y+dd^c u_Y=[D_Y]$ for $\theta_Y$ smooth and closed $(1,1)$-form. Then
\begin{equation}
    \label{eqn:H-1}
    L^{\NA}(U_{\NA})= L^{\NA}_{B_Y}(U^Y_{\NA})=\inf_{Y^{\divv}}\Big\{A_{(Y,B_Y)}+U_{\NA}^Y\Big\}
\end{equation}
where $B_Y:=D_Y-K_{Y/X}$.\newline
Moreover if $U=U^{(\Ll,\bD)}$ is algebraic and $D_Y$ is a $\Q$-divisor for any $Y\geq X$ then
\begin{equation}
    \label{eqn:H-2}
    L^{\NA}(\Ll;\bD)=\inf_{Y\geq X} L^{\NA}_{B_Y}(\Y,\Ll_\Y):=\inf_{Y\geq X}\inf_{Y^{\divv}}\Big\{A_{(Y,B_Y)}+U^{\Ll_\Y}_{\NA}\Big\}
\end{equation}
where $\Y$ dominates $\X$, $Y\times\PP^1$ and satisfies $\Y_{\pi^{-1}\C^*}\simeq Y\times\C^*$.\newline
In particular if $\bD\in \bDiv_L(X)$ is a Cartier divisor, i.e. it is determined by a $\Q$-divisor $D_Y$ for $Y\geq X$, then
$$
L^{\NA}(\Ll;\bD)=L_{B_Y}^{\NA}(\Y,\Ll_Y)=\inf_{Y^\divv}\Big\{A_{(Y,B_Y)}+U_{\NA}^{\Ll_\Y}\Big\}.
$$
\end{prop}
\begin{proof}
By construction we immediately have
\begin{equation}
    \label{eqn:H0}
    U_{\NA}(\nu)=U_{\NA}^Y(\nu)-\nu(D_Y)
\end{equation}
for any $\nu\in Y^\divv$.
On the other hand, letting $\nu=c\,\ord_E$ for $c>0$ and $E\subset Z$ prime divisor where $Z\geq Y$ (i.e. $\nu\in Y^{\divv}$), it follows that
\begin{multline}
    \label{eqn:H2}
    A_{X}(\nu)=c\big(1+\ord_E K_{Z/X}\big)=\\
    =c\big(1+\ord_E K_{Z/(Y,B_Y)}+\ord_E K_{(Y,B_Y)/X}\big)=A_{(Y,B_Y)}(\nu)+\nu\big(D_Y\big).
\end{multline}
Combining (\ref{eqn:H0}) with (\ref{eqn:H2}), we deduce (\ref{eqn:H-1}).\newline 
Next, assume that $U=U^{(\Ll,\bD)}$ is algebraic and that $D_Y$ is a $\Q$-divisor for any $Y\geq X$. 
Let also $(\X,\Ll)\in\FX_\Ll$ and $U^{(\X,\Ll,\bD)}:\R_{\geq 0}\to \PSH(X,\omega)$ be a $[\psi]$-relative psh ray associated to $(\X,\Ll,\bD)$ as in Corollary \ref{cor:Abo}. Then, the psh ray $U^Y:\R\to \PSH(Y,\eta_Y)$ given by
$$
U^Y_t:=U^{(\X,\Ll,\bD)}\circ\rho_Y-u_Y,
$$
where $u_Y\in \PSH(Y,\theta_Y)$ such that $\theta_Y+dd^c u_Y=[D_Y]$, extends to a psh metric on $(\Y,\Ll_\Y)$. Here $\Y$ is the graph of the birational map $\mu_\X^{-1}\circ(\rho_Y\times\Id): Y\times\PP^1\dashrightarrow \X$ while $\Ll_\Y:=f^*\Ll-\CD_\Y$ for $\Y\overset{f}{\geq}\X$. 
Indeed, as seen in Proposition \ref{prop:Abo}, $U^Y$ is induced by the positive and closed $(1,1)$-current $S\in c_1(\Ll_\Y)$ given by
$$
S=f^*T^{(\X,\Ll,\bD)}-[\CD_\Y].
$$
(see Proposition \ref{prop:Current} for the construction of $T^{(\X,\Ll,\bD)}$). In particular, since $S$ is clearly more singular than a current with minimal singularities in $c_1(\Ll_\Y)$, we get $U_{\NA}^Y\leq U_{\NA}^{\Ll_\Y}$ and
\begin{equation}
    \label{eqn:H1}
    U_{\NA}^{(\Ll,\bD)}(\nu)\leq U_{\NA}^{\Ll_\Y}(\nu)-\nu(D_Y)
\end{equation}
for any $\nu\in Y^{\divv}$. 
Proceeding as before, (\ref{eqn:H1}) together with (\ref{eqn:H2}) lead to
$$
L^{\NA}(\Ll;\bD)\leq \inf_{Y\geq X}L^{\NA}_{B_Y}(\Y,\Ll_\Y).
$$
For the reverse inequality we first observe that for any $\nu\in X^{\divv}$ we have
\begin{equation}
    \label{eqn:H-3}
    A_{(Y,B_Y)}(\nu)+U^{\Ll_\Y}_{\NA}(\nu)=A_X(\nu)+\tilde{U}^Y_{\NA}(\nu)
\end{equation}
where $\tilde{U}^Y:\R_{\geq 0}\to \PSH$ is the psh ray induced by the current $T^Y:=\Omega+dd^c P_\Omega[u^\Y](0)$ (see Proposition \ref{prop:Current} for the notations). In particular, as a consequence of Lemma \ref{lem:Inf}, we can consider a sequence $Y_{k+1}\geq Y_k\geq X$ such that the associated currents $T^{Y_k}$ converges to $T^{(\X,\Ll,\bD)}$. At level of model type envelopes we have $\cM^+_D(\X,\Omega)\ni P_\Omega[u^{\Y_k}](0)\searrow \Psi\in \cM^+_D(\X,\Omega)$ where $T^{(\X,\Ll,\bD)}=\Omega+dd^c\Psi$. In particular we have convergence of Lelong numbers, i.e. $\nu(T^{Y_k},F)\nearrow \nu(T^{(\X,\Ll,\bD)},F)$ for any $F$ prime divisor above $\X$ (subsection \ref{ssec:KE}). We therefore claim that $\tilde{U}^{Y_k}_{\NA}$ pointwise converges to $U_\NA^{(\Ll;\bD)}$.\newline
Fix $\nu\in X^\divv$. By \cite[Lemma 4.7]{BHJ17} there exists a (compactification of a) normal test configuration $\tilde{\X}\geq X\times\PP^1$ and an irreducible component $E$ of $\tilde{\X}_0$ such that $G(\nu)=c\,\ord_E$ for $c>0$. Observe that, unless considering a different element $(\X',\Ll')\in\FX_\Ll$, we can and we will assume that
$$
\X\overset{\mu}{\longrightarrow} \tilde{\X}\overset{\tilde{\mu}}{\longrightarrow} X\times\PP^1.
$$
Denote by $F$ be $\Q$-divisor on $\X$ such that $\Ll=\mu^*\tilde{\mu}^*p_X^* L+F$. Define $S^{(\X,\Ll,\bD)}, S^{Y_k}$ as the closed and positive $(1,1)$-currents on $X\times\PP^1$ given as extension respectively of $(T^{(\X,\Ll,\bD)}-[F]+A[\X_0])_{\pi^{-1}(\PP^1\setminus\{0\})}$ and of $(T^{Y_k}-[F]+A[\X_0])_{\pi^{-1}(\PP^1\setminus\{0\})}$ for $A\gg 0$ big enough so that $A\X_0\geq F$. By construction (see Corollary \ref{cor:Abo}) we have
$$
U_\NA^{(\Ll,\bD)}(\nu)=-G(\nu)(U^{(\Ll,\bD)}+A\log\lvert \tau\rvert)+AG(\nu)(\tau)=-c\,\nu(\tilde{\mu}^*S^{(\X,\Ll;\bD)},E)+AG(\nu)(\tau)
$$
and similarly for $\tilde{U}^{Y_k}_\NA$. Therefore the claim follows by linearity of Lelong numbers, considering the strict transform of $E$ over $\X$ and using that $\nu(T^{Y_k},E)\nearrow \nu(T^{(\X,\Ll,\bD)},E)$ as said above.\newline
Hence, for any $\epsilon>0$ fixed, there exists $\nu_\epsilon \in X^\divv$ such that
$$
L^\NA(\Ll;\bD)\geq A_X(\nu_\epsilon)+U^{(\Ll;\bD)}_\NA(\nu_\epsilon)-\epsilon=A_X(\nu_\epsilon)+\lim_{k\to +\infty} \tilde{U}^{Y_k}_\NA(\nu_\epsilon)-\epsilon\geq \inf_{Y\geq X} L_{B_Y}^\NA(\Y,\Ll_\Y)-\epsilon
$$
where in the last inequality we also used (\ref{eqn:H-3}). The arbitrariness of $\epsilon$ concludes the proof. 
\end{proof}
\subsubsection{$\bD$-log Ding stability}
\begin{defi}
Let $\bD\in \bDiv_L(X)$ and let $\Ll\in N^1(\FX_{\PP^1}^{\C^*})$ be a test configuration class. The \emph{$\bD$-log Ding functional of $\Ll$} is defined as
\begin{equation}
    \label{eqn:Ding}
    D^{\NA}(\Ll;\bD):=L^{\NA}\big(\Ll+m(X\times\{0\});\bD\big)-E^{\NA}\big(\Ll+m(X\times\{0\});\bD\big)
\end{equation}
for $m>m_0$, where $m_0$ is given by Lemma \ref{lem:Effe}.
\end{defi}
Recall that 
$$
E^{\NA}\big(\Ll+m(X\times\{0\});\bD\big):=\frac{\langle \Ll_m^{n+1}\rangle}{(n+1)V_\bD}
$$
where $\Ll_m:=\Ll+m(X\times\{0\})-\CD\in N^1(\FX_{\PP^1}^{\C^*})$ and $V_\bD:=\langle (L-\bD)^n\rangle$ (see Definition \ref{defi:JandE})\newline
When $\bD=0$ (i.e. $\psi=0$) the $\bD$-log Ding functional coincides with the \emph{Non-Archimedean Ding functional} (\cite[Definition 7.26]{BHJ17}) on the set of ample test configuration classes considering $m=0$ in the definition.\newline

We need to check that $D^{\NA}(\Ll;\bD)$ is well-defined, i.e. that the right hand side in (\ref{eqn:Ding}) does not depend on $m>m_0$.
\begin{prop}
The $\bD$-log Ding functional is well-defined.
\end{prop}
\begin{proof}
For any $m>m_0$ set $\Ll_m:=\Ll+m(X\times \{0\}-\CD$. As seen in the proof of Proposition \ref{prop:TI}, 
\begin{equation*}
    E^\NA\big(\Ll+m_1(X\times\{0\});\bD\big)=E^\NA\big(\Ll+m_2(X\times\{0\});\bD\big)+m_1-m_2
\end{equation*}
for any $m_1\geq m_2>m_0$ since as consequence of Lemma \ref{lem:Below}
\begin{equation}
    \label{eqn:V}
    \langle \Ll_{m_1}^{n+1} \rangle=\langle \Ll_{m_2}^{n+1}\rangle +(m_1-m_2)(n+1)V_{\bD}.
\end{equation}
Thus it is enough to prove that
\begin{equation}
    \label{eqn:Claim}
    L^\NA\big(\Ll+m_1(X\times\{0\});\bD\big)=L^\NA\big(\Ll+m_2(X\times\{0\});\bD\big)+m_1-m_2.
\end{equation}
Fix $(\X,\Ll)\in\FX_\Ll$ and, keeping the notation of Proposition \ref{prop:Current}, define
$$
S:=T^{(\X,\Ll+m_2(X\times\{0\}),\bD)}+(m_1-m_2)\omega_{\PP^1}
$$
where $\omega_{\PP^1}:=\pi^*\omega_{FS}$ is the pullback of the Fubini-Study K\"ahler form through $\pi:\X\to \PP^1$. By a slight abuse of notation we will also write $T^{(\X,\Ll^m,\bD)}:=T^{(\X,\Ll+m(X\times\{0\}),\bD)}$. Then $S$ is a closed and positive current in $c_1(\Ll_{m_1})$, and moreover by construction it is easy to check that $S$ is more singular than $T^{(\X,\Ll^{m_1},\bD)}$. Therefore by monotonicity of the non-pluripolar product \cite[Theorem 1.2]{WN17} and Proposition \ref{prop:Current} it follows that
\begin{equation}
    \label{eqn:C1}
    \int_\X\langle S^{n+1}\rangle\leq \int_\X \langle T^{(\X,\Ll^{m_1},\bD),n+1}\rangle=\langle \Ll_{m_1}^{n+1}\rangle=\langle \Ll_{m_2}^{n+1}\rangle+ (m_1-m_2)(n+1)V_{\bD}
\end{equation}
where the last equality is the content of (\ref{eqn:V}). \newline
On the other hand, by an easy computation, using again Proposition \ref{prop:Current}, we also have
\begin{multline}
    \label{eqn:C2}
    \int_\X \langle S^{n+1} \rangle =\int_\X\langle \big(T^{(\X,\Ll^{m_2},\bD)}+(m_1-m_2)\omega_{\PP^1}\big)^{n+1} \rangle=\\
    =\int_\X \langle T^{(\X,\Ll^{m_2},\bD),n+1} \rangle+(m_1-m_2)(n+1)\int_\X \langle \omega_{\PP^1}\wedge T^{(\X,\Ll^{m_2},\bD),n}\rangle=\\
    =\langle\Ll_{m_2}^{n+1} \rangle+(m_1-m_2)(n+1)\int_{X\times\C^*}\langle\omega_{\PP^1}\wedge T^{(\X,\Ll^{m_2},\bD),n}\rangle=\langle \Ll_{m_2}^{n+1} \rangle+(m_1-m_2)(n+1)\int_X \langle T^{(\X,\Ll^{m_2},\bD),n}_{|\X_1}\rangle
\end{multline}
where the last equality follows from the fact that $T^{(\X,\Ll^{m_2},\bD)}$ is $\C^*$-invariant while $\omega_{\PP^1}$ clearly vanishes along the fibers (and $\int_{\PP^1}\omega_{\mathrm{FS}}=1$). Then we observe that $T_{|\X_1}$ has $[\psi]$-relative minimal singularities since $\omega+dd^c\psi$ extends to a current $T\in c_1(\Ll_{m_2})$ which is less singular than $T^{(\X,\Ll^{m_2},\bD)}$ (recall that $\Ll+m_2\X_0-\mu_\X^*p_X^*L$ is effective). In particular we have
$$
\int_X \langle T_{|\X_1}^{(\X,\Ll^{m_2},\bD),n+1} \rangle=\int_X MA_\omega(\psi)=V_\bD.
$$
Combining (\ref{eqn:C1}) and (\ref{eqn:C2}) we deduce that
$$
\int_\X \langle S^{n+1}\rangle =\int_\X \langle T^{(\X,\Ll^{m_1},\bD),n+1}\rangle>0.
$$
Thus, letting $\Omega'\in c_1(\Ll_{m_2})$ be a smooth and closed form and setting $\Omega:=\Omega'+(m_1-m_2)\omega_{\PP^1}\in c_1(\Ll_{m_1})$, it follows that $\Phi\in \PSH(\X,\Omega)$ given by $S=\Omega+dd^c \Phi$ has $[\Psi]$-relative full Monge-Ampère mass where $\Psi\in \cM^+(\X,\Omega)$ is given by $T^{(\X,\Ll^{m_1},\bD)}=\Omega+dd^c \Psi$. In particular, $\nu(\Phi,F)=\nu(\Psi,F)$ for any $F$ prime divisor above $\X$ (subsection \ref{ssec:KE}). Therefore, proceeding as in the last part of the proof of Proposition \ref{prop:LF}, we obtain that the $[\psi]$-relative psh ray $U:\R_{\geq 0}\to \PSH$ associated to $\Phi$ satisfies
$$
U_{\NA}=U^{(\X,\Ll+m_1(X\times\{0\}),\bD)}_{\NA}.
$$
On the other hand by construction we also have
$$
U_{\NA}=U^{(\X,\Ll+m_2(X\times\{0\}),\bD)}_{\NA}+(m_1-m_2).
$$
Hence (\ref{eqn:Claim}) clearly follows and concludes the proof.
\end{proof}

We can give the following definition.
\begin{defi}
Let $\bD\in \bDiv_L(X)$. Then $(X,L)$ is said to be \emph{$\bD$-log Ding stable} if $D^{\NA}(\Ll;\bD)\geq 0$ for any $\Ll\in N^1(\FX_{\PP^1}^{\C^*})$ ample test configuration class with equality if and only if $\Ll$ is $\bD$-trivial.\newline
$(X,L)$ is said to be \emph{uniformly $\bD$-log Ding-stable} if there exists $\delta>0$ such that $D^{\NA}(\Ll;\bD)\geq \delta J^{\NA}(\Ll;\bD)$ for any ample test configuration class.
\end{defi}
Observe that by definition of $\bD$-triviality if $(X,L)$ is uniformly $\bD$-log Ding-stable then it is $\bD$-log Ding stable (see Corollary \ref{cor:J}).\newline

The following result is the analog of Proposition \ref{prop:KstClass} for the $\bD$-log Ding stability and allows to interpret the latter as an \emph{asymptotic} log Ding stability.
\begin{prop}
\label{prop:DStCar}
Let $\bD\in \bDiv_L(X)$ such that $D_Y$ is a $\Q$-divisor for any $Y\geq X$, $\Ll\in N^1(\FX_{\PP^1}^{\C^*})$ be test configuration class, let $(\X,\Ll)\in \FX_\Ll$, and consider for any $Y\geq X$ the commutative diagram
$$
\begin{tikzcd}
    \Y \ar[r, "\tilde{\rho}_Y"] \ar[d, "\mu_\Y"']& \X \ar[d, "\mu_\X"]\\
    Y\times\PP^1 \ar[r,"\rho_Y\times\Id"] & X\times \PP^1
\end{tikzcd}
$$
where $\Y$ is the graph of $\mu_\X^{-1}\circ (\rho_Y\times\Id):Y\times\PP^1\dashrightarrow \X$. Then for any $m\gg 1$ big enough 
\begin{equation}
    \label{eqn:MMM}
    D^\NA(\Ll;\bD)=\sup_{Y'}\inf_{Y\geq Y'} D^\NA_{B_Y}(\Y,\Ll_{m,\Y})
\end{equation}
where
$B_Y:=D_Y-K_{Y/X}$, $(\Y,\Ll_{m,\Y})$ is the big test configuration for $(Y,L_Y:\rho_Y^*L-D_Y)$ given by $\Ll_{m,\Y}:=\tilde{\rho}_Y^*\Ll+m\Y_0-\CD_\Y$ and where
$$
D^\NA_{B_Y}(\Y,\Ll_{m,\Y})=L^\NA_{B_Y}(\Y,\Ll_{m,\Y})-\frac{\langle \Ll_{m,\Y}^{n+1} \rangle}{(n+1)V_{D_Y}}
$$
for $V_{D_Y}=\langle (L-D_Y)^n \rangle$. In particular if $\bD=D_Y$ then
$$
D^\NA(\Ll;\bD)=D^\NA_{B_Y}(\Y,\Ll_{m,\Y}).
$$
\end{prop}
\begin{proof}
As already seen in the proof of Proposition \ref{prop:KstClass}, Lemma \ref{lem:Convergence} implies that
$$
E^\NA\big(\Ll+m(X\times\{0\});\bD\big)=\frac{\langle \Ll_m^{n+1}\rangle}{(n+1)V_\bD}=
\inf_{Y'}\sup_{Y\geq Y'}\frac{\langle \Ll_{m,\Y}^{n+1}\rangle}{(n+1)V_{D_Y}}.
$$
Moreover by Proposition \ref{prop:LF}
$$
L^\NA\big(\Ll+m(X\times\{0\});\bD\big)=\inf_{Y\geq X}L^\NA_{B_Y}(\Y,\Ll_{m,\Y})
$$
and the quantity on the right hand side is monotone decreasing in $Y\geq X$ (see the proof of Proposition \ref{prop:LF}). Hence Lemma \ref{lem:Convergence} leads to (\ref{eqn:MMM}) and concludes the proof. 
\end{proof}
\begin{remark}
\label{rem:DingCohomo}
An analog of Proposition \ref{prop:DStCar} holds in general when $\bD$ is just a generalized $b$-divisor. Indeed, similarly to Remark \ref{rem:KstClass}, it is enough to work with the associated cohomology classes and to replace test configuration by the notion of \emph{cohomological} test configurations (\cite[Definition 3.2.5]{Sjo17}). 
\end{remark}
\subsubsection{$\bD$-log Ding implies $\bD$-log K-stability}
\label{ssec:DingToK}
Assuming $L=-K_X$, we have the following result which generalizes \cite[Proposition 7.32]{BHJ17}.
\begin{prop}
\label{prop:DingK}
Let $L=-K_X$, $\bD\in\bDiv_L(X)$ and $\Ll\in N^1(\FX_{\PP^1}^{\C^*})$ be a test configuration class. Then
\begin{equation}
    \label{eqn:Tired}
    DF(\Ll;\bD)\geq D^\NA(\Ll;\bD).
\end{equation}
In particular if $(X,-K_X)$ is (uniformly) $\bD$-log Ding stable then it is (uniformly) $\bD$-log K-stable.
\end{prop}
\begin{proof}
The proof is strongly inspired by \cite[Proposition 7.32]{BHJ17}.\newline
Let $(\X,\Ll)\in \FX_\Ll$. Up to replace $\Ll$ by $\Ll+m\X_0$ we can also assume that $\Ll-\mu_\X^*p_X^*L$ is effective and that $\Ll-\CD$ is big (Proposition \ref{prop:Big}). Then as immediate consequence of Propositions \ref{prop:KstClass}, \ref{prop:DStCar} (see also Remarks \ref{rem:KstClass}, \ref{rem:DingCohomo}), it is enough to prove that
\begin{equation}
    DF_{B_Y}(\Y,\Ll_\Y)\geq D^\NA_{B_Y} (\Y,\Ll_\Y)
\end{equation}
where, according to the commutative diagram
$$
\begin{tikzcd}
    \Y \ar[r, "\tilde{\rho}_Y"] \ar[d, "\mu_\Y"] & \X \ar[d, "\mu_\X"] \\
    Y\times\PP^1 \ar[r, "\rho_Y\times\Id"] & X\times\PP^1
\end{tikzcd}
$$
$B_Y:=D_Y-K_{Y/X}$, $\Y$ is given by the graph of the birational map $Y\times\PP^1\dashrightarrow \X$ for a fixed $Y\geq X$ and where $\Ll_\Y=\tilde{\rho}_Y^*\Ll-\CD_\Y$. To simplify the notations we set $L_{\PP^1}:=p_X^*L$, similarly for $L_{Y,\PP^1}=(\rho_Y^*L-D_Y)\times\PP^1$, $D_{Y,\PP^1}:=D_Y\times\PP^1$ and so on. Thus, letting $G$ be the effective divisor such that $\Ll=\mu_\X^*L_{\PP^1}+G$, an easy calculation shows that $ \Ll_\Y=\mu_\Y^*L_{Y,\PP^1}+G_\Y $ for $G_\Y=\tilde{\rho}_Y^*G+\mu_\Y^*D_{Y,\PP^1}-\CD_\Y$.\newline
Moreover, as $\overline{S}_{B_Y}=n$, letting $V_Y:=V_{D_Y}$,
\begin{equation}
    V_Y DF_{B_Y}(\Y,\Ll_\Y)=\langle \Ll_\Y^n \rangle\big(K_{(\Y,\CB_\Y)/\PP^1}-\mu_\Y^*K_{(Y_{\PP^1},B_{Y,\PP^1})}\big)+\langle \Ll_\Y^n \rangle \Big(\mu_\Y^*K_{(Y_{\PP^1},B_{Y,\PP^1})}+\frac{n}{n+1}\Ll_\Y\Big)
\end{equation}
where we also used the orthogonality $\langle \Ll_\Y^{n+1} \rangle=\langle \Ll_\Y^n \rangle\cdot \Ll_\Y$ (\cite[Corollary 3.6]{BFJ09}). Next, by \cite[Corollary 4.12]{BHJ17}
\begin{equation}
    \label{eqn:K1}
    K_{(\Y,\CB_\Y)/\PP^1}-\mu_\Y^*K_{(Y_{\PP^1},B_{Y,\PP^1})}=\sum_E b_E A_{(Y,B_Y)}(v_E) E
\end{equation}
where the sum runs over all irreducible components of $\Y_0$, $b_E:=\ord_E (\Y_0)$ and $v_E:=b_E^{-1}r(\ord_E)$ is the valuation on $\C(X)$ induced by $\ord_E$ (subsection \ref{ssec:NA}). On the other hand, by definition
\begin{equation}
    \label{eqn:K2}
    \mu_\Y^*K_{(Y_{\PP^1},B_{Y,\PP^1})}+\frac{n}{n+1}\Ll_\Y=G_\Y-\frac{1}{n+1}\Ll_\Y.
\end{equation}
Thus, setting $E^\NA(\Y,\Ll_\Y):=\frac{\langle \Ll_\Y^{n+1} \rangle}{(n+1)V_Y}$, combining (\ref{eqn:K1}), (\ref{eqn:K2}) we get
\begin{multline*}
    DF_{B_Y}(\Y,\Ll_\Y)=V_Y^{-1}\Big(\sum_E b_E A_{(Y,B_Y)}(v_E) E\cdot \langle \Ll_\Y^n\rangle+G_\Y\cdot \langle \Ll_\Y^n\rangle\Big)-E^\NA(\Y,\Ll_\Y)=\\
    =V_Y^{-1}\sum_E b_E E\cdot \langle \Ll_\Y^n\rangle\big(A_{(Y,B_Y)}+\varphi_{(\Y,\Ll_\Y)}\big)(v_E)-E^\NA(\Y,\Ll_\Y)
\end{multline*}
where in the last equality we define $\varphi_{(\Y,\Ll_\Y)}(\nu):=G(\nu)(G_\Y)$, the \emph{Non-Archimedean metric associated to $(\Y,\Ll_\Y)$} (subsection \ref{ssec:NA}). In particular, setting $c_E:=b_E E\cdot \langle \Ll_\Y^n \rangle$, we have $c_E\geq 0$ and
$$
V_Y^{-1}\sum_E c_E=V_Y^{-1}\Y_0 \cdot \langle \Ll_\Y\rangle=V^{-1}_Y\Y_1\cdot \langle \Ll_\Y \rangle=1 
$$
using also Lemma \ref{lem:Below}. Moreover, by an immediate generalization of Lemma \ref{lem:Indu}, we also have that $\varphi_{(\Y,\Ll_\Y)}\geq U_\NA^{\Ll_\Y}$ (see the discussion before Proposition \ref{prop:LF} for the definition of $U_\NA^{\Ll_\Y}$). Hence, we deduce that
\begin{multline*}
    DF_{B_Y}(\Y,\Ll_\Y)\geq \min_E \big\{A_{(Y,B_Y)}+\varphi_{(\Y,\Ll_\Y)}\big\}-E^\NA(\Y,\Ll_\Y)\geq\\
    \geq \inf_{Y^\divv}\big\{A_{(Y,B_Y)}+\varphi_{(\Y,\Ll_\Y)}\big\}-E^\NA(\Y,\Ll_\Y)\geq \inf_{Y^\divv}\big\{A_{(Y,B_Y)}+U_\NA^{\Ll_\Y}\big\}-E^\NA(\Y,\Ll_\Y)=D_{B_Y}^\NA(\Y,\Ll_\Y)
\end{multline*}
which concludes the proof.
\end{proof}

\section{The Ding version of the uniform Yau-Tian-Donaldson conjecture}
\label{sec:YTD}
In this section we will assume the following setting:
\begin{itemize}
    \item[i)] $(X,\omega)$ is pair where $X$ is a Fano manifold and $\omega$ is a K\"ahler form in $c_1(X)$;
    \item[ii)] $\psi\in \cM^+_{D,klt}(X,\omega)$, i.e. $\psi$ is a model type envelope given as decreasing limit of algebraic model type envelopes, $\int_X MA_\omega(\psi)>0$ and $(X,\psi)$ is klt (i.e. $\mathcal{I}(\psi)=\mathcal{O}_X$, see subsection \ref{ssec:KE}).
\end{itemize}
\begin{lemma}
\label{lem:bKlt}
Let $\{\psi_k\}_{k\in\N}\subset \cM^+_D(X,\omega)$ be a decreasing sequence converging to $\psi\in \cM_D(X,\omega)$. Then
$$
\lct_X(\psi_k):=\sup\big\{c>0\, : \, \mathcal{I}(c\psi_k)=\mathcal{O}_X\big\}\searrow \lct_X(\psi).
$$
\end{lemma}
$\lct_X(\psi)$ is also called \emph{complex singularity exponent} (see \cite{DK99}).
\begin{proof}
By monotonicity we clearly have
$$
\lct_X(\psi)\leq \inf_{k\in\N}\lct_X(\psi_k).
$$
Moreover, by the well-known correspondence between the analytic version of log canonical thresholds and its algebraic version we have
$$
\lct_X(\psi)=\sup\big\{c>0\, : \, A_X(E)>c\,\nu(\psi,E) \, \forall E \, \mathrm{prime}\, \mathrm{divisor}\, \mathrm{over}\, X\big\}
$$
where $\nu(\psi, E)$ is the Lelong number of $\psi$ over $E$ while \emph{"$E$ prime divisor over X"} means that $E$ is a prime divisor on  $Y\geq X$. Thus if by contradiction
$$
\lct_X(\psi)< a<b< \inf_{k\in\N}\lct_X(\psi_k)
$$
then there exists a prime divisor $E$ over $X$ such that
$$
0<A_X(E)\leq a\nu(\psi,E)< b\nu(\psi,E).
$$
On the other hand, we also have $\nu(\psi_k,E)\searrow \nu(\psi,E)$ by \cite[subsection 3.1]{Tru20c} since $\psi_k\in \cM^+_D(X,\omega)$. Hence we get the contradiction letting $k\to \infty$ in
$$
A_X(E)> b \nu(\psi_k,E).
$$
\end{proof}
We recall that, for a pair $(Z,\Delta)$ and an effective $\Q$-divisor $D$, the log canonical threshold of $D$ with respect to $(Z,\Delta)$ is given as
$$
\lct_{(Z,\Delta)}(D)=\inf_{E/Z} \frac{A_{(Z,\Delta)}(E)}{\ord_E(D)}.
$$
Such definition can be extended to $\R$-divisors $\Delta, D$ and to the case when they are a formal countable sum of divisors as in Definition \ref{defi:genbdiv}.
\begin{defi}
Let $\bD\in\bDiv_L(X)$. We define the \emph{log-canonical threshold} of $D$ with respect to $X$ as
$$
\lct_X(\bD):=\inf_{Y\geq X} \lct_{(Y,B_Y)}(D_Y)+1
$$
where $B_Y=D_Y-K_{Y/X}$, and we say  that $(X,\bD)$ is \emph{klt} if $\lct_X(\bD)>1$.
\end{defi}
This definition is justified by the following result.
\begin{lemma}
Let $\bD\in \bDiv_L(X)$ and let $\psi\in\cM^+_D(X,\omega)$ be the associated model type envelope. Then
$$
\lct_X(\bD)=\lct_X(\psi).
$$
In particular $(X,\psi)$ is klt if and only if $(X,\bD)$ is klt.
\end{lemma}
\begin{proof}
It immediately follow from Lemma \ref{lem:bKlt} observing that
$$
\lct_X(\psi_Y)=\lct_{(Y,B_Y)}(D_Y)+1
$$
as $A_{(Y,B_Y)}(E)=A_X(E)-\ord_E(D_Y)$ for any $E$ prime divisor above $Y$, for any $Y\geq X$.
\end{proof}
\subsection{Variational approach for KE metrics with prescribed singularities}
\label{ssec:Variational}
We here collect some results from \cite{Tru20a, Tru20c} which represent the analytic part in the proof of Theorem \ref{thmD}.
\subsubsection{The metric space $\cE^1_\psi$}
To solve the complex Monge-Ampère equation associated to $[\psi]$-KE metrics through a variational method, one first enlarges the space of solutions seeking for a \emph{weak} solution.\newline

We set $V_\psi:=\int_X MA_\omega(\psi)>0$, and we say that $u\preccurlyeq\psi$ has $[\psi]$\emph{-relative minimal singularities} if $\lvert\psi-u \rvert \leq C$ for a constant $C$.
\begin{defi}{\cite[section 4.2]{DDNL17b}}
For any $u\in \PSH(X,\omega)$, $u\preccurlyeq\psi$, the $[\psi]$-relative Monge-Ampère energy  $E_\psi: \PSH(X,\omega,\psi):=\{u\in \PSH(X,\omega)\, : \, u\preccurlyeq\psi\}\to \R\cup \{-\infty\}$ is defined as
$$
E_\psi(u):= \frac{1}{(n+1)V_\psi}\sum_{j=0}^n(u-\psi)\langle (\omega+dd^c\psi)^j\wedge (\omega+dd^c u)^{n-j} \rangle
$$
if $u$ has $[\psi]$-relative minimal singularities and as
$$
E_\psi(u):=\inf\big\{E_\psi(v)\, :\, v  \,\mathrm{with}\, [\psi]\mathrm{-relative}\, \mathrm{minimal}\, \mathrm{singularities},\, u\leq v \big\}
$$
otherwise.\newline
The set of of $[\psi]$\emph{-relative finite energy} is then defined as
$$
\cE_\psi^1:=\{u\in \cE(X,\omega,\psi)\, :\, E_\psi(u)>-\infty\}.
$$
\end{defi}
We refer the reader to \cite{DDNL17b, Tru19, Tru20a} and references therein for many interesting properties of the $[\psi]$-relative Monge-Ampère energy, while in the following proposition we summarize some results which will be useful in the sequel of the paper.
\begin{theorem}
\label{thm:E1}
\begin{itemize}
    \item[i)] $\cE_\psi^1$ becomes a complete geodesic metric space when endowed of the distance $d_1(u,v):=E_\psi(u)+E_\psi(v)-2E_\psi\big(P_\omega(u,v)\big)$ \cite[Theorem A]{Tru19}, \cite[Theorem 1.2]{X19};
    \item[ii)] any two elements $u,v\in \cE_\psi^1$ can be joined by a psh geodesic path which is geodesic in $(\cE^1_\psi,d_1)$ \cite[Proposition 3.9]{Tru20c};
    \item[iii)] the $\psi$-relative Monge-Ampère energy $E_\psi$ is linear along psh geodesic paths \cite[Theorem 3.11]{Tru20c}.
\end{itemize}
\end{theorem}
Next, we observe that the $d_1$ distance is not invariant under the clear translation action on $\cE^1_\psi$. Thus, its translation invariant version is given by the $[\psi]$-relative functional $J_\psi:\cE^1_\psi\to \R$,
$$
J_\psi(u):=\frac{1}{V_\psi}\int_X (u-\psi)MA_\omega(\psi)-E_\psi(u),
$$
as $d_1(\psi,u)-C\leq J_\psi(u)\leq d_1(\psi,u) $ for a uniform constant $C=C(X,\omega)$ \cite[Lemma 3.7]{Tru20b}.
\subsubsection{The $[\psi]$-relative Ding functional}
Recalling that by $\mu$ we denoted the probability adapted measure associated to $\omega$ (subsection \ref{ssec:KE}), we then introduce the following functional.
\begin{defi}{\cite[section 3.2]{Tru20b}, \cite[section 4.2]{Tru20c}}
The \emph{$[\psi]$-relative Ding functional} $D_\psi:\cE^1\to \R$ is the translation invariant defined as $D_\psi(u):=L(u)-E_\psi(u)$ where
$$
L(u):=-\frac{1}{2}\log \int_X e^{-2u}d\mu.
$$
\end{defi}
Note that $D_\psi(u)\in\R$ for any $u\in\cE^1_\psi$ as immediate consequence of \cite[Proposition 3.10]{Tru20b}.\newline
An important feature of the $\psi$-relative Ding functional is the following convexity property which extends \cite[Theorem 1.1]{Bern15} to the prescribed singularities setting.
\begin{prop}{\cite[Theorem 4.19]{Tru20c}}
\label{prop:ConveDing}
The functional $L$ is convex along $[\psi]$-relative psh paths contained in $\cE^1_\psi$. In particular the $\psi$-relative Ding functional is convex along $[\psi]$-relative psh geodesic paths contained in $\cE^1_\psi$.
\end{prop}
Define $\mathrm{Aut}(X,[\psi])^{\circ}:=\mathrm{Aut}(X)^\circ\cap \mathrm{Aut}(X,[\psi])$ where $\mathrm{Aut}(X)^\circ$ is the connected component of the identity while $\mathrm{Aut}(X,[\psi]):=\{F\in \mathrm{Aut}(X)\, : \, F^*\psi-\psi\in L^\infty\}$ is the group of automorphisms that preserve the singularities of $\psi$.
\begin{theorem}{\cite[Theorems C, D]{Tru20c}}
\label{thm:Ding}
The followings are equivalent:
\begin{itemize}
    \item[i)] $\omega+dd^c u$ is a $[\psi]$-KE metrics;
    \item[ii)] $D_\psi(u)=\inf_{\cE^1_\psi}D_\psi$.
\end{itemize}
Moreover, assuming $\mathrm{Aut}(X,[\psi])^\circ=\{\Id\}$, we also have uniqueness of $[\psi]$-KE metrics and the followings are equivalent:
\begin{itemize}
    \item[iii)] there exists a unique $[\psi]$-KE metric $\omega+dd^c u$;
    \item[iv)] there exists $A>0, B\geq 0$ such that
    $$
    D_\psi(u)\geq AJ_\psi(u)-B
    $$
    for any $u\in\cE^1_\psi$.
\end{itemize}
\end{theorem}
\subsection{$[\psi]$-KE metric $\Longrightarrow$ uniform $\bD$-log Ding stability}
In this subsection we show how the existence of a unique $[\psi]$-KE metric implies the uniform $\bD$-log Ding stability, which is the implication $(i)\Rightarrow (ii)$ in Theorem \ref{thmD}.

\subsubsection{Relative Deligne pairing}
As we want to consider Deligne pairing with respect to possible singular reference metrics, we give the following definition.
\begin{defi}
\label{defi:DP}
Let $p:\X\to \Y$ be a flat surjective morphism of relative dimension $n$ among  smooth complex algebraic varieties. Let $\Theta_0,\dots, \Theta_n$ be smooth and closed $(1,1)$-forms on $\X$ representatives of \emph{relatively big} cohomology classes (i.e. $\{\Theta_{j|\X_y}\}$ is big for any $y\in \Y$) and let $V_j\in \PSH(\X,\Theta_j)$. Then the \emph{positive Deligne pairing} of $(\Theta_j,V_j,U_j)$, for $U_j\in\PSH(\X,\Theta_j)$ such that $\lvert U_{j|\X_y}- V_{j|\X_y}\rvert\leq C_y$ for any $y\in \Y$, is defined as
\begin{multline*}
    \langle\langle U_0,\dots,U_n \rangle\rangle_{(\Theta_{0,V_0}, \dots, \Theta_{n,V_n})}:=p_*\Big((U_0-V_0)\langle \Theta_{1,U_1}\wedge\cdots \wedge \Theta_{n,U_n} \rangle\Big)+\\
    +p_*\Big((U_1-V_1)\langle \Theta_{0,V_0}\wedge\Theta_{2,U_2}\wedge\cdots \wedge \Theta_{n,U_n} \rangle\Big)+\ldots+p_*\Big((U_n-V_n)\langle \Theta_{0,V_0}\wedge\cdots \wedge \Theta_{n-1,V_{n-1}} \rangle\Big)
\end{multline*}
\end{defi}
We used the notations $\Theta_{i,U_j}:=\Theta_i+dd^cU_j$.\newline
We will simply write $\langle\langle U_0,\dots,U_n \rangle\rangle$ when the references forms $\Theta_j$ and the $\Theta_j$-psh functions $V_j$ are clear from the context.
\begin{example}
\label{exa:Fun}
We can recover the $[\psi]$-relative Monge-Ampère energy and the $J_\psi$-functional in terms of positive Deligne pairing. Indeed, letting $\X=X$, $\Y$ be a point, $\Theta_j=\omega$ and $V_j=\psi$ for any $j=0,\dots,n$, we have
\begin{gather*}
    E_\psi(u)=\frac{1}{(n+1)V_\bD} \langle\langle u^{n+1}\rangle\rangle,\\
    J_\psi(u)=V_\bD^{-1}\langle\langle u,\psi^n \rangle \rangle-E_\psi(u)
\end{gather*}
for any $u\in \cE^1_\psi$ with $\psi$-relative minimal singularities. It is also possible to extend the description of these functionals in terms of positive Deligne pairings to the whole space $\cE^1_\psi$ (see also \cite[Lemma 3.8]{Tru20a} and \cite{Reb21}).
\end{example}
\subsubsection{Slope of $J$ and $E$}
\label{ssec:JandE}
We note that given $n$ $[\psi]$-relative psh rays $U_0,\dots, U_n:\R_{\geq 0}\to \PSH$ such that $u_{j,t}$ has $\psi$-relative minimal singularities for any $t\in[0,+\infty)$ fixed and for any $j=0,\dots,n$, the positive Deligne pairing
$$
\langle \langle U_0,\dots,U_n \rangle \rangle(t):=\langle \langle U_0,\dots,U_n \rangle \rangle_{p_X^*(\omega_\psi),\dots,p_X^*(\omega_\psi)}(t)
$$
for $t\in [0,+\infty)$ is well-defined. In terms of Definition \ref{defi:DP}, we are considering $\X=X\times\D^*$, $\Y=\D^*$ and $t=-\log \lvert \tau \rvert$. If the rays $U_j$ are induced by compactifications of test configurations, then $\langle \langle U_0,\dots,U_n\rangle \rangle $ is also well-defined over $\Y=\PP^1\setminus\{0\}$ as the $\C^*$-action acts trivially at $\infty\in \PP^1$.
\begin{prop}
\label{prop:Slope}
For any $j=0,\dots,n$ let $\bD_j\in \bDiv_L(X)$ and let $\Ll_j\in N^1(\FX_{\PP^1}^{\C^*})$ be a test configuration class such that $\Ll_j-\CD_j$ is big and $\Ll_j-L\times\PP^1$ is psef. Let also $(\X,\Ll_j)\in \FX_{\Ll_j}$. \newline
If $T^{(\X,\Ll_j,\bD_j)}\in c_1(\Ll_j)$ are the closed and positive $(1,1)$-currents given by Proposition \ref{prop:Current} then
\begin{equation}
    \int_\X \langle T^{(\X,\Ll_0,\bD_0)}\wedge \cdots \wedge T^{(\X,\Ll_n,\bD_n)} \rangle=\lim_{t\to+\infty} \frac{\langle\langle U_0,\dots,U_n\rangle \rangle(t)}{t}
\end{equation}
where $U_j:\R\to \PSH$ are the psh rays associated to $T^{(\X,\Ll_j,\bD_j)}$.
\end{prop}
\begin{proof}
As said in the Introduction the following proof is an adaptation to the prescribed singularities setting of \cite[Theorem 5.4]{DR22}.\newline
Denote by $\psi_j:=\psi_{\bD_j}\in \cM^+(X,\omega)$ the model type envelopes given by Proposition \ref{prop:Corre}, and to lighten the notation we also set $T_j:=T^{(\X,\Ll_j,\bD_j)}$.\newline
Since $\Ll_j-\mu_\X^*p_X^*L$ is effective, the $\omega$-psh functions $\psi_j$ induce natural $\Omega_j$-psh functions $V_j$ where $\Omega_j$ is a fixed smooth and closed $(1,1)$-forms in $c_1(\Ll_j)$. In particular, if $\Psi_j$ is chosen so that $T_j=\Omega_j+dd^c \Psi_j$ then by construction of $T_j$ we have
$$
V_j\leq \Psi_j+C_j
$$
over $\X$ for constant $C_j\in\R$.\newline
Next, letting $f\in C^{\infty}(\PP^1)$ with compact support in $\C^*$, we have
\begin{equation}
    \label{eqn:IP1}
    \int_{\C^*} fdd^c\langle \langle U_0,\dots,U_n\rangle \rangle=\int_{\C^*} \langle \langle U_0,\dots,U_n \rangle \rangle dd^c f =\int_{X\times \C^*}dd^c f \wedge \sum_{j=0}^n S_j
\end{equation}
where $S_j:=(\Psi_j-V_j)\langle \Omega_{0,V_0}\wedge \cdots \wedge \Omega_{j-1,V_{j-1}}\wedge \Omega_{j+1,\Psi_{j+1}}\wedge \cdots \wedge \Omega_{n,\Psi_n} \rangle. $. Then, since $U_j$ has $\psi_j$-linear growth (Lemma \ref{lem:Indu}) we also have $\Psi_j+A\log |\tau|\leq V_j$ for $A>0$ over $\X_{\pi^{-1}\D^*}$. Thus by a simple maximum argument we can construct $\tilde{V}_j\in \PSH(\X,\Omega_j)$ with $[\Psi_j]$-relative minimal singularities such that $\tilde{V_j}=V_j$ on $X\times (\C^*\setminus D_a)$ where $D_a=\{\tau\in\C^*\, : \, \lvert \tau \rvert\leq a\}$ for $0<a<1$ small enough so that the support of $f$ is contained in $\C^*\setminus D_a$. This simple observation is needed to apply integration by parts in (\ref{eqn:IP1}) (see \cite[Theorem 1.2]{Lu20} and reference therein). Therefore
\begin{multline*}
    \int_{X\times\C^*}dd^c f \wedge \sum_{j=0}^n S_j=\sum_{j=0}^n
    \int_\X f dd^c(\Psi_j-V_j)\wedge\langle  \Omega_{0,V_0}\wedge \cdots \wedge \Omega_{j-1,V_{j-1}}\wedge \Omega_{j+1,\Psi_{j+1}}\wedge \cdots \wedge \Omega_{n,\Psi_n} \rangle=\\
    =\int_\X f \langle \Omega_{0,\Psi_0}\wedge \cdots \wedge \Omega_{n,\Psi_n}\rangle- \int_\X f \langle \Omega_{0,V_0}\wedge \cdots \wedge \Omega_{n,V_n} \rangle 
\end{multline*}
where the last equality follows by the linearity of the non-pluripolar product writing $dd^c (\Psi_j-V_j)=\Omega_{j,\Psi_j}-\Omega_{j,V_j}$ and observing that the sum is telescopic. Moreover, since $\Omega_{j,V_j}=p_X^*\omega+dd^c\psi_j$ over $X\times\C^*$ and the non-pluripolar product has no mass on pluripolar sets, we also have $\langle \Omega_{0,V_0}\wedge \cdots \wedge \Omega_{n,V_n}\rangle =0$. Summarizing we got
$$
\int_{\C^*}f dd^c \langle\langle U_0,\dots,U_n \rangle\rangle=\int_\X f \langle T_0\wedge \cdots\wedge T_n  \rangle
$$
for any $f\in C^{\infty}(\PP^1)$ with compact support in $\C^*$. Letting $f\nearrow 1$, it follows that
$$
\int_\X \langle T_0\wedge \cdots\wedge T_n  \rangle=\int_{\C^*}dd^c\langle \langle U_0,\dots,U_n\rangle \rangle=\lim_{t\to +\infty} \frac{\langle \langle U_0,\dots,U_n \rangle \rangle}{t},
$$
where we clearly set $t=-\log \lvert \tau\rvert$ and where the last equality is obtained proceeding exactly as in \cite[Lemma 2.6]{Berm16}.
\end{proof}
\begin{corollary}
\label{cor:SlopeEJ}
Let $\bD\in\bDiv_L(X)$, let $\Ll\in N^1(\FX_{\PP^1}^{\C^*})$ be a test configuration class and let $(\X,\Ll)\in \FX_\Ll$. Then for any $m\gg 1$ big enough
\begin{gather}
    \label{eqn:Energy}
    \lim_{t\to +\infty} \frac{E_\psi\big(u_t^{(\X,\Ll+m\X_0,\bD)}\big)}{t}=E^\NA\big(\Ll+m(X\times\{0\});\bD\big),\\
    \label{eqn:J-Functional}
    \lim_{t\to +\infty}\frac{J_\psi\big(u_t^{(\X,\Ll+m\X_0,\bD)}\big)}{t}=J^\NA\big(\Ll;\bD\big),
\end{gather}
where $U^{(\X,\Ll+m\X_0,\bD)}:\R_{\geq 0}\to \PSH$ is the $[\psi]$-relative psh ray given by Corollary \ref{cor:Abo}.
\end{corollary}
See Definition \ref{defi:JandE} for $E^\NA, J^\NA$.
\begin{proof}
Let $m>m_0$ where $m_0$ is given by Lemma \ref{lem:Effe}, and set $u_t:=u_t^{(\X,\Ll+m\X_0,\bD)}$ for simplicity. Then the equality (\ref{eqn:Energy}) directly follows from Proposition \ref{prop:Slope} since
$$
\int_\X \langle T^{(\X,\Ll+m\X_0,\bD),n+1}\rangle=\langle \big(\Ll+m(X\times\{0\})-\CD\big)^{n+1}\rangle
$$
by Proposition \ref{prop:Current}. Regarding (\ref{eqn:J-Functional}), although $L\times\PP^1-\CD$ is just pseudoeffective (but not big) we can still proceed exactly as in Proposition \ref{prop:Slope} to obtain
$$
\int_\X \langle (\mu_\X^*p_X^*\omega_\psi)^n\wedge T^{(\X,\Ll+m\X_0,\bD)} \rangle=\lim_{t\to +\infty} \frac{\langle \langle U^{(\X,\Ll+m\X_0,\bD)}, \psi^n \rangle \rangle}{t}.
$$
Then adapting the proof of Proposition \ref{prop:Current} it is not hard to check that
$$
\int_\X \langle (\mu_\X^*p_X^*\omega_\psi)^n\wedge T^{(\X,\Ll+m\X_0,\bD)} \rangle=\langle \big((L-\bD)\times\PP^1\big)^n, \Ll_m \rangle
$$
where $\Ll_m:=\Ll+m(X\times\{0\})-\CD\in N^1(\FX_{\PP^1}^{\C^*})$. Hence (\ref{eqn:J-Functional}) follows by linearity.
\end{proof}
\subsubsection{Log canonical thresholds of $(\Ll,\bD)$.}
\begin{prop}
\label{prop:L}
Let $U:\R_{\geq 0}\to \cE^1_\psi$ be a finite energy $[\psi]$-relative psh ray with linear growth. Then $L^\NA (U_\NA)$ is finite and coincide with
\begin{equation}
    \label{eqn:Star}
    \sup\Big\{c\in\R\, |\, \int_1^\infty e^{2\big(ct-L(u_t)\big)}dt<+\infty\Big\}
\end{equation}
and with $\lim_{t\to +\infty}\frac{L(u_t)}{t}$. \newline
In particular, if $\bD\in\bDiv_L(X),\Ll\in N^1(\FX_{\PP^1}^{\C^*})$ such that $\Ll-\CD$ is big and $\Ll-L\times\PP^1$ is psef, then for any $(\X,\Ll)\in \FX_\Ll$
$$
L^\NA(\Ll;\bD)=\lim_{t\to +\infty}\frac{L(u_t^{(\X,\Ll,\bD)})}{t}.
$$
\end{prop}
\begin{proof}
Since $U$ is in particular a psh ray, once that Lemma \ref{lem:Below2} below is proved, the proof proceeds exactly as that in \cite[Theorem 5.4]{BBJ15}, and hence it will be left to the reader.
\end{proof}
\begin{lem}
\label{lem:Below2}
Let $W$ be the set of all $\C^*$-invariant divisorial valuations $w$ on $X\times\C$ such that $w(\tau)=1$, and let $U:\R_{\geq 0}\to \cE^1_\psi$ be a finite energy $[\psi]$-relative psh ray such that $\sup_X U\leq 0$. Then there exists $\epsilon\in (0,1)$ and $C>0$ such that $w(U)\leq(1-\epsilon)A_{X\times\C}(w)+C$ for all $w\in W$.
\end{lem}
\begin{proof}
As $u_t\in \cE^1_\psi$ for any $t\in \R_{\geq 0}$ and $(X,\psi)$ is klt, there exists $\epsilon>0$ such that $e^{-u_t}\in L^{2(1+\epsilon)}(X)$ for any $t\in \R_{\geq 0}$ (see also \cite[Proposition 3.10]{Tru20b}). Thus by \cite[Proposition 2.2]{DK99} it follows that $e^{-U}$ is in $L^{2(1+\epsilon)}$ in a neighborhood of $(x,t)$ for any $(x,t)\in X\times\D^*$, i.e. $e^{-U}\in L^{2(1+\epsilon)}_{loc}(X\times\D^*)$.\newline
Therefore $\mathcal{I}\big((1+\epsilon)U\big)$ is cosupported on $X\times\{0\}$, which implies that $\sup_{w\in W}w\Big(\mathcal{I}\big((1+\epsilon)U\big)\Big)<+\infty$ since $\mathcal{I}\big((1+\epsilon)U\big)$ contains a power of $\tau$. The proof concludes observing that 
$$
w\Big(\mathcal{I}\big(1+\epsilon\big)U\Big)\geq (1+\epsilon) w(U)-A_{X\times\C}(w)
$$
\cite[Lemma B.4]{BHJ17}.
\end{proof}
\begin{remark}
\label{rem:Llct}
Here we interpret $L^\NA(\Ll;\bD)$ in terms of log canonical thresholds similarly to \cite{Berm16}, assuming that $\Ll-\CD$ is big and that $\Ll-L\times\PP^1$ is psef. Fix $(\X,\Ll)\in \FX_\Ll$.\newline
We need to recall some definitions/facts from \cite[section 4]{DR22}. As $\pi_*\big(r(\Ll+K_{\X/\PP^1})$ is a line bundle for $r\in\N$ divisible enough, we first construct a $\Q$-divisor $\Delta$ on $\X$ such that $K_\X+\Delta=-\Ll+K_{\PP^1}$, then we set
$$
\lct\big((\X,\Delta,\lVert \Ll\rVert), \X_0\big):=\sup\big\{c>0\, : \, \mathcal{I}\big((\X,\Delta,\lVert \Ll\rVert),c\X_0\big)=\mathcal{O}_\X\big\}
$$
where
\begin{multline*}
    \mathcal{I}\big((\X,\Delta,\lVert \Ll\rVert),c\X_0\big)=\big\{f\in \mathcal{O}_\X(U)\, : \, \exists \epsilon>0\,\mathrm{s.t.}\, \ord_E(f)\geq\\
    \geq(1+\epsilon)\big(\ord_E\lVert \Ll\rVert+c\,\ord_E\X_0-A_{(\X,\Delta)}(E)\big)\, \mathrm{for}\, \mathrm{all}\, E\, \mathrm{over}\, X\big\}.
\end{multline*}
As in \cite{DR22} $\ord_E\lVert\Ll\rVert=\lim_{k\to +\infty}\frac{1}{k}\ord_E H^0(\X,k\Ll)$. Analytically, we have
$$
\lct\big((\X,\Delta,\lVert \Ll\rVert),\X_0\big)=\sup\big\{c>0\, : \, e^{-v_{min}-c\log \lvert \tau\rvert}\in L^2(\mu_{(\X,\Delta)})\big\}
$$
where $v_{min}$ is a $\C^*$-invariant $\Omega$-psh function with minimal singularities for $\Omega\in c_1(\Ll)$, and where $\mu_{(\X,\Delta)}$ is an adapted measure to $K_\X+\Delta$ (see also \cite[subsection 2.1.3]{Berm16}). Thus, since (\ref{eqn:Star}) corresponds to
$$
\sup\big\{c>0\, : \, e^{-\Psi-c\log \lvert \tau \rvert}\in L^2(\mu_{(\X,\Delta)})\big\}-1
$$
where $\Psi\in \PSH(\X,\Omega)$ satisfies $T^{(\X,\Ll,\bD)}=\Omega+dd^c \Psi$ (Proposition \ref{prop:Current}), it appears natural to connects it in terms of log canonical thresholds where the singularities are taken under considerations. Indeed, proceeding as in Lemma \ref{lem:bKlt} we have
$$
(\ref{eqn:Star})=\inf_{Y\geq X}\sup\big\{c>0\, : \, e^{-\Psi^Y-c\log \lvert \tau \rvert}d\mu_{(\X,\Delta)}\big\}-1
$$
where $\Psi^Y\in \PSH(\X,\Omega)$ such that $\Omega+dd^c \Psi^Y=T^\Y$ for $\Y$ given as the graph of the birational morphism $Y\times\PP^1\dashrightarrow \X$ (see again Proppsition \ref{prop:Current} for the notations). Then, considering the commutative diagram
$$
\begin{tikzcd}
    \Y\ar[r, "\tilde{\rho_Y}"] \ar[d, "\mu_\Y"] & \X\ar[d, "\mu_\X"]\\
    Y\times\PP^1 \ar[r, "\rho_Y\times\Id"] & X\times\PP^1
\end{tikzcd}
$$
and defining $\Delta_\Y:=\tilde{\rho}_Y^*\Delta+\CD_\Y-K_{\Y/\X}$ (i.e. $K_\Y+\Delta_\Y=\tilde{\rho}_Y^*(K_\X+\Delta)+\CD_\Y=-\Ll_\Y+K_{\PP^1}$), it is easy to check that
$$
\sup\big\{c>0\, : \, e^{-\Psi^Y-c\log\lvert \tau \rvert}d\mu_{(X,\Delta)}\big\}=\lct\big((\Y,\Delta_\Y,\lVert\Ll_\Y \rVert),\Y_0\big).
$$
Hence we get
$$
L^\NA(\Ll;\bD)=\inf_{Y\geq X} \lct\big((\Y,\Delta_\Y,\lVert \Ll_\Y\rVert),\Y_0\big)-1.
$$
Note that the algebraic quantity on the right-hand side \emph{a posteriori} does not depend on the choice of $(\X,\Ll)\in \FX_\Ll$.
\end{remark}
\subsubsection{Conclusion of the proof.}
\begin{theorem}
\label{thm:KEDing}
Let $\psi\in \cM^+_{D,klt}(X,\omega)$ (or equivalently let $\bD\in \bDiv_L(X)$ such that $(X,\bD)$ is klt). Assume also that $\mathrm{Aut}(X,[\psi])^\circ=\{\Id\}$.\newline
If there exists a $[\psi]$-KE metric then $(X,-K_X)$ is uniformly $\bD$-log Ding stable.
\end{theorem}
\begin{proof}
By Theorem \ref{thm:Ding} there exist $A>0$, $B\geq 0$ such that
$$
D_\psi(u)\geq AJ_\psi(u)-B
$$
for any $u\in\cE^1_\psi$. Thus, letting $\Ll\in N^1(\FX_{\PP^1}^{\C^*})$ be an ample test configuration class and letting $(\X,\Ll)\in \FX_\Ll$, we have
$$
D_\psi(u_t^m)\geq A J_\psi(u_t^m)-B
$$
for any $t\in \R_{\geq 0}$ where $U^m:\R_{\geq 0}\to \cE^1_\psi$ is the $[\psi]$-relative psh ray associated to $T^{(\X,\Ll+m\X_0,\bD)}\in c_1(\Ll+m\X_0-\CD)$ for $m\gg 1$ big enough (see Proposition \ref{prop:Current} and Corollary \ref{cor:Abo}).
As by definition $D_\psi(u_t^m)=L(u_t^m)-E_\psi(u_t^m)$, Corollary \ref{cor:SlopeEJ} and Proposition \ref{prop:L} imply
\begin{gather*}
    \lim_{t\to +\infty}\frac{D_\psi(u_t^m)}{t}=\lim_{t\to +\infty}\frac{L(u_t^m)-E_\psi(u_t^m)}{t}=L^\NA\big(\Ll+m(X\times\{0\});\bD\big)-E^\NA\big(\Ll+m(X\times\{0\});\bD\big)\\
    \lim_{t\to +\infty}\frac{J_\psi(u_t^m)}{t}=J^\NA\big(\Ll;\bD\big).
\end{gather*}
Hence
$$
D^\NA(\Ll;\bD)=L^\NA\big(\Ll+m(X\times\{0\});\bD\big)-E^\NA\big(\Ll+m(X\times\{0\});\bD\big)\geq A J^\NA\big(\Ll;\bD\big).
$$
\end{proof}
The following corollary concludes the proof of Theorem \ref{thmC}.
\begin{corollary}
With the same assumptions of Theorem \ref{thm:KEDing}, if there exists a $[\psi]$-KE metric then $(X,-K_X)$ is uniformly $\bD$-log $K$-stable. 
\end{corollary}
\begin{proof}
It immediately follows combining Theorem \ref{thm:KEDing} and Proposition \ref{prop:DingK}.
\end{proof}
\subsection{$\bD$-log delta invariants}
\label{ssec:deltas}
In this subsection we introduce the $\bD$\emph{-log delta invariant} $\delta_\bD$ and its modified version $\tilde{\delta}_\bD$, generalizing the Odaka-Fujita delta invariant (\cite{FO18, BJ20}).\newline

\begin{defi}
Let $\bD\in \bDiv_L(X)$ and let $E$ be a prime divisor over $X$. The \emph{expected vanishing order of $L-\bD$ along $E$} is then defined as
$$
S_\bD(E):=\frac{1}{\langle(L-\bD)^n\rangle}\int_0^{\tau_\bD(E)}\langle (L-\bD-xE)^n \rangle dx
$$
where $\tau_\bD(E):=\sup\{x\geq 0\, : \, \inf_{Y\geq X}\langle (L-xE-D_Y)^n\rangle\geq 0\}$.
\end{defi}
Letting $\psi\in \cM^+(X,\omega)$ be the model type envelope associated to $\bD$ (Proposition \ref{prop:Corre}), we recall that $\langle (L-\bD)^n \rangle=\inf_{Y\geq X} \langle (L-D_Y)^n\rangle=\int_X MA_\omega(\psi)>0$ by definition of $\bDiv_L(X)$, i.e. $L-\bD$ is \emph{big} in the sense of Definition \ref{defi:WBig}.\newline
We observe that $S_\bD(E)$ can be recovered by $S_{D_Y}(E)$ thanks to the following result.
\begin{lemma}
\label{lem:Dini}
Let $\bDiv_L(X)$, let $E$ be a prime divisor over $X$ and let also $Y_{k+1}\geq Y_k\geq X$ such that $\langle (L-D_{Y_k})^n\rangle\to \langle (L-\bD)^n \rangle$. Then the sequence of continuous functions $f_k: \R_{\geq 0}\to \R $,
$$
f_k(x):=\langle (L-D_Y-xE)^n \rangle
$$
converges uniformly as $k\to +\infty$ to
$$
f(x):=\langle (L-\bD-xE)^n \rangle
$$
where we clearly set $\langle (L-D_Y-xE)^n \rangle=0$ if $x>\tau_{D_{Y_k}}(E)$ and similarly $\langle (L-\bD-xE)^n\rangle=0$ if $L-\bD-xE$ is not big. In particular
\begin{gather*}
    \tau_{D_{Y_k}}(E)\searrow \tau_{\bD}(E)>0\\
    S_{D_{Y_k}}(E)\to S_\bD(E)>0.
\end{gather*}
\end{lemma}
\begin{proof}
By monotonicity of $\{D_Y\}_{Y\geq X}$ and continuity of the volume function, we immediately get that $f_k$ is a decreasing sequence of continuous functions. \newline
If $x\in \R_{\geq 0}$ is such that $f_k(x)=0$ for some $k\gg 1$ then clearly $f(x)=0$. Assume instead that $f_k(x)>0$ for any $k\in \N$. Then as consequence of Proposition \ref{prop:Corre} it is easy to see that
$$
f_k(x)=\langle (L-D_{Y_k}-xE)^n \rangle=\int_X MA_\omega(\psi_{k,x})
$$
where
$$
\psi_{k,x}:=\sup\{u\leq \psi_{Y_k}\, : \, \nu(u,E)\geq \nu(\psi_{Y_k},E)+x\}\in \cM^+(X,\omega).
$$
Thus, by monotonicity of $\psi_{Y_k}$ we obtain that $\psi_{k,x}\searrow \psi_x\in \cM(X,\omega)$ model type envelope such that $\psi_x\leq \psi$ (as $\psi_{k,x}\leq \psi_{Y_k}\searrow \psi$). Furthermore, as $\nu(\psi_{Y_k},E)\searrow \nu(\psi,E)$, the upper semicontinuity of the Lelong numbers yields $\nu(\psi_x,E)\geq \nu(\psi,E)+x$, which in turn implies
$$
\psi_x=\sup\{u\leq \psi\, : \, \nu(u,E)\geq \nu(\psi,E)+x\}. 
$$
In particular by the continuity of the total mass function along decreasing sequences in $\cM^+(X,\omega)$ (subsection \ref{ssec:KE})
$$
\lim_{k\to \infty}f_k(x)=\lim_{k\to +\infty}\int_X MA_\omega(\psi_{k,x})=\int_X MA_\omega(\psi_x)
$$
Then we claim that $\int_X MA_\omega(\psi_x)=\langle (L-\bD-xE)^n\rangle=f(x)$. Indeed, we can approximate $\bD$ first by $D_{Y_k}$ and then any $D_{Y_k}$ by an increasing sequence of $\R$-divisors to obtain $\bD$ as increasing limit of a sequence of $\R$-divisors $\{G_n\}_{n\in\N}\subset \bDiv_L(X)$ over $X$. Thus Proposition \ref{prop:Corre} and \cite[Lemma 3.12]{Tru20a} leads to $\int_X MA_\omega(\psi_x)=\lim_{n\to +\infty}\langle (L-G_n-xE)^n \rangle=\langle (L-\bD-xE)^n \rangle$, and the claim is proved.\newline
Dini's Lemma concludes the proof since $f:\R_{\geq 0}\to \R$,
$$
f(x)=\langle (L-\bD-xE)^n \rangle=
\begin{cases}
\int_X MA_\omega(\psi_x) \, \, \,\,\, \mathrm{if}\,\,\, 0\leq x\leq \tau_\bD(E)\\
0 \,\,\,\,\,\,\,\,\,\,\,\,\,\,\,\,\,\,\,\,\,\,\, \,\,\,\,\,\,\,\,\,\,\,\,\,\mathrm{if} \,\,\, x> \tau_\bD(E)
\end{cases}
$$
is continuous as consequence of \cite[Lemma 3.12]{Tru20a}.
\end{proof}
The following definitions extend the Fujita-Odaka delta invariant to the $\bD$-log setting.
\begin{defi}
Let $\bD\in \bDiv_L(X)$. The \emph{modified $\bD$-log delta invariant} is defined as
$$
\tilde{\delta}_\bD:=\inf_{E/X} \frac{A_X(E)}{S_\bD(E)+\ord_E\bD}
$$
where the infimum is over all prime divisor $E$ over $X$ and where
$$
\ord_E \bD:=\sup_{Y\geq X} \ord_E D_Y\in \R_{\geq 0}
$$
is the order of vanishing of $\bD$ along $E$.
\end{defi}
\begin{defi}
\label{defi:delta}
Let $\bD\in \bDiv_L(X)$. The \emph{$\bD$-log delta invariant} is defined as $\delta_\bD=0$ if $(X,\bD)$ is not klt and as
$$
\delta_\bD:=\inf_{E/X} \frac{A_X(E)-\ord_E \bD}{S_\bD(E)}
$$
if $(X,\bD)$ is klt, where the infimum is over all prime divisor $E$ over $X$.
\end{defi}
Observe that $\ord_E\bD=\nu(\psi,E)$ where $\psi\in\cM^+(X,\omega)$ is associated to $\bD\in\bDiv_L(X)$.\newline 
Moreover Definition \ref{defi:delta} implies that $\delta_\bD\geq 0$, as $\lct(X,\bD)>1$ if and only if there exists $\epsilon>0$ such that
$$
A_X(E)-\ord_E\bD=A_X(E)-\nu(\psi,E)\geq \epsilon \nu(\psi,E)
$$
for any $E$ prime divisor over $X$ \cite[Theorem B.5]{BBJ15}.
\begin{remark}
\label{rem:Deltas}
When $\bD=D_Y$ for a $\Q$-divisor $D_Y$ such that on $Y\overset{\rho_Y}{\geq} X$ such that $(Y,B_Y)$ is klt, $\delta_\bD$ coincides with the usual \emph{log $\delta$-invariant} of $(Y,B_Y)$, i.e. with
$$
\delta_{B_Y}(Y,L_Y):=\inf_{E/Y} \frac{A_{(Y,B_Y)}(E)}{S_{L_Y}(E)}
$$
where $L_Y:=\rho_Y^*L-D_Y$ and $S_{L_Y}(E):=\frac{1}{(L_Y^n)}\int_0^{\tau_{(Y,B_Y)}(E)}\langle (L_Y-xE )^n\rangle$ for $\tau_{(Y,B_Y)}(E):=\sup\{x\geq 0\, : \,\langle (L_Y-xE)^n \rangle\geq 0\}$ is the usual expected vanishing order of $L_Y$ along $E$. Indeed if $E$ is a prime divisor on $Z\geq Y$ then it is immediate to check that
$$
\frac{A_X(E)-\ord_E(\bD)}{S_\bD(E)}=\frac{A_{(Y,B_Y)}(E)}{S_{L_Y}(E)},
$$
which clearly leads to $\delta_\bD\leq \delta_{(Y,B_Y)}$. Then, if $E$ is a prime divisor on $Z\geq X$ but $Z$ does not dominate $Y$, then one can consider the strict trasform $E^s$ of $E$ on $W\geq Y, W\geq Z$ where $W$ is given by the graph of the birational morphism $Y\dashrightarrow Z$. Thus $E^s$ and $E$ define the same valuation in $X^\divv$ (see \cite[subsection B.5]{BBJ15}) and we obtain
$$
\frac{A_X(E)-\ord_E(\bD)}{S_\bD(E)}=\frac{A_{(Y,B_Y)}(E^s)}{S_{L_Y}(E^s)}.
$$
The equality $\delta_\bD=\delta_{B_Y}(Y,L_Y)$ follows.\newline
Similarly
$$
\tilde{\delta}_\bD=\inf_{E/Y}\frac{A_{(Y,B_Y)}(E)+\ord_E(D_Y)}{S_{L_Y}(E)+\ord_E(D_Y)}.
$$
\end{remark}

The following result can be compared to \cite[Proposition 4.14]{Tru20c}.
\begin{prop}
\label{prop:deltacompa}
Let $\bD\in\bDiv_L(X)$ such that $(X,\bD)$ is klt and let $\lct:=\lct_X(\bD)$. 
\begin{itemize}
    \item[i)] If $\tilde{\delta}_\bD <1$ then
    \begin{equation}
        \label{eqn:The1}
        \delta_\bD\leq \tilde{\delta}_\bD\leq \delta_\bD\frac{\lct}{\delta_\bD+\lct-1}.
    \end{equation}
    \item[ii)] If $\tilde{\delta}_\bD\geq 1$ then
    \begin{equation}
        \label{eqn:The2}
        \delta_\bD\geq \tilde{\delta}_\bD\geq \delta_\bD \frac{\lct}{\delta_\bD+\lct-1}.
    \end{equation}
\end{itemize}
In particular $\delta_\bD\geq 1$ (resp. $\delta_\bD>1$) if and only if $\tilde{\delta}_\bD\geq 1$ (resp. $\tilde{\delta}_\bD> 1$).
\end{prop}
\begin{proof}
For any prime divisor $E$ over $X$, set $a_E:=A_X(E), b_E:=S_\bD(E)+\ord_E \bD$, $\nu_E:=\ord_E\bD$, and define $f_E:[0,\ord_E\bD]\to \R_{>0}, f_E(x):=\frac{a_E-x}{b_E-x}$. By definition
\begin{gather*}
    \delta_\bD=\inf_{E/X} f_E(\nu_E)\\
    \tilde{\delta}_\bD=\inf_{E/X} f_E(0).
\end{gather*}
A simple analysis then shows that $f_E$ is strictly increasing (resp. strictly decreasing, constant) iff and only if $f_E(0)>1$ (resp. $f_E(0)<1, f_E(0)=1$). The first inequalities in (\ref{eqn:The1}), (\ref{eqn:The2}) immediately follows.\newline
Assume $\delta_\bD\leq\tilde{\delta}_\bD<1$, and let $\{E_k\}_{k\in\N}$ be a sequence of prime divisors over $X$ such that $1>\delta_k:=f_{E_k}(\nu_{E_k})\searrow \delta_\bD$. Setting $\tilde{\delta}:=\delta_\bD$, we then have
$$
\frac{a_{E_k}}{b_{E_k}}= \delta_k+(1-\delta_k)\frac{\nu_{E_k}}{b_{E_k}}\leq \delta_k+\frac{a_{E_k}}{b_{E_k}}\frac{1-\delta_k}{\lct}
$$
since $a_{E_k}=\delta_k b_{E_k}+(1-\delta_k)\nu_{E_k}$ and since $\nu_{E_k}\leq a_{E_k}/\lct$ by \cite[Theorem B.5]{BBJ15}. Thus,
$$
\tilde{\delta}\leq \frac{a_{E_k}}{b_{E_k}}\leq \delta_k\frac{\lct}{\delta_k+\lct-1},
$$
and the second inequality in (\ref{eqn:The1}) follows letting $k\to +\infty$.\newline
Similarly, assuming $1\leq \tilde{\delta}_\bD\leq \delta_\bD$, let $\{E_k\}_{k\in\N}$ be a sequence of prime divisors over $X$ such that $\tilde{\delta}_k:=f_{E_k}(0)\searrow \tilde{\delta}_\bD$. Setting $\delta:=\delta_\bD$, we then have
$$
\delta\leq \frac{a_{E_k}-\nu_{E_k}}{b_{E_k}-\nu_{E_k}}=1+(\tilde{\delta}_k-1)\frac{b_{E_k}}{b_{E_k}-\nu_{E_k}}
$$
since $a_{E_k}=\tilde{\delta}_k b_{E_k}$. Thus
$$
\tilde{\delta}_k-1\geq (\delta-1)\frac{b_{E_k}-\nu_{E_k}}{b_{E_k}}\geq (\delta-1)\Big(1-\frac{\tilde{\delta}_k}{\lct}\Big)
$$
where the last inequality follows from $\nu_{E_k}\leq a_{E_k}/\lct= \tilde{\delta}_k b_{E_k}/\lct$. Therefore an easy calculation leads to the second inequality in (\ref{eqn:The2}) and concludes the proof.
\end{proof}

\subsubsection{Useful inequalities}
\begin{prop}
\label{prop:Fujita}
Let $\bD\in\bDiv_L(X)$ and let $E$ be a prime divisor over $X$. Then
\begin{equation}
    \label{eqn:KO}
    \tau_\bD(E)\geq \frac{n+1}{n} S_\bD(E).
\end{equation}
\end{prop}
\begin{proof}
The proof is strongly inspired by \cite[Proposition 2.1]{Fuj19}.\newline
Let first assume that $\bD=D_Y$ for a $\Q$-divisor $D_Y$ where $Y\overset{p_Y}{\geq} X$. \newline 
Set $\tau:=\tau_{D_Y}(E)$ and fix $0<x_0\leq x<x_1<\tau$. Then by the superadditivity of the positive intersection product (Theorem \ref{thm:PIP}), observing that $x_0\leq x_0+(1-\frac{x_0}{x})x_1<\tau$, we have
\begin{multline}
    \label{eqn:Ineqq}
    \langle (L_Y-x_0 E)^{n-1} \rangle\geq \bigg\langle \Big((L_Y-x_0 E-\big(1-\frac{x_0}{x}\big)x_1 E\Big)^{n-1}\bigg\rangle=\\
    =\bigg\langle \Big(\frac{x_0}{x}(L_Y-xE)+(1-\frac{x_0}{x})(L_Y-x_1 E)\Big)^{n-1} \bigg\rangle\geq \big(\frac{x_0}{x}\big)^{n-1}\langle (L_Y-x E)^{n-1}\rangle.
\end{multline}
Thus, proceeding as in \cite[Proposition 2.1]{Fuj19}, we set $F:[0,\tau)\to \R$, $F(y):=\langle(L_Y-y E)^{n-1}\rangle\cdot E$, and
\begin{gather*}
    b:=\frac{\int_0^\tau y F(y) dy}{\int_0^\tau F(y)dy}
\end{gather*}
to obtain
\begin{equation*}
    0=\int_{-b}^{\tau-b}y F(y+b)dy\leq \int_{-b}^{\tau-b}y \big(\frac{y+b}{b}\big)^{n-1}F(b)dy=\frac{F(b)\tau^n}{n b^{n-1}}\big(\frac{n}{n+1}\tau-b\big)
\end{equation*}
where the inequality clearly follows from (\ref{eqn:Ineqq}). Hence $\tau\geq \frac{n+1}{n}b$, and (\ref{eqn:KO}) is a consequence of
$$
b=\frac{n\int_0^\tau \big(\int_0^y dx\big)F(y)dy}{\langle L_Y^n \rangle}=\frac{n\int_0^\tau\int_x^\tau F(y)dy dx}{\langle L_Y^n \rangle}=\frac{\int_0^\tau \langle(L_Y-xE)^n \rangle dx}{\langle L_Y^n \rangle}=S_{D_Y}(E),
$$
where we used $\langle (L_Y-xE)^n \rangle=n\int_x^\tau F(y)dy$ (subsection \ref{ssec:Vol}).\newline
Finally, the general case $\bD=\{D_Y\}_{Y\geq X}$ follows by approximation using Lemma \ref{lem:Dini}.
\end{proof}
\begin{prop}
\label{prop:OtherIneq}
Let $\bD\in\bDiv_L(X)$ and let $E$ be a prime divisor over $X$. Then
\begin{equation}
    \label{eqn:SS}
    \tau_\bD(E)\leq (n+1)S_\bD(E).
\end{equation}
\end{prop}
\begin{proof}
Let $\psi\in \cM^+_D(X,\omega)$ associated to $\bD\in\bDiv_L(X)$. As seen in Lemma \ref{lem:Dini}, for any $x\in \big[0,\tau_\bD(E)\big)$,
$$
\langle (L-\bD-xE)^n \rangle=\int_X MA_\omega(\psi_x) 
$$
where $\psi_x:=\sup\big\{u\leq \psi\, : \, \nu(u,E)\geq \nu(\psi,E)+x\big\}$. In particular, by linearity of the Lelong numbers,
$$
\big(1-\frac{x}{\tau_\bD(E)}\big)\psi+\frac{x}{\tau_\bD(E)}\psi_{\tau_\bD(E)}\leq \psi_x.
$$
Thus \cite[Theorem B]{DDNL18b} gives
$$
\langle (L-\bD-xE)^n\rangle= \int_X MA_\omega(\psi_x)\geq \big(1-\frac{x}{\tau_\bD(E)}\big)^n\int_X MA_\omega(\psi)=\big(1-\frac{x}{\tau_\bD(E)}\big)^n \langle (L-\bD)^n\rangle.
$$
Integrating over $x\in [0,\tau_\bD(E)]$ clearly leads to (\ref{eqn:SS}) and concludes the proof.
\end{proof}

\subsection{Relative test curves}
\label{ssec:TestCurves}
The following definition of \emph{relative psh test curve} generalizes the one given in \cite{RWN14}.
\begin{defi}
A map $\R\ni s\to v_s\in \PSH(X,\omega,\psi)$ is a $[\psi]$\emph{-relative psh test curve} if
\begin{itemize}
    \item[i)] $s\to v_s(x)$ is concave, decreasing and upper semicontinuous for any $x\in X$;
    \item[ii)] there exists $C\in \R$ such that $v_s\equiv -\infty$ for any $s>C$;
    \item[iii)] $v_s$ increases a.e. to $\psi$ as $s\to -\infty$.
\end{itemize}
A $[\psi]$-relative psh test curve is then said
\begin{itemize}
    \item[a)] \emph{bounded} if there exists $C\in \R$ such that $v_s=\psi$ for any $s\leq C$;
    \item[b)] \emph{maximal} if either $v_s\in \cM(X,\omega)$ or $v_s\equiv -\infty$ for any $s\in\R$;
    \item[c)] \emph{finite energy} if
    \begin{equation}
        \label{eqn:EDZ}
        \int_{-\infty}^{s_\psi^+}\Big(\int_X MA_\omega(v_s)-\int_X MA_\omega(\psi)\Big)ds>-\infty
    \end{equation}
    where $s_\psi^+=\inf\big\{s\in\R\,:\, v_s\equiv -\infty\big\}$.
\end{itemize}
\end{defi}
Clearly, bounded $[\psi]$-relative psh test curves are finite energy since the integral in (\ref{eqn:EDZ}) can be restricted to the interval $(s_\psi^-,s_\psi^+)$ where
$$
s_\psi^-:=\sup\big\{s\in \R\, : \, v_s=\psi\big\}.
$$
Psh test curves has been introduced in \cite{RWN14} to construct psh rays by taking the \emph{inverse Legendre transform}. More precisely, given a $[\psi]$-relative psh test curve $\{v_s\}_{s\in \R}$, for any $x\in X$ we define
$$
\check{v}_t(x):=\sup_{s\in \R}\big\{v_s(x)+ts\big\}.
$$
The same proof given in \cite[Proposition 3.6]{DZ22} shows that $\check{v}$ is upper semicontinuous in $(t,x)\in(0,+\infty)\times X$. In particular $\check{v}_t\in \PSH(X,\omega)$ for any $t>0$ and it is immediate to see that
$$
\check{v}_t\leq \psi+t s_\psi^+.
$$
Vice versa given a $[\psi]$-relative psh ray $\{u_t\}_{t\geq 0}$ we can consider its Legendre transform
$$
\hat{u}_s:=\inf_{t\geq 0}\big\{u_t-ts\big\},
$$
to get the following correspondence.
\begin{prop}
\label{prop:Legendre}
The Legendre transform $\{v_s\}_{s\in \R}\to\{\check{v}_t\}_{t\geq 0}$ is a bijection map with inverse $\{u_t\}_{t\geq 0}\to \{\hat{u}_s\}_{s\in\R}$ between
\begin{itemize}
    \item[i)] $[\psi]$-relative psh test curves and $[\psi]$-relative psh rays;
    \item[ii)] maximal $[\psi]$-relative psh test curves and $[\psi]$-relative psh geodesic rays;
    \item[iii)] maximal bounded $[\psi]$-relative psh test curves and $[\psi]$-relative psh geodesic rays with minimal singularities. In this case we also have $\psi+s_\psi^- t\leq \check{v}_t\leq\psi+s_\psi^+ t$.
\end{itemize}
\end{prop}
A $[\psi]$-relative psh geodesic rays $\{u_t\}_{t\geq 0}$ has minimal singularities if $u_t$ has $[\psi]$-relative minimal singularities for any fixed $t\geq 0$.
\begin{proof}
The proof is exactly the same as that in \cite[Theorem 3.7]{DZ22} replacing $V_\theta$ by $\psi$.
\end{proof}
Next, we want to prove that the Legendre transform produces a bijection between maximal finite energy $[\psi]$-relative psh test curves and \emph{finite energy} $[\psi]$-relative psh geodesic rays, i.e. rays $\{u_t\}_{t\geq 0}$ such that $u_t\in \cE^1_\psi$ for any $t\geq 0$, and we want to express $\lim_{t\to +\infty}\frac{E_\psi(u_t)}{t}\in \R$ in terms of (\ref{eqn:EDZ}). Instead proceeding through a direct computation as in \cite[Theorem 3.9]{DZ22} (see also \cite[section 6]{RWN14}), we establish such bijection through an approximation argument.
\begin{lem}
\label{lem:N1}
Let $\{u_t\}_{t\geq 0}$ be a $[\psi]$-relative psh geodesic ray. Then there exists a sequence $\{u_t^k\}_{t\geq 0}$ of $[\psi]$-relative psh geodesic rays with minimal singularities such that
\begin{itemize}
    \item[i)] $u_t^k$ decreases to $u_t$ for any $t\in \R_{\geq 0}$;
    \item[ii)] $s_\psi^+(\hat{u}_s)=s_\psi^+(\hat{u}^k_s)$ for any $k\in\N$ and $MA_\omega(\hat{u}^k_s)\to MA_\omega(\hat{u}_s)$ weakly for any $s\leq s_\psi^+(\hat{u}_s)$.
\end{itemize}
\end{lem}
\begin{proof}
The proof coincides with the one given in \cite[Proposition 3.8]{DZ22} (which in turn is an adaptation of \cite[Theorem 4.5]{DL20}), replacing $V_\theta$ with $\psi$.
\end{proof}
\begin{lem}
\label{lem:N2}
Let $\{u_t\}_{t\geq 0}$ be a $[\psi]$-relative psh geodesic ray with minimal singularities and let $\{\psi_k\}_{k\in\N}\in \cM^+(X,\omega)$ be a decreasing sequence of \emph{algebraic} model type envelopes converging to $\psi$. Assume also $s_\psi^+=0$ where $s_\psi^\pm:=s_\psi^\pm(\hat{u}_s)$ and let $\{a_k\}_{k\in\N}$ be an increasing and diverging sequence of real numbers. Then there exists a sequence of $[\psi_k]$-relative psh geodesic rays $\{u_t^k\}_{t\geq 0}$ such that
\begin{itemize}
    \item[i)] $\{u_t^k\}_{t\geq 0}$ has $[\psi_k]$-relative minimal singularities and $s_{\psi_k}^-(\hat{u}_t^k)=a_k s_\psi^-$;
    \item[ii)] $u_t^k$ decreases to $u_t$ for any $t\in\R_{\geq 0}$;
    \item[iii)] $MA_\omega(\hat{u}_s^k)\to MA_\omega(\hat{u}^k_s)$ weakly for any $s< s_\psi^+$.
\end{itemize}
\end{lem}
\begin{proof}
For any $k\in\N$ and for any $s\leq 0$ we set
$$
\tilde{v}_s^k:=\bigg(1-\max\Big(0,1+\frac{s}{a_k\lvert s_\psi^-\rvert}\Big)\bigg)\psi_k+\max\Big(0,1+\frac{s}{a_k\lvert s_\psi^- \rvert}\Big) v_s
$$
and $v_s^k:=P_\omega[\tilde{v}_s^k](0)$. It is immediate to check that $v_s^k$ is a bounded maximal $[\psi_k]$-relative psh test curve and that $s_{\psi_k}^-(v_s^k)=a_k s_{\psi}^-$ for any $k\in\N$. Setting $u_t^k:=\check{v}^k_t$ we easily obtain $(i)$ and $(ii)$ (regarding $(ii)$ see again \cite[Theorem 4.5]{DL20}). Moreover, for any $0>s>a_k s_\psi^-$ the sequence
$$
v_s^k=P_\omega\Big[-\frac{s}{a_k\lvert s_\psi^-\rvert}\psi_k+\big(1+\frac{s}{a_k\lvert s_\psi^-\rvert}\big)v_s\Big](0),
$$
belongs to $\cM^+(X,\omega)$ and it is decreasing. We claim that $v_s^k\searrow v_s$ and that $MA_\omega(v_s^k)\to MA_\omega(v_s)$ weakly. Indeed, since $\psi_k\geq v_s$ we immediately get $v_s^k\geq P_\omega[v_s](0)=v_s$, while on the other hand
\begin{equation*}
    \int_X MA_\omega(v_s^k)=\int_X MA_\omega (\tilde{v}_s^k)\leq \int_X MA_\omega\Big((1+\frac{s}{a_k\lvert s_\psi^- \rvert})v_s\Big)\leq \Big(1+\frac{s}{a_k\lvert s_\psi^- \rvert}\Big)^n \int_X MA_\omega(v_s)-\frac{s}{a_k}C
\end{equation*}
for a constant $C>0$, where the first equality follows from the fact that $u\in \cE_{P_\omega[u](0)}$ if $\int_X MA_\omega(u)>0$ \cite[Theorem 1.3]{DDNL17b} while the inequalities are consequence of the monotonicity of the non-pluripolar product \cite[Theorem 1.2]{WN17}. Therefore by \cite[Lemma 2.6]{Tru19} $v_s^k$ converges to a model type envelope $w_s\geq v_s$, $MA_\omega(v_s^k)\to MA_\omega(w_s)$ weakly and $\int_X MA_\omega(w_s)=\int_X MA_\omega(v_s)$. Thus the equality $v_s=w_s$, which concludes the claim and the proof, follows again from \cite[Theorem 1.3]{DDNL17b} since $\int_X MA_\omega(v_s)>0$ by the $s$-concavity. 
\end{proof}
\begin{prop}
\label{prop:EnergyFormula}
The Legendre transform $\{v_s\}_{s\in\R}\to \{\check{v}_t\}_{t\geq 0}$ induces a bijective map with inverse $\{u_t\}_{t\geq 0}\to \{\hat{u}_s\}_{s\in\R}$ between maximal finite energy $[\psi]$-relative psh test curves and finite energy $[\psi]$-relative psh geodesic rays. In this case
\begin{equation}
    \label{eqn:En}
    \lim_{t\to \infty}\frac{E_\psi(\check{v}_t)}{t}=\frac{1}{V_\psi}\int_{-\infty}^{s_\psi^+}\Big(\int_X MA_\omega(v_s)-\int_X MA_\omega(\psi)\Big)ds +s_\psi^+.
\end{equation}
\end{prop}
\begin{proof}
Without loss of generality we may and will assume $s_\psi^+=0$.\newline
Suppose first that $\{v_s\}_{s\in\R}$ is a bounded maximal $[\psi]$-relative psh test curve and let $\{u_t\}_{t\geq 0}$ the associated $[\psi]$-relative psh geodesic ray (Proposition \ref{prop:Legendre}). Let also $\{\psi_k\}_{k\in\N}$ be a decreasing sequence of algebraic model type envelopes converging to $\psi$, and denote by $\{v_s^k\}_{s\in\R}, \{u_t^k\}_{t\geq 0}$ respectively the bounded maximal $[\psi_k]$-relative psh test curve and its associated $[\psi_k]$-relative psh geodesic ray with minimal singularities given by Lemma \ref{lem:N2}.\newline
Then, for any $k\in \N$ fixed, let $\varphi_k\in \PSH(X,\omega)$ be a function with algebraic singularities encoded in $(\mathcal{I}_k,c_k)$ for a coherent ideal sheaf $\mathcal{I}_k$ and $c_k\in\Q_{>0}$. Thus, letting $\rho_k:Y_k\to X$ be a log resolution of $\mathcal{I}_k$, we have
$$
\rho_k^*(\omega+dd^c\varphi_k)=\eta_k+c_k[D_k]
$$
where $\rho_k^{-1}\mathcal{I}_k=\mathcal{O}_{Y_k}(-D_k)$ and $\eta_k$ is a smooth closed semipositive $(1,1)$-form such that $V_{\psi_k}=\int_{Y_k}\eta_k^n>0$. Moreover
$$
\lim_{t\to +\infty}\frac{E_{\psi_k}(u_t^k)}{t}=\lim_{t\to +\infty} \frac{E^k(\tilde{u}_t^k)}{t}
$$
where $\tilde{u}_t^k=(u_t^k-\varphi_k)\circ \rho_k\in \PSH(Y_k,\eta_k)$, and $E^k$ is the Monge-Ampère energy in $\mathcal{E}^1(Y_k,\eta_k)$ (see \cite[Lemma 4.6, Proposition 4.7]{Tru20b}). Furthermore it is not hard to check that $\{\tilde{u}_t^k\}_{t\geq 0}$ is a geodesic psh ray with minimal singularities in $\mathcal{E}^1(Y_k,\eta_k)$ whose associated bounded maximal test curve $\{\tilde{v}_s^k\}_{s\in\R}$ in $(Y_k,\eta_k)$ satisfies
$$
\int_{Y_k}MA_{\eta_k}(\tilde{v}_s^k)=\int_X MA_\omega(v^k_s)
$$
for any $s$ (see also \cite[Theorem 3.7]{DZ22}). Hence by \cite[Theorem 3.9]{DZ22} we get
\begin{multline}
    \label{eqn:Klevel}
    \lim_{t\to \infty}\frac{E_{\psi_k}(u_t^k)}{t}=\lim_{t\to \infty} \frac{E^k(\tilde{u}^k_t)}{t}=\\
    =\frac{1}{\int_{Y_k}\eta_k^n}\int_{-\infty}^{0}\bigg(\int_{Y_k}MA_{\eta_k}(\tilde{v}_s^k)-\int_{Y_k}MA_{\eta_k}(0)\bigg)ds=\frac{1}{V_{\psi_k}}\int_{-\infty}^{0}\bigg(\int_X MA_\omega(v_s^k)-\int_X MA_\omega(\psi_k)\bigg)ds.
\end{multline}
Then we want to let $k\to +\infty$ in (\ref{eqn:Klevel}). \newline
Using the linearity of the Monge-Ampère energy along psh geodesic path (\cite[Theorem 3.11]{Tru20b}) we obtain
$$
\lim_{k\to +\infty}\lim_{t\to +\infty}\frac{E_{\psi_k}(u_t^k)}{t}=\lim_{k\to +\infty}E_{\psi_k}(u_1^k)=E_\psi(u_1)=\lim_{t\to +\infty}\frac{E_\psi(u_t)}{t}
$$
where the second-to-last equality follows from \cite[Proposition 2.7]{Tru19}.\newline
On the right-hand side of (\ref{eqn:Klevel}) instead we can first split the integral over $(a_ks_\psi^-, s_\psi^-)$ and over $(s_\psi^-,0)$. Clearly over $(s_\psi^-,0)$ we can apply Lebesgue's Dominated Convergence Theorem, while for any $s\in (a_k s_\psi^-,s_\psi^-)$ we have
$$
\int_X MA_\omega(v_s^k)\geq\int_X MA_\omega(\psi)
$$
by \cite[Theorem 1.2]{WN17} since $v_s^k\geq v_s=\psi$. Thus, setting $V_k:=V_{\psi_k}$ and $V:=V_\psi$, we have
$$
\int_{a_k s_\psi^-}^{s_\psi^-}\Big(\int_X MA_\omega(\psi_k)-\int_X MA_\omega(v_s^k)\Big)\leq \lvert s_\psi^-\rvert (a_k-1)(V_k-V).
$$
Hence, choosing $a_k\nearrow +\infty$ slowly enough we deduce that the right-hand side in (\ref{eqn:Klevel}) converges to
$$
\frac{1}{V_\psi}\int_{-\infty}^0 \Big(\int_X MA_\omega(v_s)-\int_X MA_\omega(\psi)\Big)ds
$$
as $k\to +\infty$, and (\ref{eqn:En}) follows.\newline
Let now $\{v_s\}_{s\in \R}$ be a finite energy maximal $[\psi]$-relative psh test curve. Approximating the associated $[\psi]$-relative psh geodesic ray $\{u_t\}_{t\geq 0}$ by the sequence $\{u_t^k\}_{t\geq 0}$ of $[\psi]$-relative psh geodesic rays with minimal singularities given by Lemma \ref{lem:N1}, we can easily obtain (\ref{eqn:En}) for $\{u_t\}_{t\geq 0}$ by letting $k\to +\infty$ in
$$
\lim_{t\to +\infty} \frac{E_\psi(u_t^k)}{t}=\frac{1}{V_\psi}\int_{-\infty}^0 \Big(\int_X MA_\omega(v^k_s)-\int_X MA_\omega(\psi)\Big)ds
$$
thanks to the Monotone Convergence Theorem. Hence by (\ref{eqn:EDZ}) we deduce that $\{u_t\}_{t\geq 0}$ is a finite energy $[\psi]$-relative psh geodesic ray.\newline
Conversely, let $\{u_t\}_{t\geq 0}$ be a finite energy $[\psi]$-relative psh geodesic ray. Similarly to before, thanks to the approximation given by Lemma \ref{lem:N1} and formula (\ref{eqn:En}) it is easy to check that the associated maximal $[\psi]$-relative psh test curve $\{v_s\}_{s\in\R}$ is finite energy, which concludes the proof.
\end{proof}
\subsubsection{Slope of the $L$-functional in terms of psh test curves.}
\label{ssec:SlopeDZ}
Similarly to \cite{DZ22} for any fixed $\lambda\in \big(0,\lct_X(\psi)\big)$ we introduce the functional $L^\lambda:\cE_\psi\to \R$ as
\begin{gather*}
    L^\lambda(u):=\frac{-1}{2\lambda}\log \int_X e^{-2\lambda u}d\mu
\end{gather*}
Note that for $\lambda=1$ we recover the functional $L$ introduced in subsection \ref{ssec:Variational}.
\begin{prop}
\label{prop:L2}
Let $\{u_t\}_{t\geq 0}\subset \cE^1_\psi$ be a $[\psi]$-relative psh geodesic ray and let $\{v_s\}_{s\in\R}$ be the associated maximal $[\psi]$-relative psh test curve. Assume also $\lambda\in \big(0,\lct_X(\psi)\big)$. Then
\begin{gather}
    \label{eqn:L1}
    \liminf_{t\to +\infty}\frac{L^\lambda(u_t)}{t}=\sup\Big\{s\in\R\, :\, \int_X e^{-2\lambda v_s}d\mu<+\infty\Big\}.
\end{gather}
\end{prop}
\begin{proof}
The proof proceeds as that in \cite[Theorem 4.1]{DZ22} but we will report the details here as a courtesy to the reader.\newline
We first observe that we can assume $s^+:=s_\psi^+(v_s)=0$ without loss of generalities replacing $u_t$ by $u_t-s^+t$ and $v_s$ by $v_{s+s^+}$.
Set also $\bar{s}$ for the quantity on the right hand side of $(\ref{eqn:L1})$. Letting $s<\bar{s}$, we have $C:=\int_X e^{-2\lambda v_s}d\mu\geq \int_X e^{-2\lambda (u_t-st)}d\mu$ (as $v_s\geq u_t-st$), which clearly leads to
$$
\liminf_{t\to +\infty }\frac{L^\lambda(u_t)}{t}\geq s.
$$
Vice versa, choosing $p\in\R,\epsilon>0$ such that $p+\epsilon<\liminf_{t\to+\infty}\frac{L^\lambda(u_t)}{t}$, there exists $t_0\geq 0$ such that $\int_X e^{-2\lambda u_t}d\mu\leq e^{-2(p+\epsilon)\lambda t}$ for any $t\geq t_0$. Therefore, an easy calculation implies
\begin{equation}
    \label{eqn:F1}
    \int_X \int_0^{+\infty} e^{2\lambda(pt-u_t)}dtd\mu=\int_0^{+\infty}e^{2\lambda p\lambda t}\Big(\int_X e^{-2\lambda u_t}d\mu\Big)dt<+\infty.
\end{equation}
Next, by basic properties of Legendre transform, $\frac{\sup_X u_t}{t}=\frac{\sup_X (u_t-\psi)}{t}\nearrow s^+=0$ and $t\to u_t(x)$ is decreasing for any $x\in X$. Thus, $v_p\not \equiv -\infty$ as $p<\liminf_{t\to+\infty}\frac{L^\lambda(u_t)}{t}\leq 0$. Moreover, letting $x\in X$ such that $v_p(x)\neq -\infty$, there exists $t_0>0$ such that $v_p(x)+1\geq u_{t_0}(x)-pt_0$. In particular, as $t\to u_t(x)$ is decreasing, $v_p(x)-p+1\geq u_t(x)-pt$ and
$$
\int_0^{+\infty}e^{2\lambda(pt-u_t(x))}dt\geq \int_{t_0}^{t_0+1}e^{2\lambda (pt-u_t(x))}dt\geq e^{2\lambda(p-1)}e^{-2\lambda v_p(x)}.
$$
Hence, integrating over $X$, from (\ref{eqn:F1}) we deduce that $\int_X e^{-2\lambda v_p}d\mu<+\infty$, which clearly concludes the proof.
\end{proof}

\subsubsection{Projection of psh test curves and maximization}
Similarly to \cite[Proposition 4.6]{DDNL18a}, letting $\{v'_s\}_{s\in\R}$ be a $[\psi]$-relative psh test curve, we define
$$
v_s:=P_\omega[v'_s](0)
$$
for any $s\leq s_\psi^+(v'_s)$ and $v_s:\equiv-\infty$ if $s> s_\psi^+(v'_s)$. It is the relative \emph{maximization} of $\{v'_s\}_{s\in\R}$.
\begin{prop}
\label{prop:Maximi}
$\{v_s\}_{s\in \R}$ is a maximal $[\psi]$-relative psh test curve. Moreover, 
\begin{itemize}
    \item[i)] denoting with $\{u'_t\}_{t\geq 0}, \{u_t\}_{t\geq 0}$ the associated $[\psi]$-relative psh rays, $\{u_t\}_{t\geq 0}$ is the smallest $[\psi]$-relative psh geodesic ray which satisfies $u_t\geq u'_t$ for any $t\geq 0$;
    \item[ii)] if $\{u'_t\}_{t\geq 0}$ is finite energy, then
    $$
    \lim_{t\to +\infty} \frac{ E_\psi(u_t)}{t}=\lim_{t\to +\infty} \frac{ E_\psi(u'_t)}{t};
    $$
    \item[iii)] $U_\NA=U'_\NA$.
\end{itemize}
\end{prop}
\begin{proof}
The proof is similar to \cite[Proposition 4.6]{DDNL18a}.\newline
It is easy to check that $s\to v_s$ is decreasing, that $v_s\equiv -\infty$ for any $s>s^+(v'_s)$ and that $s\to v_s$ is upper semicontinuous. As $v'_s\to \psi$ when $s\to -\infty$ we also get
$$
\psi=P_\omega[\psi](0)\geq\lim_{s\to -\infty}P_\omega[v'_s](0)\geq \lim_{s\to -\infty}P_\omega(v'_s+C,0)=P_\omega(\psi+C,0)=\psi,
$$
i.e. $v_s\nearrow \psi$ as $s\to -\infty$. Moreover, exploiting the concavity of $s\to v'_s$, it follows that for $s=as_1+(1-a)s_2$, where $s_1,s_2\in \R, a\in [0,1]$,
$$
v_s\geq P_\omega(v'_s+C,0)\geq aP_\omega(v'_{s_1}+C,0)+(1-a)P_\omega(v'_{s_2}+C,0)
$$
for any $C>0$, where the last inequality follows from that fact that the right hand side is a candidate in the supremum defining $P_\omega(v_s'+C,0)$. Letting $C\to+\infty$ and taking the upper semicontinuous regularizations, we obtain $v_s\geq av_{s_1}+(1-a)v_{s_2}$, i.e. $s\to v_s$ is concave. Hence $\{v_s\}_{s\in\R}$ is a $[\psi]$-relative psh test curve, and it remains to prove that $v_s\in \cM^(X,\omega)$ for any $s\leq s^+:=s^+(v_s)$. We first observe that for $s_0\in\R$ small enough $\int_X MA_\omega(v_{s_0})>0$ as the latter increases to $\int_X MA_\omega(\psi)$ by \cite[Lemma 2.6]{Tru19}. In particular by the concavity of $s\to v_s$ and \cite[Theorem 1.2]{WN17} we obtain
$$
\int_X MA_\omega(v_s)\geq C_s \int_X MA_\omega(v_{s_0})
$$
for any $s<s^+$ where $C_s>0$. Observe that a similar calculation leads to the continuity of $s\to \int_X MA_\omega(v_s)$. Moreover it also leads to $\int_X MA_\omega(v'_s)>0$ for any $s< s^+$. Hence $v_s=P_\omega[v'_s](0)\in \cM^+(X,\omega)$ for any $s< s^+$ as immediate consequence of what said in subsection \ref{ssec:KE}. Finally, letting $s\searrow s^+$, $v_s$ decreases to a model type envelope again by \cite[Lemma 2.6]{Tru19}. Hence, by the upper semicontinuity of $s\to v_s$, we deduce that $v_{s^+}=\lim_{s\nearrow s^+}v_s\in \cM(X,\omega)$, which concludes the first part of the proof.\newline
Next, $(i)$ follows by an easy application of Proposition \ref{prop:Legendre}, while regarding $(ii)$ we proceed as in \cite[Proposition 3.13]{DZ22}. Namely, we consider the $[\psi]$-relative psh ray $\{u_t^k\}_{t\geq 0}$ that coincides with the $[\psi]$-relative psh geodesic path joining $\psi$ and $u'_k$ for $t\in[0,k]$ and with $u'_t$ for any $t>k$. By construction we then have $u_t\geq u_t^k\geq u_t'$ for any $t\geq 0$, and $k\to u_t^k$ is increasing. In particular $\{u_t^k\}_{t\geq 0}$ increases to a $[\psi]$-relative psh geodesic ray which is necessarily $\{u_t\}_{t\geq 0}$. Thus, for any $k$,
\begin{equation}
    \label{eqn:M1}
    \lim_{t\to+\infty}\frac{E_\psi(u_t)}{t}\geq\lim_{t\to +\infty} \frac{E_\psi(u'_t)}{t}= \lim_{t\to +\infty} \frac{E_\psi(u_t^k)}{t}.
\end{equation}
Then, the same proof of \cite[Theorem 3.11]{Tru20c} shows that $E_\psi$ is convex along $[\psi]$-relative psh rays. Therefore,
\begin{equation}
    \label{eqn:M2}
    \lim_{k\to +\infty}\lim_{t\to +\infty}\frac{E_\psi(u_t^k)}{t}\geq \lim_{k\to +\infty} E_\psi(u_1^k)=E_\psi(u_1)=\lim_{t\to +\infty}\frac{E_\psi(u_t)}{t}
\end{equation}
where the second-to-last equality follows from $u_1^k\nearrow u_1$ while the last equality is given by Proposition \ref{thm:E1}$.(iii)$. Combining (\ref{eqn:M1}) with (\ref{eqn:M2}) concludes the proof of $(ii)$.\newline
Finally $(iii)$ is an immediate consequence of the properties of the sequence of $[\psi]$-relative psh rays $\{u_t^k\}_{t\geq 0}$ as
$$
U_\NA=\lim_{k\to +\infty} U^k_\NA=U'_\NA
$$
since the first equality follows from the continuity of the Lelong numbers along increasing sequences of $\pi^*\omega$-psh functions.
\end{proof}
Next, let $\tilde{\psi},\psi\in \cM^+(X,\omega)$ such that $\tilde{\psi}\geq \psi$, and let $\{\tilde{v}_s\}_s$ be a maximal $[\tilde{\psi}]$-relative psh test curve. We then set
$$
v_s:=P_\omega[\psi](\tilde{v}_s),
$$
and we denote by $\{\tilde{u}_t\}_{t\geq 0}$ the $[\tilde{\psi}]$-relative psh geodesic ray associated to $\{\tilde{v}_s\}_{s\in\R}$. Denote also by $\tilde{\bD}\in\bDiv_L(X)$ the generalized $b$-divisor associated to $\tilde{\psi}$.
\begin{prop}
\label{prop:Proj}
$\{v_s\}_s$ is either $-\infty$ for any $s\in \R$ or it is a maximal $[\psi]$-relative psh test curve. Moreover if $\{\tilde{u}_t\}_{t\geq 0}$ satisfies $\tilde{U}_\NA=U_\NA^{(\Ll,\tilde{\bD})}$ for a test configuration class $\Ll$ such that $\Ll-\bD$ is big and $\Ll-L\times\PP^1$ is psef, then
$$
U_\NA=U_\NA^{(\Ll,\bD)}
$$
where $\{u_t\}_{t\geq 0}$ is the $[\psi]$-relative psh geodesic ray associated to $\{v_s\}_{s\in \R}$.
\end{prop}
\begin{proof}
Set $s^+:=\sup\{s\in\R\, :\, P_\omega(\tilde{v}_s,\psi)\not\equiv -\infty\}\in \R\cup\{-\infty\}$. Clearly $s^+>-\infty$ if and only if $\{v_s\}_{s\in\R}$ is not identically $-\infty$. So wlog we will assume $s^+>-\infty$.\newline
We first claim that
\begin{equation}
    \label{eqn:Claim2}
    v_s=P_\omega\big[P_\omega(\tilde{v}_s,\psi)\big](0)=P_\omega(\tilde{v}_s,\psi)
\end{equation}
for any $s\leq s^+$. Indeed, for any $C>0$,
$$
P_\omega(\tilde{v}_s,\psi)\leq P_\omega(\tilde{v}_s,\psi+C)\leq P_\omega\big(P_\omega(\tilde{v}_s,\psi)+C,0\big)
$$
where the second inequality follows combining $P_\omega(\tilde{v}_s,\psi+C)\leq P_\omega(\tilde{v}_s,\psi)+C$ with $P_\omega(\tilde{v}_s,\psi)\leq \psi\leq 0$. Letting $C\to +\infty$ and taking the upper semicontinuity regularization we deduce that $ P_\omega(\tilde{v}_s,\psi)\leq v_s\leq P_\omega\big[(P_\omega(\tilde{v}_s,\psi)\big](0)$. However any $\omega$-function $u\preccurlyeq\psi$ actually satisfies $u\leq \psi+ \sup_X u$ since $\psi\in \cM(X,\omega)$ (subsection \ref{ssec:KE}), and similarly for $v_s$, $s\leq s^+$. Hence
$$
P_\omega\big(P_\omega(\tilde{v}_s,\psi)+C,0\big)=P_\omega(\tilde{v}_s,\psi),
$$
which concludes the proof of the claim (\ref{eqn:Claim2}). Moreover, similarly to the proof of Proposition \ref{prop:Maximi}, it is easy to check that $s\to v_s$ is decreasing, that $s\to v_s$ is upper semicontinuous and that $s\to v_s$ is concave. Hence $\{v_s\}_{s\in\R}$ is a maximal $[\psi]$-relative psh test curve.\newline
Next, letting $\Ll\in N^1(\FX_{\PP^1}^{\C^*})$ and $\{\tilde{u}_t\}_{t\geq 0}$ as in the hypothesis, Lemma \ref{lem:Indu} and Corollary \ref{cor:Abo} give
$$
U_\NA\leq U^{(\Ll,\bD)}_\NA.
$$
Then we consider $\{u'_t\}_{t\geq 0}$ $[\psi]$-relative psh ray such that $U'_\NA=U^{(\Ll,\bD)}_\NA$. Unless taking its maximization (Proposition \ref{prop:Maximi}) we can also assume that $\{u'_t\}_{t\geq 0}$ is a $[\psi]$-relative psh geodesic ray and that $ u_t\leq u'_t\leq \tilde{u}_t $ for any $t\geq 0$, Thus Proposition \ref{prop:Legendre} implies
$$
v_s\leq v'_s\leq \tilde{v}_s
$$
for any $s\in \R$, where we denoted by $\{v'_s\}_{s\in\R}$ the maximal $[\psi]$-relative psh test curve associated to $\{u'_t\}_{t\geq 0}$. On the other hand $v'_s\leq \psi$ for any $s\in\R$. Therefore $v'_s\leq P_\omega(\tilde{v}_s,\psi)$, and by (\ref{eqn:Claim2}) we conclude that $v'_s=v_s$ for any $s\in\R$. Hence
$$
U_\NA=U_\NA^{(\Ll,\bD)}.
$$
\end{proof}
We can now characterize the $\bD$-triviality of a test configuration $\Ll\in N^1(\FX_{\PP^1}^{\C^*})$ (see Definition \ref{defi:Triviality}), in terms of the maximal $[\psi]$-relative psh test curve associated through Propositions \ref{prop:Maximi}, \ref{prop:Proj}.
\begin{prop}
\label{prop:Dtriviality}
Let $\bD\in\bDiv_L(X)$ and let $\Ll\in N^1(\FX_{\PP^1}^{\C^*})$ be a test configuration. For any $m\gg 1$ big enough let also $\{v_{m,s}\}_{s\in\R}$ be the maximal $[\psi]$-relative psh test curve associated to $\Ll+m(X\times\{0\})$. Then $\Ll$ is $\bD$-trivial if and only if for any $s\in \R$, $v_{m,s}$ is either equal to $-\infty$ or equal to $\psi$.\newline
In particular an ample test configuration $(\X,\Ll)$ with associated maximal test curve $\{\tilde{v}_s\}_{s\in\R}$ defines a $\bD$-trivial ample test configuration class if and only if $P_\omega[\psi](\tilde{v}_s)$ is either equal to $-\infty$ or equal to $\psi$.
\end{prop}
\begin{proof}
Without loss of generality we may assume that $\Ll-L\times\PP^1$ is psef and that $\Ll-\bD$ is big. We denote by $\{v_s\}_{s\in\R}$ the associated maximal $[\psi]$-relative psh test curve given through Propositions \ref{prop:Maximi}, \ref{prop:Proj}. Set also $s^+:=s^+_\psi(v_s)$.\newline
If $v_s=\psi$ for any $s\leq s^+$ then we clearly get $u_t=\psi+ts^+$ for any $t\geq 0$ where $\{u_t\}_{t\geq 0}$ is the associated $[\psi]$-relative psh geodesic ray. Thus by definition we obtain $J^\NA(\Ll;\bD)=0$ applying Corollary \ref{cor:SlopeEJ}, i.e. $\Ll$ is $\bD$-trivial.\newline
Vice versa assume that $\Ll$ is $\bD$-trivial. Then if by contradiction $v_s$ is not equal to $\psi$ for any $s\leq s^+$, by Proposition \ref{prop:EnergyFormula} we get
$$
\lim_{t\to +\infty}\frac{E_\psi(u_t)}{t}<s^+
$$
since $\int_X MA_\omega(v_s)=\int_X MA_\omega(\psi)$ if and only if $v_s=\psi$ (recall that $v_s, \psi$ are model type envelopes and that $\int_XMA_\omega(\psi)>0$). Therefore again by Corollary \ref{cor:SlopeEJ} we deduce that
\begin{equation}
    \label{eqn:IneqContr}
    J^\NA(\Ll;\bD)=\lim_{t\to+\infty}\frac{J_\psi(u_t)}{t}=s^+-\lim_{t\to +\infty}\frac{E_\psi(u_t)}{t}>0
\end{equation}
where in the second equality we used that $\int_X (u_t-\psi)MA_\omega(\psi)=V_\psi\sup_X(u_t-\psi) +O(1)$. The inequality (\ref{eqn:IneqContr}) clearly contradicts the assumption on the $\bD$-triviality of $\Ll$. Hence the proof of the first part is concluded.\newline
The second statement is instead a clear consequence of what we have just proved and of Proposition \ref{prop:Proj}.
\end{proof}
\subsection{From Uniform $\bD$-log Ding stability to $\delta_\bD>1$.}
In this subsection we prove that uniform $\bD$-log Ding stability implies $\delta_\bD>1$, which is the implication $(ii)\Rightarrow (iii)$ in Theorem \ref{thmD}.\newline

Letting $E$ be a prime divisor over $X$, similarly to \cite{RWN14} we introduce the psh test curve
\begin{equation}
    \label{eqn:VE}
    v^E_s:=
    \begin{cases}
    0 \, \mathrm{if}\, s\leq 0\\
    \phi_s^E \, \mathrm{if}\, 0<s\leq \tau_0(E)\\
    -\infty \, \mathrm{if}\, s>\tau_0(E),
    \end{cases}
\end{equation}
where
$$
\phi_s^E:=\sup\{u\in \PSH(X,\omega)\, :\, u\leq 0, \, \nu(u,E)\geq s\}.
$$
We let the reader to check that $v_s^E$ is indeed a maximal relative psh test curve (for the existence of $\tilde{v}^E_{\tau_0(E)}=\lim_{s\nearrow \tau_\psi(E)}\tilde{v}_s^E$ see for instance \cite[Lemma 2.6]{Tru19}). We denote by $u_t^E$ the associated relative psh geodesic ray.
\begin{prop}
\label{prop:Esmooth}
Let $E$ be a prime divisor over $X$ and let $m>\tau_0(E)$ be a rational number. Then there exists a sequence of psh geodesic rays $\{u_{k,t}^{E,m}\}_{t\geq 0}$ such that
\begin{itemize}
    \item[i)] for any $k\in\N$, $U_{k,\NA}^{E,m}=\varphi_{\Ll_{k}^{E,m}}$ where $\Ll_k^{E,m}$ is an ample test configuration class such that $\Ll^{E,m}_k-\CD$ is big and $\Ll_k^{E,m}-L\times\PP^1$ is psef;
    \item[ii)] the associated sequences of maximal psh test curves $\{v_{k,s}^{E,m}\}_{s\in \R}$ is decreasing in $k\in\N$ and converges to $v^E_{s+\tau_0(E)-m}$.
\end{itemize}
\end{prop}
\begin{proof}
Similarly to \cite[subsection 5.3]{BBJ15}, we set $\tau_0:=\tau_0(E)$ and we consider the normalized psh ray $\tilde{u}^E:\R_{\geq 0}\to \cE^1$ given by
$$
\tilde{u}_t^E:=u_t^E-\tau_0t,
$$
which clearly extends to a $p_1^*\omega$-psh function on $X\times \D$. Moreover the multiplier ideal sheaves $\mathfrak{a}_k:=\mathcal{I}(2^k \tilde{U}^{E,\tau_0})$ are cosupported in the central fiber and there exists $k_0\in \N$ such that $\mathcal{O}\big((2^k+k_0)p_1^*L\big)\otimes \mathfrak{a}_k$ is generated by its global sections on $X\times\C$ for any $k\geq 0$ (\cite[Lemma 5.6]{BBJ15}). In other words, we can produce a sequence of normal ample test configurations $(\X_k^E,\Ll^E_k)$ associated to the singularities formally encoded in $\big(\mathfrak{a}_k, \frac{1}{2^k+k_0}\big)$. 
Next, by the subadditivity of multiplier ideal sheaves (\cite{DEL00}) it is not hard to check that the singularities of $\Big\{\big(\mathfrak{a}_{2^k}, \frac{1}{2^k+k_0}\big)\Big\}_{k\in\N}$ are increasing. In particular we can consider the maximal psh test curves $\{\tilde{v}_{k,s}^E\}_{s\in\R}$ given by the maximization (Proposition \ref{prop:Maximi}) of the psh test curve associated to a psh ray induced by a locally bounded metric on $(\X_k^E,\Ll_k^E)$. Thus, by construction $\{\tilde{v}_{k,s}^E\}_{s\in\R}$ is a decreasing sequence in $k\in\N$. 
We denote by $\{\tilde{u}_{k,t}^E\}_{t>0}$ the associated decreasing family of psh geodesic rays, and we set
$$
u_{k,t}^{E,m}:=\tilde{u}_{k,t}^E+t m.
$$
As $\tilde{U}_{k,\NA}^E=\varphi_{\Ll_k^E}$ (Proposition \ref{prop:Maximi}), we immediately get $U_{k,\NA}^{E,m}=\varphi_{\Ll_k^{E,m}}$ where $\Ll_k^{E,m}:=\Ll_k^E+m(X\times\{0\})$ is clearly an ample test configuration class.

Next, denoting by $\tilde{v}_s^E$ the maximal psh test curve associated to $\tilde{u}_t^E$, we have $ \tilde{v}_s^E=v^E_{s+\tau_0} $ (see \ref{eqn:VE}). In particular, for any $s\leq-\tau_0$ we have
$$
0\geq \tilde{v}^E_{k,s}\geq \tilde{v}^E_s=v^E_{s+\tau_0}=0,
$$
i.e. the maximal psh test curve $\{v_{k,s}^{E,m}\}_{s\in\R}$ associated to $\{u_{k,t}^{E,m}\}_{t\geq 0}$ satisfies
$$
v_{k,s}^{E,m}=\tilde{v}_{k,s-m}^E=0
$$
for any $s\leq m-\tau_0$. Thus, $u_{k,t}^{E,m}\geq \psi+t(m-\tau_0)$ for any $k\in\N$, which leads to $\Ll_k^{E,m}-\CD$ big and $\Ll_k^{E,m}-L\times\PP^1$ psef by Lemma \ref{lem:Effe} and Proposition \ref{prop:Abo}.\newline
It then remains to show that $ v_{k,s}^{E,m}\searrow v_{s+\tau_0-m}^E $, which is equivalent to prove that
\begin{equation}
    \label{eqn:Conv}
    \tilde{v}^E_{k,s}\searrow \tilde{v}^E_s
\end{equation}
for any $s\in\R$. But by construction the sequence of algebraic psh geodesic rays $\{\tilde{u}_{k,t}^E\}_{t\geq 0}$ decreases to the smallest psh geodesic ray which is bigger than $\{\tilde{u}_t^E\}_{t\geq 0}$ and which has the same singularities of $\{\tilde{u}_t^E\}_{t\geq 0}$ (see again \cite[Lemma 5.6]{BBJ15}). Thus the convergence (\ref{eqn:Conv}) follows from \cite[Theorem 1.1]{DX20} since by definition $\tilde{v}_s^E\in \cM_D(X,\omega)$ for any $s\leq 0$, while $\tilde{v}_s^E\equiv -\infty$ for $s>0$.
\end{proof}
Next, similarly to (\ref{eqn:VE}) we define
\begin{equation}
    \label{eqn:VE2}
    v^{\psi,E}_s:=
    \begin{cases}
    \psi \, \mathrm{if}\, s\leq 0\\
    \phi_s^{\psi,E} \, \mathrm{if}\, 0<s\leq \tau_\bD(E)+\ord_E(\bD)\\
    -\infty \, \mathrm{if}\, s>\tau_\bD(E)+\ord_E(\bD),
    \end{cases}
\end{equation}
where
$$
\phi_s^{\psi,E}:=\sup\{u\in \PSH(X,\omega)\, :\, u\leq \psi, \, \nu(u,E)\geq s\}.
$$
Proposition \ref{prop:Proj} implies that $\{v_s^{\psi,E}\}_{s\in\R}$ is a maximal $[\psi]$-relative psh test curve as $v_s^{\psi,E}=P_\omega[v_s^E](0)$ (see (\ref{eqn:VE}) for the definition of $\{v_s^E\}_{s\in\R}$).\newline
We will denote by $\{u^{\psi,E}_t\}_{t\geq 0}$ the $[\psi]$-relative psh geodesic ray associated to $\{v^{\psi,E}_s\}_{s\in\R}$. 
\begin{prop}
\label{prop:Appoo}
Let $\bD\in \bDiv_L(X)$, let $E$ be a prime divisor over $X$ and let $m>\tau_0(E)$ be a rational number. Then there exists a sequence $\{\Ll^{E,m}_k\}_k\in N^1(\FX_{\PP^1}^{\C^*})$ of ample test configurations classes such that
\begin{itemize}
    \item[i)] $\Ll^{E,m}_k-\CD$ is big and $\Ll_k^{E,m}-L\times\PP^1$ is psef for any $k\in\N$;
    \item[ii)] $L^\NA(U^{\psi,E}_\NA)=\lim_{k\to +\infty} L^\NA\big(\Ll^{E,m}_k;\bD\big)-m+\tau_0(E)$;
    \item[iii)] $\lim_{t\to +\infty}\frac{E_\psi(u_t^{\psi,E})}{t}=\lim_{k\to +\infty} E^\NA\big(\Ll_k^{E,m};\bD\big)-m+\tau_0(E)$;
    \item[iv)] $\lim_{t\to +\infty}\frac{J_\psi(u_t^{\psi,E})}{t}=\lim_{k\to +\infty} J^\NA\big(\Ll_k^{E,m};\bD\big)$.
\end{itemize}
\end{prop}
\begin{proof}
Using the same notations of Proposition \ref{prop:Esmooth}, $(i)$ follows immediately. Then we define
$$
v_{k,s}^{\psi,E,m}:=P_\omega[\psi](v_{k,s}^{E,m}),
$$
observing that it is a maximal $[\psi]$-relative psh test curve by Proposition \ref{prop:Proj}. Observe also that, as $v_{k,s}^{E,m}\searrow v_s^{E,m}$,
$$
v_{k,s}^{\psi,E,m}\searrow P_\omega[\psi](v_{s+\tau_0-m}^E)=v_{s+\tau_0-m}^{\psi,E}
$$
where the equality can be easily checked from their definitions. Moreover again by Proposition \ref{prop:Proj} the associated $[\psi]$-relative psh geodesic rays $\{u^{\psi,E,m}_{k,t}\}_{t\geq 0}$ are algebraic and satisfy $U^{\psi,E,m}_{k,\NA}=U_\NA^{(\Ll_k^{E,m};\bD)}$.\newline
Then, setting $v_s^{\psi,E,m}:=v^{\psi,E}_{s+\tau-m}$, we claim that
\begin{multline*}
    \lct_X(v_s^{\psi,E,m})=\sup\Big\{s\in \R\, : \, \int_X e^{-v_s^{\psi,E,m}}d\mu<+\infty \Big\}=\\
    =\lim_{k\to +\infty}\sup\Big\{s\in \R\, : \, \int_X e^{-v_{s,k}^{\psi,E,m}}d\mu<+\infty \Big\}=\lim_{k\to +\infty} \lct_X(v_{s,k}^{\psi,E,m}).
\end{multline*}
As recalled in Lemma \ref{lem:bKlt}, if by contradiction $\lct_X(v_s^{\psi,E,m})< s_1<s_2< \inf_{k\geq 1} \lct_X(v_{s,k}^{\psi,E,m})$ then there exists a prime divisor $F$ above $X$, and $s_0< \lct_X(v_s^{\psi,E,m})$ such that
\begin{equation}
    \label{eqn:KKK}
    \nu(v_{s_0}^{\psi,E,m},F)< A_X(F)\leq \nu(v_{s_1}^{\psi,E,m},F)\leq\nu(v_{s_2}^{\psi,E,m},F).
\end{equation}
On the other hand, as $\cM_D^+(X,\omega)\ni v_{k,s}^{\psi,E,m}\searrow v_s^{\psi,E,m}$ we also have $\nu(v_{k,s_2}^{\psi,E,m},F)\nearrow \nu(v_{s_2}^{\psi,E,m},F)$ (see subsection \ref{ssec:KE}). Hence, from $A_X(F)>\nu(v_{k,s_2}^{\psi,E,m},F)$ we obtain
$$
A_X(F)\geq \nu(v_{k,s_2}^{\psi,E,m},F)\geq \nu(v_{k,s_1}^{\psi,E,m},F),
$$
which combined with (\ref{eqn:KKK}) leads to $A_X(F)=\nu(v_{k,s_2}^{\psi,E,m},F)=\nu(v_{k,s_1}^{\psi,E,m},F)$. Thus, by the concavity of the map $s\to v_s^{\psi,E,m}$ and the linearity of Lelong numbers we get $\nu(v_{k,s_0}^{\psi,E,m},F)=A_X(F)$ which contradicts (\ref{eqn:KKK}).\newline
Therefore Propositions \ref{prop:L}, \ref{prop:L2} yield
\begin{equation*}
    L^\NA(U_\NA^{\psi,E,m})=\lim_{t\to +\infty}\frac{L(u_t^{\psi,E,m})}{t}=\lct_X(v_s^{\psi,E,m})=\lim_{k\to +\infty}\lct_X(v_{s,k}^{\psi,E,m})=\lim_{k\to +\infty}L^\NA\big(\Ll_k^E+m(X\times\{0\});\bD\big),
\end{equation*}
which leads to $(ii)$ as $L^\NA(U_\NA^{\psi,E,m})=L^\NA(U_\NA^{\psi,E})+m-\tau_0$ by construction and again Proposition \ref{prop:L}.\newline
Next Proposition \ref{prop:EnergyFormula} implies
$$
\lim_{t\to+\infty} \frac{E_\psi(u_t^{\psi,E,m})}{t}=\lim_{k\to +\infty}E^\NA\big(\Ll_k^{E,m};\bD\big)
$$
by Lebesgue Dominated Convergence Theorem as $s_\psi^+(v_{k,s}^{\psi,E,m})\searrow s_\psi^+(v_s^{\psi,E,m})$ and $\int_X MA_\omega(v_{k,s}^{\psi,E,m})\searrow \int_X MA_\omega (v_s^{\psi,E,m})$ (subsection \ref{ssec:KE}). Thus $(iii)$ follows observing that $u_t^{\psi,E,m}=u_t^{\psi,E}+(m-\tau_0)t$.\newline
Regarding $(iv)$, by Corollary \ref{cor:SlopeEJ} we have
\begin{multline*}
    J^\NA(\Ll_k^{E,m};\bD)=\lim_{t\to+\infty} \frac{J_\psi(u_{k,t}^{\psi,E,m})}{t}=s_\psi^+(v_{k,s}^{\psi,E,m})-\lim_{t\to +\infty}\frac{E_\psi(u_{k,t}^{\psi,E,m})}{t}=s_\psi^+(v_{k,s}^{\psi,E,m})-E^\NA(\Ll_k^{E,m};\bD)
\end{multline*}
where the second-to-last equality follows observing that $\frac{1}{V_\psi}\int_{X}(u_{k,t}^{\psi,E,m}-\psi)MA_\omega(\psi)=\sup_X u_{k,t}^{\psi,E,m}+O(1)$ \cite[Lemma 3.7]{Tru20b}. Hence, using also the translation property of $J_\psi$, $(iii)$ and $s_\psi^+(v_{k,s}^{\psi,E,m})\searrow s_\psi^+(v_s^{\psi,E,m})$ yields
\begin{multline*}
    \lim_{t\to+\infty} \frac{J_\psi(u_t^{\psi,E})}{t}=\lim_{t\to+\infty} \frac{J_\psi(u_t^{\psi,E,m})}{t}=s_\psi^+(v_s^{\psi,E,m})-\lim_{t\to+\infty}\frac{E_\psi(u_t^{\psi,E,m})}{t}=\\
    =\lim_{k\to+\infty}\Big(s_\psi^+(v_{k,s}^{\psi,E,m})-E^\NA(\Ll_k^{E,m};\bD)\Big)=\lim_{k\to +\infty}J^\NA(\Ll_k^{E,m};\bD),
\end{multline*}
which concludes the proof.
\end{proof}
The following Lemma generalizes \cite[Lemma 4.4]{DZ22} in the prescribed singularities setting.
\begin{lemma}
\label{lem:Z1}
Let $\bD\in \bDiv_L(X)$, let $\psi\in \cM^+_D(X,\omega)$ the model type envelope associated (Proposition \ref{prop:Corre}) and let $E$ be a prime divisor over $X$. Then
$$
\lim_{t\to+\infty}\frac{E_\psi(u_t^{\psi,E})}{t}=S_\bD(E)+\ord_E(\bD)
$$
\end{lemma}
\begin{proof}
By Proposition \ref{prop:Corre} we can choose a sequence $Y_{k+1}\geq Y_k\geq X$ such that the decreasing sequence of model type envelopes $\{\psi_k\}_{k\in\N}\subset \cM^+_D(X,\omega)$ associated to $\{D_{Y_k}\}_{k\in\N}$ converges to $\psi$. Moreover by Proposition \ref{prop:EnergyFormula} we have
$$
\lim_{t\to+\infty} \frac{E_{\psi_k}(u_t^{\psi_k,E})}{t}=\frac{1}{V_{\psi_k}}\int_{-\infty}^{\tau_{D_{Y_k}}(E)+\ord_E D_{Y_k}}\Big(\int_X MA_\omega(v_s^{\psi_k,E})-\int_X MA_\omega(\psi)\Big)ds+\tau_{D_{Y_k}}(E)+\ord_E D_{Y_k}
$$
where with obvious notation $v_s^{\psi_k,E}$ is defined as in (\ref{eqn:VE2}) replacing $\psi$ by $\psi_k$. In particular, as $\ord_E D_{Y_k}=\nu(\psi_k,E)\searrow \nu(\psi,E)=\ord_E \bD$ and $v_s^{\psi_k,E}\searrow v_s^{\psi,E}$ we get $\tau_{D_{Y_k}}(E)+\ord_E D_{Y_k}\searrow \tau_{\bD}(E)+\ord_E \bD$ (see also Lemma \ref{lem:Dini}) and
$$
\lim_{k\to+\infty}\lim_{t\to+\infty}\frac{E_{\psi_k}(u_t^{\psi_k,E})}{t}=\lim_{t\to+\infty} \frac{E_{\psi}(u_t^{\psi,E})}{t}.
$$
Therefore Lemma \ref{lem:Dini} implies that we can assume $\bD=D_Y$ for $Y\geq X$ without loss of generalities. Then, for any $s\leq \ord_E \bD$ we clearly have $v_s^{\psi,E}=\psi$ while for any $\ord_E \bD< s\leq \tau_\bD E+\ord_E \bD$,
$$
v_s^{\psi,E}=\sup\{u\in \PSH(X,\omega)\, : \, u\leq \psi, \nu(u,E)\geq \ord_E \bD+(s-\ord_E \bD)\}.
$$
Thus, by what observed during the proof of Lemma \ref{lem:Dini}, Proposition \ref{prop:Corre} and \cite[Lemma 3.12]{Tru20a} lead to
$$
\int_X MA_\omega(v_s^{\psi,E})=\langle\big(L-\bD-(s-\ord_E \bD)E\big)^n \rangle
$$
for any $s\in [\ord_E\bD, \tau_\bD (E)+\ord_E \bD]$. Hence
\begin{multline*}
    \lim_{t\to+\infty}\frac{E_\psi(u_t^{\psi,E})}{t}=\frac{1}{V_\psi}\int_{\ord_E\bD}^{\tau_\bD(E)+\ord_E \bD}\Big(\int_X MA_\omega(v_s^{\psi,E})-\int_X MA_\omega(\psi)\Big)ds+\tau_\bD(E)+\ord_E \bD=\\
    =\frac{1}{V_\psi}\int_0^{\tau_\bD(E)}\langle(L-\bD-xE)^n \rangle dx+\ord_E\bD=S_\bD(E)+\ord_E \bD,
\end{multline*}
which concludes the proof.
\end{proof}
We will also need the following easy consequence of Propositions \ref{prop:L}, \ref{prop:L2}.
\begin{lemma}
\label{lem:Z2}
Let $\bD\in \bDiv_L(X)$ and let $E$ be a prime divisor above $X$. Then
$$
L^\NA(U_\NA^{\psi,E})\leq A_X(E).
$$
\end{lemma}
\begin{proof}
By Propositions \ref{prop:L}, \ref{prop:L2},
$$
L^\NA(U_\NA^{\psi,E})=\lim_{t\to+\infty}\frac{L(u_t^{\psi,E})}{t}=\sup\Big\{s\in \R\, : \, \int_Xe^{-2v_s^{\psi,E}}d\mu<+\infty\Big\}.
$$
Thus, letting $Y\overset{p}{\geq} X$ such that $E\subset Y$, we can consider holomorphic coordinates $(w_1,\dots,w_n)$ on a small open set $U\Subset Y$ such that $U\cap E=\{w_1=0\}$ and such that $U$ does not meet other exceptional divisors. Moreover, for fixed $s\in \R$, by the upper semicontinuity of the Lelong numbers, we can also consider a possibly smaller open set $V\Subset U$ such that $V\cap \{w_1=0\}\neq \emptyset$ and such that $\nu(v_s\circ p, y)\geq s$ for any $y\in \overline{V}\cap\{w_1=0\}$. Therefore, letting $d\tilde{V}$ be the local volume form on $V$ such that $p^{-1}d\mu=\lvert w_1 \rvert^2(A_X(E)-1)d\tilde{V}$ (recall that $K_Y=p^*K_X+ (A_X(E)-1)E+F$ for $F$ divisor whose support does not meet $V$),
$$
\int_X e^{-2\lambda v_s^{\psi,E}}d\mu=\int_Y e^{-2\lambda v_s^{\psi,E}\circ p}p^{-1}d\mu\geq C\int_V \lvert w_1 \rvert^{2(A_X(E)-1-s)}d\tilde{V}
$$
for a constant $C>0$. Hence if $\int_X e^{-2v_s^{\psi,E}}d\mu<+\infty$ then necessarily $s<A_X(E)$ and the Lemma follows.
\end{proof}
We are now ready to prove the main result of this subsection.
\begin{theorem}
\label{thm:UniformToDelta}
Let $\bD\in\bDiv_L(X)$. If $D^\NA(\Ll;\bD)\geq \delta J^\NA(\Ll;\bD)$ for any $\Ll\in N^1(\FX_{\PP^1}^{\C^*})$ ample test configuration class, then
\begin{equation}
    \label{eqn:deltaine}
    \delta_\bD\geq 1+ \frac{\delta}{n}.
\end{equation}
\end{theorem}
\begin{proof}
Fix $E$ prime divisor over $X$, $m>\tau_0(E)$ and consider the sequence of ample test configurations $\{\Ll_k^{E,m}\}$ given by Proposition \ref{prop:Appoo}. By hypothesis on the uniform $\bD$-log Ding Stability and by Proposition \ref{prop:Appoo} it follows that
\begin{multline*}
    L^\NA(U_\NA^{\psi,E})=\lim_{k\to +\infty} L^\NA(\Ll_k^{E,m};\bD)-m+\tau_0(E)\geq\\
    \geq \lim_{k\to\infty} \Big(E^\NA(\Ll_k^{E,m};\bD)+\delta J^\NA(\Ll_k^{E,m};\bD)\Big)-m+\tau_0(E) =\lim_{t\to+\infty} \Big( \frac{E_\psi(u_t^{\psi,E})}{t}+\delta \frac{J_\psi(u_t^{\psi,E})}{t}\Big)
\end{multline*}
where $\{u_t^{\psi,E}\}_{t\geq 0}$ is the $[\psi]$-relative psh geodesic ray associated to the $[\psi]$-relative psh test curve defined in (\ref{eqn:VE2}). Moreover, as $\frac{1}{V_\psi}\int_X(u_t^{\psi,E}-\psi)MA_\omega(\psi)=\sup_X u_t^{\psi,E}+O(1)$ as $t\to+\infty$ \cite[Lemma 3.7]{Tru20b}, we also have
$$
\lim_{t\to+\infty} \frac{ J_\psi(u_t^{\psi,E})}{t}=\lim_{t\to+\infty} \frac{\sup_X u_t^{\psi,E}}{t}-\lim_{t\to+\infty}\frac{E_\psi(u_t^{\psi,E})}{t}=\tau_\bD(E)+\ord_E(\bD)-\lim_{t\to+\infty}\frac{E_\psi(u_t^{\psi,E})}{t}
$$
where the last equality follows by the definition of $\{u_t^{\psi,E}\}_{t\geq 0}$. Therefore Lemmas \ref{lem:Z1}, \ref{lem:Z2} leads to
$$
A_X(E)\geq S_\bD(E)+\ord_E(\bD)+ \delta\big(\tau_\bD(E)-S_\bD(E)\big).
$$
On the other hand $\tau_\bD(E)\geq \frac{n+1}{n}S_\bD(E)$ by Proposition \ref{prop:Fujita}. Thus, it follows that
$$
A_X(E)-\ord_E(\bD)\geq \Big(1+\frac{\delta}{n}\Big)S_\bD(E),
$$
i.e. $\delta_\bD>1+\frac{\delta}{n}$.
\end{proof}
\subsection{From $\delta_{\bD}>1$ to the existence of a $[\psi]$-KE metric} In this subsection we conclude the proof of Theorem \ref{thmD} showing that $\tilde{\delta}_{\bD}>1$ implies the existence of a $[\psi]$-KE metric. Note that for $\bD\in\bDiv_L(X)$ klt the equivalence $\delta_\bD>1$ if and only if $\tilde{\delta}_\bD>1$ is proved in Proposition \ref{prop:deltacompa}.\newline

\begin{prop}
\label{prop:Necessaria}
Let $\{v_s\}_{s\in\R}$ be a maximal $[\psi]$-relative psh test curve and let $E$ be a prime divisor over $X$. Let also $s_0=\lim_{t\to +\infty}\frac{E_\psi(u_t)}{t}$ where $\{u_t\}_{t\geq 0}$ is the $[\psi]$-relative psh geodesic ray associated to $\{v_s\}_{s\in \R}$. Then
\begin{equation}
    \label{eqn:Easy}
    \nu(v_{s_0},E)\leq S_\bD(E)+\ord_E \bD.
\end{equation}
\end{prop}
\begin{proof}
The proof is an easy adaptation of \cite[Proposition 4.5]{DZ22} but we will write the details as a courtesy to the reader.\newline
Proposition \ref{prop:EnergyFormula} implies that $s_0=s^+:=s_\psi^+(v_s)$ if and only if $ \int_X MA_\omega(v_s)=\int_X MA_\omega(\psi)$ for any $s\leq s^+$. But, since $v_s$ are model type envelopes more singular than $\psi\in \cM^+(X,\omega)$, \cite[Theorem 1.3]{DDNL17b} implies that $v_s=\psi$ for any $s\leq s^+$. Hence $s_0=s^+$ if and only if $u_t=\psi+t s_0$ and (\ref{eqn:Easy}) follows in this case from $\nu(\psi,E)=\ord_E \bD$.\newline
Let $f(s):=\nu(v_s,E), s\in (-\infty,s^+]$. Observe that $f$ is non-negative, convex and non-decreasing, and set $f(-\infty)=\lim_{s\to -\infty}f(s)$.\newline
If $f(s_0)=f(-\infty)$, then $f(s)=f(s_0)=:a$ for any $s\leq s_0$, i.e. $v_s$ is more singular than
$$
\psi_{a}:=\{u\leq \psi \, ; \, \nu(u,E)\geq a\}.
$$
In particular Proposition \ref{prop:EnergyFormula} gives
$$
-\infty<s_0\leq \frac{1}{V_\psi}\int_{-\infty}^{s_0} \Big(\int_X MA_\omega(\psi_a)-\int_X MA_\omega(\psi)\Big)ds,
$$
which necessarily implies $\int_X MA_\omega(\psi_a)=\int_X MA_\omega(\psi)$. Hence, by \cite[Theorem 1.3]{DDNL17b} we have $\psi_a=\psi$, $a=\nu(\psi,E)$ and $(\ref{eqn:Easy})$ follows.\newline
Therefore we can assume $s_0<s^+$ and $f(-\infty)<f(s_0)$. Set also $c:=\ord_E\bD$ to lighten notations. By the convexity of $f$ we have that $ b:=f'_-(s_0)$ is a finite positive number and that the function
$$
g(s):=
\begin{cases}
c\,\,\,\,\,\,\,\,\,\, s\leq s_0-b^{-1} f(s_0)+b^{-1}c\\
b(s-s_0)+f(s_0) \,\,\,\,\,\,\, s\in \big(s_0-b^{-1}f(s_0)+b^{-1}c, s^+\big]
\end{cases}
$$
satisfies $f(s)\geq g(s)$ for any $s\leq s^+$. Thus, letting $v_s^{\psi,E}$ be the maximal $[\psi]$-relative psh test curve defined in (\ref{eqn:VE2}), we obtain
$$
\int_X MA_\omega(v_s)\leq \int_X MA_\omega(v_{g(s)}^{\psi,E})
$$
for any $s\leq s^+$ by \cite[Theorem 1.2]{WN17}. Hence, by Proposition \ref{prop:EnergyFormula}
\begin{multline}
    \label{eqn:Tog}
    s_0\leq \frac{1}{V_\psi}\int_{s_0-b^{-1}\big(f(s_0)-c\big)}^{s^+}\Big(\int_X MA_\omega(v_{g(s)}^{\psi,E})-\int_X MA_\omega(\psi)\Big)ds +s^+=\\
    =\frac{1}{bV_\psi}\int_c^{b(s^+-s_0)+f(s_0)}\Big(\int_X MA_\omega(v_s^{\psi,E})\Big) ds+s_0-b^{-1}\big(f(s_0)-c\big)\leq\\
    \leq \frac{1}{bV_\psi} \int_c^{\tau_\bD(E)+c}\Big(\int_X MA_\omega(v_s^{\psi,E})\Big)ds+s_0-b^{-1}\big(f(s_0)-c\big).
\end{multline}
On the other hand, as seen in the proof of Proposition \ref{lem:Z1},
$$
\frac{1}{V_\psi}\int_c^{\tau_\bD(E)+c}\Big(\int_X MA_\omega(v_s^{\psi,E})\Big)ds=S_\bD(E),
$$
which together with \ref{eqn:Tog} leads to $ 0\leq b^{-1}\Big(S_\bD(E)+c-f(s_0)\Big) $.
i.e. (\ref{eqn:Easy}) since $b>0$.
\end{proof}
Similarly to \cite{DZ22}, for any $\lambda\in \big(0,\lct_X(\psi)\big)$, we define the translation-invariant functional $D^\lambda_\psi:\cE^1_\psi\to \R$ as
$$
D^\lambda_\psi(u):=L^\lambda(u)-E_\psi(u).
$$
We recall that $L^\lambda(u):=\frac{-1}{2\lambda}\int_X e^{-2\lambda u}d\mu$ has been introduced in subsection \ref{ssec:SlopeDZ}. Observe that $D^\lambda_\psi(u)\in \R$ since $\lct_X(u)=\lct_X(\psi)$ for any $u\in\cE^1_\psi$ \cite[Proposition 3.10]{Tru20b}.
\begin{lemma}
\label{lem:Delta&lct}
Let $\bD\in \bDiv_L(X)$. Then
\begin{gather*}
    \tilde{\delta}_\bD\leq\lct_X(\bD).
\end{gather*}
In particular if $\tilde{\delta}_\bD>1$ then $(X,\bD)$ is klt.
\end{lemma}
\begin{proof}
As immediate consequence of its definition and \cite[Theorem B.5]{BBJ15},
$$
\lct_X(\bD)=\sup\big\{c>0\, : \, A_X(E)\geq c \, \ord_E\bD \, \mathrm{for}\, \mathrm{any} \, E/X\big\}
$$
where as usual with $E/X$ we mean $E$ prime divisor over $X$. Hence
$$
\tilde{\delta}_\bD=\inf_{E/X} \frac{A_X(E)}{S_\bD(E)+\ord_E\bD}\leq \inf_{E/X} \frac{A_X(E)}{\ord_E \bD}\leq \lct_X(\bD).
$$
\end{proof}
\begin{prop}
\label{prop:438}
Let $\bD\in\bDiv_L(X)$ such that $\tilde{\delta}_\bD>1$. Then, for any $\lambda \in (0,\tilde{\delta}_\bD)$ and for any $U:\R_{\geq 0}\to \cE^1_\psi$ $[\psi]$-relative psh geodesic ray,
\begin{equation}
    \label{eqn:lambdaDing}
    \liminf_{t\to+\infty}\frac{D^\lambda_\psi(u_t)}{t}\geq 0.
\end{equation}
\end{prop}
\begin{proof}
Fix $\lambda< \tilde{\delta}_\bD$. By Lemma \ref{lem:Delta&lct}, $\lambda<\lct_X(\bD)=\lct_X(\psi)$ for any $u\in \cE^1_\psi$, i.e. $D_\psi^\lambda(u_t)$ is well-defined.\newline
Denote by $\{v_s\}_{s\in\R}$ the maximal $[\psi]$-relative psh test curve associated to $\{u_t\}_{t\geq 0}$ (Proposition \ref{prop:Legendre}). Then, as observed during the proof of by Proposition \ref{prop:Necessaria}, if $\lim_{t\to+\infty}\frac{E_\psi(u_t)}{t}=s_\psi^+(v_s)$ then $v_s=\psi$ for any $s\leq s_\psi^+(v_s)$ and $u_t=\psi+t s_\psi^+(v_s)$. In this case (\ref{eqn:lambdaDing}) easily follows.\newline
If instead $s_0:=\lim_{t\to+\infty}\frac{E_\psi(u_t)}{t}<s_\psi^+(v_s)$, then by Proposition \ref{prop:Necessaria}
$$
\frac{\tilde{\delta}_\bD}{\lambda} \nu(\lambda v_{s_0},E)\leq \tilde{\delta}_\bD (S_\bD(E)+\ord_E\bD)\leq A_X(E)
$$
for any $E$ prime divisor over $X$. Thus, combining \cite[Theorem B.5]{BBJ15} with Proposition \ref{prop:L2} we deduce that
$$
\liminf_{t\to+\infty} \frac{L^\lambda(u_t)}{t}=\sup\Big\{s\in \R\, : \, \int_Xe^{-\lambda v_s}d\mu<+\infty\Big\}\geq s_0.
$$
Hence (\ref{eqn:lambdaDing}) follows by definition of $s_0$.
\end{proof}
Next, we define
$$
\alpha_\bD:=\inf_{E/X}\frac{A_X(E)-\ord_E(\bD)}{\tau_\bD(E)},
$$
which is the algebraic version of the $\psi$-relative $\alpha$-invariant introduced in \cite{Tru20c}. Indeed it is not hard to check that
$$
\alpha_\bD=\alpha_\psi:=\sup\Big\{\alpha>0\, : \, \sup_{\{u\in \PSH(X,\omega)\, :\, u\leq \psi,\, \sup_X u=0\}}\int_X e^{\alpha(\psi-u)}e^{-\psi}d\mu<+\infty\Big\}.
$$
Similarly we also define its modified version
$$
\tilde{\alpha}_\bD:=\inf_{E/X} \frac{A_X(E)}{\tau_\bD(E)+\ord_E(\bD)},
$$
which coincides with
\begin{equation}
    \label{eqn:Alpha}
    \tilde{\alpha}_\bD=\tilde{\alpha}_\psi:=\sup\Big\{\alpha>0\, : \, \sup_{\{u\in \PSH(X,\omega) \, :\, u\leq \psi,\, \sup_X u=0\}}\int_X e^{-\alpha u}d\mu<+\infty\Big\}
\end{equation}
(see again \cite{Tru20c}).\newline

The following result shows that $\tilde{\delta}_\bD>1$ implies a \emph{geodesic stability} notion in $\cE^1_\psi$.
\begin{theorem}
\label{thm:GeodStab}
Let $\bD\in\bDiv_L(X)$ such that $\tilde{\delta}_\bD>1$. Then, for any $[\psi]$-relative psh geodesic ray $U:\R_{\geq 0}\to \cE^1_\psi$,
$$
\lim_{t\to+\infty}\frac{D_\psi(u_t)}{t}\geq c^2\lim_{t\to +\infty} \frac{J_\psi(u_t)}{t}
$$
where $c^2:=\min(1,\tilde{\alpha}_\bD,\tilde{\delta}_\bD-1)$.
\end{theorem}
\begin{proof}
By Lemma \ref{lem:Delta&lct} $D_\psi^\lambda(u_t)$ is well defined for any $\lambda< \tilde{\delta}_\bD$. Thus, let $1<\lambda< \tilde{\delta}_\bD$ such that $\lambda-1\leq \min(1,\tilde{\alpha}_\bD)$.\newline
Fixed $u\in \cE^1_\psi$ with $\sup_X u=0$. Writing $u=\lambda(2-\lambda)u+(\lambda-1)^2u$, the convexity of $\log\int_X e^fd\mu$ yields
\begin{multline*}
    -\frac{1}{2}\log\int_X e^{-2u}d\mu\geq -\frac{2-\lambda}{2}\log \int_X e^{-2\lambda u}d\mu-\frac{\lambda-1}{2}\log\int_X e^{-2(\lambda-1)u}d\mu\geq\\
    \geq-\frac{2-\lambda}{2}\log \int_X e^{-2\lambda u}d\mu-C_\lambda
\end{multline*}
where $C_\lambda>0$ is a uniform constant given by the definition of $\tilde{\alpha}_\bD$. Thus we deduce
$$
D_\psi(u)\geq \big((2-\lambda)\lambda\big)D^{\lambda}_\psi(u)-(\lambda-1)^2E_\psi(u)-C_{c,\lambda}\geq\big((2-\lambda)\lambda\big)D^{\lambda}_\psi(u)+(\lambda-1)^2J_\psi(u)-C_\lambda
$$
where the last inequality follows from $J_\psi(u)\leq -E_\psi(u)$ if $\sup_X u\leq 0$. Hence, for any $\psi$-relative psh geodesic ray $U:\R_{\geq 0}\to \cE^1_\psi$, the translation invariance of $D_\psi,J_\psi$ leads to
$$
\lim_{t\to +\infty} \frac{D_\psi(u_t)}{t}\geq\big((2-\lambda)\lambda\big) \liminf_{t\to+\infty}\frac{D^\lambda(u_t)}{t}+(\lambda-1)^2\lim_{t\to+\infty}\frac{J_\psi(u_t)}{t}\geq(\lambda-1)^2\lim_{t\to+\infty}\frac{J_\psi(u_t)}{t}
$$
where we also used Proposition \ref{prop:438}. The arbitrariness of $\lambda$ concludes the proof.
\end{proof}
Together with Theorem \ref{thm:KEDing}, Theorem \ref{thm:UniformToDelta} and Proposition \ref{prop:deltacompa}, the following result concludes the the proof of Theorem \ref{thmD}.
\begin{theorem}
\label{thm:ConclusionKE}
Let $\bD\in \bDiv_L(X)$ such that $\tilde{\delta}_\bD>1$. Then there exists a unique $[\psi]$-KE metric.
\end{theorem}
\begin{proof}
By Theorem \ref{thm:Ding} it is sufficient to prove that there exists $A>0,B\geq 0$ such that
\begin{equation}
    \label{eqn:CoercivityG}
    D_\psi(u)\geq AJ_\psi(u)-B
\end{equation}
for any $u\in\cE^1_\psi$. Assume by contradiction there exists a sequence $\{u_k\}_{k\in\N}\subset \cE^1_\psi$ with $\sup_X u_k =0$ such that
$$
D_\psi(u_k)\leq \epsilon_k J_\psi(u_k)-C_k\leq \epsilon_k d_1(\psi,u_k)-C_k.
$$
where $\epsilon_k\searrow 0$ while $C_k\nearrow +\infty$. Observe that by weak compactness of $\{u\in \cE^1_\psi \, : \, E_\psi(u)>-C, \, \sup_X u=0\}\subset \cE^1_\psi$ \cite[Proposition 4.19]{DDNL17b} and by lower-semicontinuity of $D_\psi$ with respect to the weak topology \cite[Proposition 4.18]{Tru20c}, we necessarily have $J_\psi(u_k)\to+\infty$ as $k\to +\infty$.\newline
Let $\{u_{k,t}\}_{t\in[0,d_1(\psi,u_k)]}$ the unit speed $[\psi]$-relative geodesic psh path joining $\psi$ and $u_k$. By convexity of $D_\psi$ (Proposition \ref{prop:ConveDing}), we deduce that
$$
\frac{D_\psi(u_{k,t})}{t}\leq \epsilon_k
$$
for any $t\in [0,d_1(\psi,u_k)]$ and for any $k\in\N$. Then we want to proceed similarly to \cite[Theorem 5.2]{DZ22}. Namely,we observe that the functional
$$
D^\beta_\psi(u_{k,t}):=D_\psi(u_{k,t})+\beta E_\psi(u_{k,t})=L(u_{k,t})-(1-\beta)E_\psi(u_{k,t})
$$
is still lower semicontinuous with respect to the weak topology if $\beta\in(0,1)$ and that it is convex along $[\psi]$-relative psh geodesic paths in $\cE^1_\psi$ (as immediate consequence of Theorem \ref{thm:E1}$.(iii)$ and of Proposition \ref{prop:ConveDing}). Note that
$$
\frac{D_\psi^\beta(u_{k,t})}{t}\leq \epsilon_k-\beta
$$
as $E_\psi(u_{k,t})=-d_1(\psi,u_{k,t})=-t$. Indeed, combining \cite[Proposition 3.9.(iii)]{Tru20c} with \cite[Theorem 1]{Dar17}, it is not hard to check that $\sup_X u_{k,t}=\sup_X(u_{k,t}-\psi)=0$ for any $t\in [0,d_1(\psi,u_k)]$ and for any $k\in\N$.\newline
Next, setting $\D_k:=\{\tau\in\D \, : \, -\log\lvert \tau\rvert\in [0,d_1(\psi,u_k)] \}$, a diagonal argument shows that $v_k(x,\tau):=u_{k,-\log\lvert \tau\rvert}(x)\in X\times \D_k$ converges to a $S^1$-invariant $\pi_X^*\omega$-psh function $v\in X\times\D^*$. In particular there exists a set $Z\subset (0,+\infty)$ of zero Lebesgue measure such that $u_t:=u_{-\log\lvert \tau\rvert}=v(\cdot,\tau)\in \PSH(X,\omega)$ is the weak limit of $u_{k,t}$ as $k\to +\infty$ for any $t\in(0,+\infty)\setminus Z$. However, Hartogs' Lemma \cite[Proposition 8.5]{GZ17} and the $t$-convexity of $\{u_t\}_{t\geq 0}$ show that $Z=\emptyset$ and that $\sup_X(u_t-\psi)=\sup_X u_t=0$ for any $t\in X$. Furthermore the upper semicontinuity of the $[\psi]$-relative Monge-Ampère energy \cite[Proposition 4.19]{DDNL17b} implies that $0\leq E_\psi(u_t)\leq -t$, while $d_1(u_t,\psi)\to 0$ as $t\to 0^+$ leads to $u_t\to \psi$ weakly \cite[Proposition 3.6]{Tru20c}. Therefore, exploiting again the lower-semicontinuity of $D_\psi^\beta$, $\{u_t\}_{t\geq 0}$ is a $[\psi]$-relative psh ray such that
$$
\frac{D^\beta_\psi(u_t)}{t}\leq -\beta.
$$
Then, considering the $[\psi]$-relative psh geodesic ray $\{u_t^M\}_{t\geq 0}$ given by the maximization of $\{u_t\}_{t\geq 0}$ (Proposition \ref{prop:Maximi}), we have
$$
\lim_{t\to+\infty} \frac{D^\beta_\psi(u_t^M)}{t}=\lim_{t\to+\infty} \frac{D^\beta_\psi(u_t)}{t}\leq -\beta.
$$
In particular, $\lim_{t\to +\infty}\frac{E_\psi(u_t^M)}{t}<0$ since otherwise $u_t^M=\psi$ (see the proof of Proposition \ref{prop:Necessaria}) which would lead to $D^\beta_\psi(u_t^M)$ constant. Thus, letting $\{\tilde{u}_t\}_{t\geq 0}$ to be the unit speed reparametrization of $\{u_t^M\}_{t\geq 0}$, we deduce that $\lim_{t\to+\infty}\frac{D_\psi^\beta(\tilde{u}_t)}{t}\leq 0$, i.e. that
\begin{equation}
    \label{eqn:Contr1}
    \lim_{t\to +\infty} \frac{D_\psi(\tilde{u}_t)}{t}\leq \beta.
\end{equation}
On the other hand, as immediate consequence of Theorem \ref{thm:GeodStab} we have that $D_\psi$ is uniformly geodesically coercive, i.e. there exists $\lambda>0$ such that
\begin{equation}
    \label{eqn:Contr2}
    \lim_{t\to+\infty} \frac{D_\psi(\tilde{u}_t)}{t}\geq \lambda \lim_{t\to+\infty}\frac{J_\psi(\tilde{u}_t)}{t}.
\end{equation}
Thus, as $J_\psi(\tilde{u}_t)=-E_\psi(\tilde{u}_t)+O(1)=t+O(1)$, it is enough to choose $0<\beta<\lambda$ to get a contradiction from (\ref{eqn:Contr1}), (\ref{eqn:Contr2}).
\end{proof}

\section{Varying the singularities}
\label{sec:Last}
Keeping assuming $X$ Fano manifold and $L=-K_X$, we will study the modified delta invariant $\tilde{\delta}_{\bD}$ to deduce some interesting properties of the uniform $\bD$-log Ding stability varying the generalized $b$-divisor $\bD\in\bDiv_L(X)$.\newline

Consider $\bD_0,\bD_1\in\bDiv_L(X)$ such that $\bD_0\leq \bD_1$, and let $\psi_0,\psi_1\in\cM^+_D(X,\omega)$ the associated model type envelopes where as usual $\omega$ is a fixed reference K\"ahler form such that $\{\omega\}=c_1(X)$ (Proposition \ref{prop:Corre}). We can then join $\psi_0,\psi_1$ through the path $[0,1]\ni t \to \psi_t:=P_\omega[t\psi_1+(1-t)\psi_0](0)\in \cM^+(X,\omega)$. By construction $t\to \psi_t$ is decreasing. Moreover, by the exact same argument of \cite[Proposition 3.7]{Tru20c} it follows that $\psi_t\in \cM^+_D(X,\omega)$, that $t\to \psi_t$ is continuous with respect to the weak topology and that
$$
V_t:=\int_XMA_\omega(\psi_t)=\int_XMA_\omega(t\psi).
$$
We denote by $\bD_t\in \bDiv_L(X)$ the generalized $b$-divisor associated to $\psi_t$ through Proposition \ref{prop:Corre}.
\begin{theorem}
\label{thm:Main}
Let $\bD_0,\bD_1\in \bDiv_L(X)$ such that $\bD_0\leq \bD_1$, $\bD_0\neq \bD_1$. Let also $[0,1]\ni t\to \bD_t\in\bDiv_L(X)$ be the increasing path explained above. Set $\tilde{\delta}_t:=\tilde{\delta}_{\bD_t}$, $V_t:=\langle(L-\bD_t)^n \rangle$, $\nu_t^E:=\ord_{\bD_t}(E)$ and $\tau_t^E:=\tau_{\bD_t}(E)$. Then
\begin{itemize}
    \item[i)] letting $\tilde{\alpha}_s:=\tilde{\alpha}_{\psi_s}:=\inf_{E/X}\frac{A_X(E)}{\tau^E_s+\nu^E_s}$, for any $0\leq s\leq t\leq 1$
    $$
    \frac{V_s}{\tilde{\delta}_s}\geq \frac{V_t}{\tilde{\delta}_t}\geq \frac{V_s}{\tilde{\delta}_s}-\frac{V_s-V_t}{\tilde{\alpha}_s}
    $$
    \item[ii)]
    \begin{equation}
        \label{eqn:Main1}
        \frac{V_0}{V_1}\geq \frac{\tilde{\delta}_0}{\tilde{\delta}_1}\geq\Big(\frac{V_1}{V_0}\Big)^n;
    \end{equation}
    \item[iii)] letting $c:=\inf_{E/X}\Big\{\frac{\tau_0^E-\tau_1^E+\nu_0^E}{\nu_1^E}\Big\}\in \R_{\geq 1}$ and $\lct_1:=\lct_X(\bD_1)$,
    \begin{equation}
        \label{eqn:Main2}
        \tilde{\delta}_1^{-1}\leq \frac{n}{n+1}\tilde{\alpha}_0^{-1}+\max\Big(0,1-\frac{cn}{n+1}\Big)\lct_1^{-1}.
    \end{equation}
\end{itemize}
\end{theorem}
Observe also that the invariant $c$ in $(iii)$ is bigger than $1$ as
$$
\tau_0^E+\nu_0^E\geq \tau_1^E+\nu_1^E
$$
for any $E$ prime divisor over $X$ by definition of $\tau_i^E,\nu_i^E$.\newline
See (\ref{eqn:Alpha}) for the analytic version of the \emph{modified $\alpha$-invariant} $\tilde{\alpha}_\psi$.\newline

The proof of Theorem \ref{thm:Main} will be given in subsection \ref{ssec:Proof}.
\subsection{Consequences of Theorem \ref{thm:Main}}
We collect here some corollaries of Theorem \ref{thm:Main}, and in particular we prove Theorem \ref{thmA} and Theorem \ref{thmB} of the Introduction.\newline

\begin{defi}
Let $L:=-K_X$. The \emph{$\bD$-log $K$-stability locus} $\bDiv^{K}(X)$ is defined as
$$
\bDiv^{K}(X):=\Big\{\bD\in \bDiv_L(X)\, : \, \lct_X(\bD)>1 \, \mathrm{and}\, (X,L)\, \bD\mathrm{-log}\, K\mathrm{-stable} \Big\}.
$$
The \emph{uniform $\bD$-log $K$-stability locus} $\bDiv^{UK}(X)$ is given as
$$
\bDiv^{UK}(X):=\Big\{\bD\in \bDiv_L(X)\, : \, \lct_X(\bD)>1 \, \mathrm{and}\, (X,L)\, \mathrm{uniformly}\, \bD\mathrm{-log}\, K\mathrm{-stable} \Big\}.
$$
The $\bD$\emph{-log Ding stability locus} $\bDiv^D(X)$ and the \emph{uniform $\bD$-log Ding stability locus} $\bDiv^{UD}(X)$ are defined similarly. 
\end{defi}
We have the following inclusions
$$
\begin{tikzcd}
    \bDiv^{UD}(X) \arrow[d, phantom, sloped, "\subset"] \arrow[r, phantom, sloped, "\subset"]& \bDiv^D(X) \arrow[d, phantom, sloped, "\subset"]\\
    \bDiv^{UK}(X) \arrow[r, phantom, sloped, "\subset"] & \bDiv^K(X)
\end{tikzcd}
$$
as immediate consequence of Proposition \ref{prop:DingK}. Moreover by Theorem \ref{thmD} it follows that
$$
\mathcal{M}_{KE}^{U}:=\big\{\psi\in \cM^+(X,\omega)\, : \, \exists\, !\, [\psi]\mathrm{-KE}\, \mathrm{metric}\big\}=\bDiv^{UD}(X),
$$
(see also \cite{Tru20c}), which implies that there exists a unique genuine KE metric if and only if $0\in \bDiv^{UK}(X)$ if and only if $0\in \bDiv^{UD}(X)$ by \cite{BBJ15, Fuj19b}.
\subsubsection{Openness of $\bDiv^{UD}(X)$}
The set $\bDiv_L(X)$ inherits a natural \emph{strong topology} given by Proposition \ref{prop:Corre}. Indeed, as showed in \cite{DDNL19}, the set $\cM^+(X,\omega)$ is endowed of a distance $d_S$ induced by the radial $L^1$-distance on the space of psh geodesic rays. We refer to \cite{DDNL19} for the precise definition of $d_S$. Here we recall the following properties.
\begin{prop}
\label{prop:d_S}
Let $\{\psi_k\}_{k\in\N}\subset \cM^+(X,\omega)$ be a sequence $d_S$-convergent to $\psi\in \cM^+(X,\omega)$. Set also $\tilde{\psi}_k:=P_\omega\big[\max(\psi_k,\psi)\big](0)\in\cM^+(X,\omega)$. Then
\begin{itemize}
    \item[i)] $\psi_k\to \psi$ and $\int_X MA_\omega(\psi_k)\to \int_X MA_\omega(\psi)$ as $k\to+\infty$;
    \item[ii)] $d_S(\tilde{\psi}_k,\psi)\to 0$ as $k\to+\infty$.
\end{itemize}
\end{prop}
\begin{proof}
If $\psi_k\overset{d_S}{\longrightarrow}\psi$ then $\int_X MA_\omega(\psi_k)\to\int_X MA_\omega(\psi)>0$ as immediate consequence of \cite[Lemma 3.7]{DDNL19}. Thus, \cite[Theorem 5.6]{DDNL19} leads to $\psi_k\to \psi$ and $(i)$ is proved.\newline
Then, $(ii)$ immediately follows from \cite[Proposition 3.5]{DDNL19}.
\end{proof}
As said before we call \emph{strong topology} the topology induced by $d_S$ on $\bDiv_L(X)$ thanks to Proposition \ref{prop:Corre}. The following result gives the first part of Theorem \ref{thmB}.
\begin{theorem}
\label{thm:SO}
The function $\tilde{\delta}:\bDiv_L(X)\to \R_{\geq 0}$ is continuous with respect to the strong topology induced by $d_S$ on $\bDiv_L(X)$. In particular the set $\bDiv^{UD}(X)$ is strongly open.
\end{theorem}
\begin{proof}
Let $\bD_k\in \bDiv_L(X)$ strongly converging to $\bD\in \bDiv_L(X)$, i.e. passing to the associated model type envelopes $\psi_k\overset{d_S}{\longrightarrow} \psi$ where $\psi_k,\psi\in\cM^+_D(X,\omega)$. Set also $\tilde{\psi}_k:=P_\omega[\max(\psi_k,\psi)](0)\in \cM^+(X,\omega)$. It is not hard to check that $\tilde{\psi}_k\in\cM^+_D(X,\omega)$, i.e. it is associated to an element $\tilde{\bD}_k$.\newline 
Denoting by $\tilde{\delta}_k:=\tilde{\delta}_{\bD_k}, \tilde{\delta}_{\max,k}:=\tilde{\delta}_{\tilde{\bD}_k}$ and $\tilde{\delta}:=\tilde{\delta}_\bD$, we then have
\begin{equation}
    \label{eqn:d_S}
    \lVert \log \tilde{\delta}_k-\log\tilde{\delta} \rVert\leq \lVert \log \tilde{\delta}_k-\log\tilde{\delta}_{\max,k} \rVert+\lVert \log \tilde{\delta}_{\max,k}-\log\tilde{\delta} \rVert\leq n \log \Big(\frac{\big(\int_XMA_\omega(\tilde{\psi}_k)\big)^2}{\big(\int_X MA_\omega(\psi_k)\big)\big(\int_X MA_\omega(\psi)\big)}\Big)
\end{equation}
where the last inequality follows from Theorem \ref{thm:Main}$.(ii)$. Then Proposition \ref{prop:d_S} implies that the right hand side in (\ref{eqn:d_S}) converges to $0$ as $k\to +\infty$. We have proved continuity of $\tilde{\delta}:\bDiv_L(X)\to \R_{\geq 0}$.\newline
Next, the strongly openness of $\bDiv^{UD}(X)$ follows from the equivalence $(ii)\Longleftrightarrow (iv)$ in Theorem \ref{thmD}.
\end{proof}

\subsubsection{Other consequences of Theorem \ref{thm:Main}}
Assuming $\bD_0=0$, we have the following interesting result, which in particular concludes the proofs of Theorems \ref{thmA} and \ref{thmB}.
\begin{prop}
\label{prop:QQQQQ}
With the same notations of Theorem \ref{thm:Main}, the following statements holds.
\begin{itemize}
    \item[i)] If $0\in \bDiv^{UD}(X)$, i.e. $\delta_0>1$, then
    $$
    \Big\{\bD_1\in \bDiv_L(X)\, : \, V_1> V_0/ \delta_0 \Big\}\subset \bDiv^{UD}(X).
    $$
    In particular \begin{equation}
        \label{eqn:FFF}
        \delta_0\leq (L^n) \inf_{\bD\in\bDiv_L(X)} \frac{\tilde{\delta}_\bD}{\langle (L-\bD)^n \rangle} \leq (L^n) \inf_{\bD\in\bDiv_L(X)} \frac{\lct_X(\bD)}{\langle (L-\bD)^n \rangle}.
    \end{equation}
    \item[ii)] If $\alpha_0> \frac{n}{n+1}$ then
    \begin{gather*}
        \Big\{\bD_1\in \bDiv_L(X)\, : \, c_{\bD_1}\geq(n+1)/n\Big\}\subset \bDiv^{UD}(X)\\
        \Big\{\bD_1\in \bDiv_L(X)\, : \, c_{\bD_1}<(n+1)/n, \, \lct_X(\bD_1)>1+\frac{n}{n+1}\Big(\frac{1-\alpha_0c_{\bD_1}}{\alpha_0-n/(n+1)}\Big)\Big\}\subset\bDiv^{UD}(X).
    \end{gather*}
    where $c_{\bD_1}:=\inf_{E/X}\Big\{\frac{\tau_0^E-\tau_1^E}{\nu_1^E}\Big\}$;
    \item[iii)] $0\in \bDiv^{UD}(X)$ if and only if
    $$
    \sup_{\bD\in \bDiv_L(X)}\tilde{\delta}_{\bD}^{\frac{1}{n}}\langle (L-\bD)^n\rangle>(L^n).
    $$
\end{itemize}
\end{prop}
\begin{proof}
Let $\bD_1\in\bDiv_L(X)$. By Theorem \ref{thm:Main}$.(ii)$ it follows that $\tilde{\delta}_1\geq \delta_0 V_1/V_0$ (by definition $\tilde{\delta}_0=\delta_0$). Thus, if $V_1> V_0/\delta_0$ then
$$
\lct_X(\bD_1)\geq \tilde{\delta}_1> 1
$$
where the first inequality is the content of Lemma \ref{lem:Delta&lct}. Hence, $\bD_1\subset \bDiv^{UD}(X)$ as a consequence of Theorem \ref{thmD}. Observe that the same calculation also leads to $\delta_0\leq \tilde{\delta}_{\bD_1} V_0/V_1 \leq \lct_X(\bD_1)V_0/V_1$, which clearly implies (\ref{eqn:FFF}).\newline
Then, by the other inequality in Theorem \ref{thm:Main}$.(ii)$ we have $ \delta_0\geq \sup_{\bD\in\bDiv_L(X)}\tilde{\delta}_\bD\Big( \frac{V_\bD}{V_0}\Big)^n $, which gives $(iii)$.\newline
Finally, the proof in $(ii)$ proceeds by an easy calculation using Theorem \ref{thm:Main}$.(iii)$ and it is left to the reader.
\end{proof}
\subsection{Proof of Theorem \ref{thm:Main}}
\label{ssec:Proof}
\begin{lemma}
\label{lem:Concav}
Let $t\to \bD_t$ as in Theorem \ref{thm:Main} and let $E$ be a prime divisor over $X$. Then the function $[0,1]\ni t\to g_E(t):=\tau_t^E+\nu_t^E$ is decreasing, continuous and concave.
\end{lemma}
\begin{proof}
We observe that by definition
$$
g_E(t)=\sup\Big\{\nu(u,E)\, : \, u\leq \psi_t, \int_X MA_\omega(u)>0\Big\}=\sup\big\{\nu(u,E)\, : \, u\leq \psi_t \big\},
$$
which clearly implies the monotonicity of $t\to g_E(t)$. Moreover, letting $u_1\leq \psi_{t_1}$, $u_2\leq \psi_{t_2}$, $a\in [0,1]$ and setting $t:=at_1+(1-a)t_2$, we have
\begin{multline}
    au_1+(1-a)u_2\leq a\psi_{t_1}+(1-a)\psi_{t_2}=\\
    =aP_\omega[t_1\psi_1+(1-t_1)\psi_0](0)+(1-a)P_\omega[t_2\psi_1+(1-t_2)\psi_0](0)=\\
    =P_{a\omega}[at_1\psi_1+a(1-t_1)\psi_0](0)+P_{(1-a)\omega}[(1-a)t_2\psi_1+(1-a)(1-t_2)\psi_0](0)\leq\\
    \leq P_\omega[t\psi_1+(1-t)\psi_0](0)=P_\omega[\psi_t](0)=\psi_t,
\end{multline}
where we used easy properties of $P_\omega[\cdot](0)$ coming directly from its definition. Thus,
$a\nu(u_1,E)+(1-a)\nu(u_2,E)\leq g_E(t)$, which leads to the concavity of $t\to g_E(t)$ by taking the supremum in $u_1\leq \psi_{t_1}$ and in $u_1\leq \psi_{t_2}$.\newline
It remains to prove the continuity of $g_E(t)$ at $t=0$ and at $t=1$. Assume first that $t_k\nearrow 1$, and let $u_k\leq \psi_{t_k}$ such that $\nu(u_k,E)\geq g_E(t_k)-\frac{1}{k}$ and such that $\sup_X u_k=0$ for any $k\in\N$. Then, by compactness of $\{u\in \PSH(X,\omega) \, :\, \sup_X u=0 \}$ \cite[Proposition 8.5]{GZ17}, we can also suppose that $u_k$ converges weakly to $u\in \PSH(X,\omega)$. As by $\psi_{t_k}\searrow \psi_1$, we have $u\leq \psi_1$ and $\nu(u,E)\leq g_E(1)$. Hence, by the upper semicontinuity of Lelong numbers we deduce
$$
g_E(1)\geq \nu(u,E)\geq \limsup_{k\to +\infty}\nu(u_k,E)\geq \limsup_{k\to +\infty}g_E(t_k).
$$
On the other hand, by monotonicity, $g_E(t_k)\geq g_E(1)$ for any $k\in \N$. Therefore $g_E(t_k)\searrow g_E(1)$.\newline
Next, fix $\epsilon>0$ and let $u\leq \psi_0$ such that $g_E(0)\leq\nu(u,E)+\epsilon$. Then for any $t>0$ the $\omega$-psh function $(1-t)u+t\psi_1$ is more singular than $t\psi_1+(1-t)\psi_0$ and hence of $\psi_t$. Thus,
$$
g_E(0)\leq \nu(u,E)+\epsilon\leq\frac{1}{1-t}\nu\big((1-t)u+t\psi_1,E\big)+\epsilon\leq\frac{1}{1-t}g_E(t)+\epsilon,
$$
which gives the continuity at $t=0$ and concludes the proof.
\end{proof}
\begin{prop}
\label{prop:Magic}
Let $t\to \bD_t$ as in Theorem \ref{thm:Main} and let $E$ be a prime divisor over $X$. Keeping using the notations of Theorem \ref{thm:Main}, set 
$F_E(t):=V_t(S_{\bD_t}(E)+\nu_t^E)$. Then $[0,1]\ni t \to F_E(t)$ is decreasing and for any convex combination $t=at_1+(1-a)t_2$
\begin{gather}
    \label{eqn:MM2}
    F_E(t)\geq a^{n+1}F_E(t_1)+(1-a)^{n+1}F_E(t_2)
\end{gather}
Moreover, for any $0\leq s\leq t\leq 1$
\begin{gather}
    \label{eqn:Magic}
    F_E(s)\geq F_E(t)\geq F_E(s)+\int_s^t (\tau_r^E+\nu^E_r)\frac{d V_u}{du}_{|u=r} dr.
\end{gather}
\end{prop}
\begin{proof}
To lighten notations, we will drop the dependence on $E$. Set $g(t):=\tau_t+\nu_t$.\newline
As seen in Lemma \ref{lem:Z1}
\begin{equation}
    \label{eqn:Formu}
    F_E(t)=\int_0^{g(t)}\Big(\int_X MA_{\omega}(v_r^t)\Big)dr
\end{equation}
where $v_r^t:=P_\omega[\psi_t](v_r^E)$ (see (\ref{eqn:VE2})). By Lemma \ref{lem:Concav} $g(t)\geq ag(t_1)+(1-a)g(t_2)=:h(t)$ if $t=at_1+(1-a)t_2$, $a\in[0,1]$.  Moreover, by easy properties of $P_\omega[\cdot](\cdot)$ we also have that 
\begin{multline}
    \label{eqn:Nata1}
    a v_{rg(t_1)/h(t)}^{t_1}+(1-a)v_{rg(t_2)/h(t)}^{t_2}=aP_\omega[\psi_{t_1}](v_{rg(t_1)/h(t)}^E)+(1-a)P_\omega[\psi_{t_2}](v_{rg(t_2)/h(t)}^E)=\\
    =P_{a\omega}[at_1\psi_1+a(1-t_1)\psi_0](av_{rg(t_1)/h(t)}^E)+P_{(1-a)\omega}[(1-a)t_2\psi_1+(1-a)(1-t_2)\psi_0]\big((1-a)v_{rg(t_2)/h(t)}^E\big)\leq\\
    \leq P_\omega[\psi_t]\big(av_{rg(t_1)/h(t)}^E+(1-a)v_{rg(t_2)/h(t)}^E\big).
\end{multline}
Then, by definition (\ref{eqn:VE}) we have
\begin{equation}
    \label{eqn:Nata2}
    a v_{rg(t_1)/h(t)}^E+(1-a)v_{rg(t_2)/h(t)}^E\leq v_{r(ag(t_1)+(1-a)g(t_2))/h(t)}^E=v_r^E.
\end{equation}
Thus, combining (\ref{eqn:Nata1}) and (\ref{eqn:Nata2}), by \cite[Theorem 1.2]{WN17} we get
\begin{multline*}
    f(r):=\int_X MA_\omega (v_r^t)\geq \int_X MA_\omega\big(a v_{rg(t_1)/h(t)}^{t_1}+(1-a)v_{rg(t_2)/h(t)}^{t_2}\big)\geq\\
    \geq a^n\int_X MA_\omega(v_{rg(t_1)/h(t)}^{t_1})+(1-a)^{n}\int_X MA_\omega(v_{rg(t_2)/h(t)}^{t_2})=:a^n f_1\big(rg(t_2)/h(t)\big)+(1-a)^n f_2\big(rg(t_1)/h(t)\big)
\end{multline*}
where in the last equality we clearly defined $f_i(r):=\int_X MA_\omega(v_r^{t_i})$. Hence, by an easy calculation,
\begin{multline*}
    a^{n+1}F_E(t_1)+(1-a)^{n+1}F_E(t_2)=a^{n+1}\int_0^{g(t_1)}f_1(u)du+(1-a)^{n+1}\int_0^{g(t_2)}f_2(u)du=\\
    =a^n\int_0^{h(t)}f_1\big(rg(t_1)/h(t)\big)\frac{ag(t_1)}{h(t)}dr +(1-a)^n\int_0^{h(t)}f_2\big(rg(t_2)/h(t)\big)\frac{(1-a)g(t_2)}{h(t)}dr\leq\\
    \leq a^n\int_0^{h(t)}f_1\big(rg(t_1)/h(t)\big)dr +(1-a)^n\int_0^{h(t)}f_2\big(rg(t_1)/h(t)\big)dr\leq\\
    \leq\int_0^{h(t)}f(r)dr\leq\int_0^{g(t)}f(r)dr=F_E(t),
\end{multline*}
i.e. (\ref{eqn:MM2}) is proved.\newline
We now want to prove (\ref{eqn:Magic}). The monotonicity of $t\to F_E(t)$ follows again from the formula (\ref{eqn:Formu}) and \cite[Theorem 1.2]{WN17} as $t\to v_r^t, t\to g(t)$ are decreasing.  
Moreover, from (\ref{eqn:MM2}) and the monotonicity of $t\to F_E(t)$ it is not hard to deduce that $t\to F_E(t)$ is absolute continuous. In particular, it is differentiable almost everywhere and
\begin{equation}
    \label{eqn:TFCI}
    F_E(t)=F_E(s)+\int_s^t \frac{dF_E}{du}_{|u=r}dr.
\end{equation}
for any $0\leq s\leq t\leq 1$.

Then we observe that for any $0\leq s\leq t\leq 1$ fixed and any $r\leq g(t)$,
\begin{multline*}
    \int_X MA_\omega(v_r^t)+\int_X MA_\omega(\psi_s)\geq \int_X MA_\omega\big(P_\omega(v_r^s,\psi_t)\big)+\int_X MA_\omega\big(\max(v_r^s,\psi_t)\big)\geq\\
    \geq\int_X MA_\omega(v_r^s)+\int_X MA_\omega(\psi_t)
\end{multline*}
where the first inequality is given by the montonicity of the non-pluripolar product while the last inequality is the content of \cite[Theorem 1.2]{DDNL19}. Thus
\begin{multline}
    \label{eqn:Label0}
    0\geq F_E(t)-F_E(s)=\\
    =\int_0^{g(t)}\Big(\int_X MA_\omega(v_r^t)-\int_X MA_\omega(v_r^s)\Big)dr-\int_{g(t)}^{g(s)}\Big(\int_X MA_\omega(v_r^s) \Big)dr\geq \\
    \geq \int_0^{g(t)}\Big(\int_X MA_\omega(\psi_t)-\int_X MA_\omega(\psi_s)\Big)dr-\int_{g(t)}^{g(s)}\Big(\int_X MA_\omega(v_r^s) \Big)dr=\\
    =-\int_{g(t)}^{g(s)}\Big(\int_X MA_\omega(v_r^s) \Big)dr+g(t)(V_t-V_s)\geq -\int_{g(t)}^{g(s)}\Big(\int_X MA_\omega(v_r^s) \Big)dr +g(t)(V_t-V_s)
\end{multline}
for any $0\leq s\leq t\leq 1$. On the other hand, assuming $0\leq s<1$, by Lemma \ref{lem:Concav} we have
\begin{equation}
    \label{eqn:Label}
    \int_{g(t)}^{g(s)}\Big(\int_X MA_\omega(v_r^s)\Big)dr\leq \big(g(s)-g(t)\big)\int_X MA_\omega(v_{g(t)}^s)\leq (t-s)\frac{g(s)-g(1)}{1-s}\int_X MA_\omega(v_{g(t)}^s).
\end{equation}
Hence, combining (\ref{eqn:Label0}) and (\ref{eqn:Label}) we obtain
$$
\frac{dF_E}{du}_{|u=r}\geq g(r)\frac{dV_u}{du}_{|u=r},
$$
for any $r\in [0,1)$ where $F_E$ is differentiable, as $\int_X MA_\omega(v^s_{g(t)})\to 0$ if $t\searrow s$. The latter, together with (\ref{eqn:TFCI}), concludes the proof.
\end{proof}
\begin{remark}
Note that (\ref{eqn:Label0}) represents an analog of \cite[Proposition 4.3]{Tru19} for psh geodesic rays thanks to Lemma \ref{lem:Z1}. Indeed, letting $\bD,\bD'\in \bDiv_L(X)$ such that $\bD'\leq \bD$, (\ref{eqn:Label0}) leads to
\begin{multline*}
    \lim_{t\to+\infty} \frac{VE_\psi(u_t-At)}{t}=V\big(S_{\bD}(E)+\ord_{E}(\bD)-\tau_{\bD'}(E)-\ord_E(\bD')\big)\geq\\
    \geq V'\big(S_{\bD'}(E)-\tau_{\bD'}(E)\big)=\lim_{t\to +\infty}\frac{V'E_{\psi'}(u'_t-At)}{t},
\end{multline*}
where $\psi,\psi'\in\mathcal{M}_D^+(X,\omega)$ are the model type envelopes associated to $\bD,\bD'$, $\{u_t\}_{t>0}, \{u_t'\}_{t>0}$ are the psh geodesic ray associated to the psh test curves given by (\ref{eqn:VE2}), $V':=\langle (L-\bD')^n \rangle, V:=\langle(L-\bD)^n  \rangle$ and $A:=\tau_{\bD'}(E)+\ord_E(\bD')$.
\end{remark}
\begin{proof}[Proof of Theorem \ref{thm:Main}]
\textbf{Step 1: proof of (i)}.\newline
Letting $E$ be a prime divisor over $X$ and let $0\leq s\leq t\leq 1$. Set also $S_t^E:=S_{\bD_t}(E)$ and similarly as in the statement of Theorem \ref{thm:Main} for the other quantities. As immediate consequence of Proposition \ref{prop:Magic} we obtain that
$$
\frac{V_s}{\tilde{\delta}_s}=\sup_{E/X}\frac{V_s(S_s^E+\nu_s^E)}{A_X(E)}\geq \sup_{E/X} \frac{V_t(S^E_t+\nu^E_t)}{A_X(E)}=\frac{V_t}{\tilde{\delta}_t},
$$
i.e. the map $[0,1]\ni t \to \frac{V_t}{\tilde{\delta}_t}$ is monotone. On the other hand, Proposition \ref{prop:Magic} also leads to
\begin{equation}
    \label{eqn:Utile}
    \frac{V_t(S_t^E+\nu_t^E)}{A_X(E)}\geq \frac{V_s(S_s^E+\nu_s^E)}{A_X(E)}+\frac{\int_s^t(\tau_r^E+\nu_r^E)\frac{d V_u}{du}_{|u=r}dr}{A_X(E)}\geq \frac{V_s(S_s^E+\nu_s^E)}{A_X(E)}+\frac{\tau_s^E+\nu_s^E}{A_X(E)}(V_t-V_s)
\end{equation}
where the last inequality follows from the monotonicity of $t\to \tau_r^E+\nu_t^E$ (Lemma \ref{lem:Concav}). Moreover, $V_t\leq V_s$ and by definition of $\tilde{\alpha}_s$ we have $\tau_s^E+\nu_s^E\leq A_X(E)/\tilde{\alpha}_s$ for any $E$ prime divisor over $X$. Thus, we also obtain
\begin{equation}
    \frac{V_t(S_t^E+\nu_t^E)}{A_X(E)}\geq \frac{V_s(S_s^E+\nu_s^E)}{A_X(E)}-\frac{V_s-V_t}{\tilde{\alpha}_s}.
\end{equation}
Taking the supremum over all $E/X$ concludes the proof of $(i)$.\newline
\textbf{Step 2: proof of (ii)}.\newline
The left-hand side inequality has already been proved in $(i)$. For the right-hand side, letting $E$ be a fixed prime divisor over $X$, by Proposition \ref{prop:OtherIneq} we have $\tau_s^E+\nu_s^E\leq (n+1)(S_s^E+\nu_s^E)$. Thus, the inequality (\ref{eqn:Utile}) yields
\begin{equation}
    \label{eqn:Nice111}
    S_t^E+\nu_t^E\geq (S_s^E+\nu_s^E)\Big(n+1-n\frac{V_s}{V_t}\Big) 
\end{equation}
for any $0\leq s\leq t\leq 1$. Moreover, as $[0,1]\ni t\to V_t$ is continuously differentiable and $V_0\geq V_t\geq V_1>0$, the proof of Proposition \ref{prop:Magic} implies that $[0,1]\ni t\to F_E(t), G_E(t):=S_t^E+\nu_t^E$ are Lipschitz. Therefore for any $s\in(0,1)$ where $G_E(s)$ is differentiable, from (\ref{eqn:Nice111}) we get
\begin{equation}
    \label{eqn:Nice}
    \frac{dG_E(r)}{dr}_{|r=s}\geq n\frac{G_E(s)}{V_s}\frac{d V_r}{dr}_{|r=s}.
\end{equation}
Thus, as $G_E\geq c>0$ for a small constant $c>0$, $[0,1]\ni t\to \log G_E(t), [0,1]\ni t\to \log V_t$ are absolute continuous and
$$
\log \frac{G_E(1)}{G_E(0)}\geq n\log \frac{V_1}{V_0}
$$
follows by integrating (\ref{eqn:Nice}). Hence $G_E(1)\geq G_E(0)\big(V_1/V_0\big)^n$ and we easily get the right-hand inequality in (\ref{eqn:Main1}) by
$$
\frac{1}{\tilde{\delta}_1}=\sup_{E/X}\frac{G_E(1)}{A_X(E)}\geq \Big(\frac{V_1}{V_0}\Big)^n \sup_{E/X}\frac{G_E(0)}{A_X(E)}= \Big(\frac{V_1}{V_0}\Big)^n\frac{1}{\tilde{\delta}_0}.
$$
\textbf{Step 3: proof of (iii).}\newline
Let $c\in \R_{\geq 1}$ such that $\tau_1^E+c\nu_1^E\leq \tau_0^E+\nu_0^E$ for any prime divisor $E$ over $X$. By Proposition \ref{prop:Fujita} we have
$$
S_1^E+\nu_1^E\leq \frac{n}{n+1}\tau_1^E+\nu_1^E\leq \frac{n}{n+1}(\tau_0^E+\nu_0^E)+\big(1-\frac{cn}{n+1}\big)\nu_1^E\leq \frac{n}{n+1}(\tau_0^E+\nu_0^E)+\max\Big(0,1-\frac{cn}{n+1}\Big)\nu_1^E.
$$
Thus, as $\nu_1^E\leq \frac{A_X(E)}{\lct_X(\bD_1)}=:\frac{A_X(E)}{\lct}$ by definition of $\lct$, we obtain
$$
\frac{S_1^E+\nu_1^E}{A_X(E)}\leq \frac{n}{n+1}\frac{\tau_0^E+\nu_0^E}{A_X(E)}+\max\Big(0,1-\frac{cn}{n+1}\Big)\lct^{-1}.
$$
Taking the supremum over $E/X$, we get (\ref{eqn:Main2}) which concludes the proof.
\end{proof}

{\footnotesize
\bibliographystyle{acm}
\bibliography{main}
}
\end{document}